\documentclass[12pt,a4paper,twoside,reqno]{amsart}
\usepackage{amssymb,latexsym}
\usepackage{amsmath,amsthm}
\usepackage{breqn}
\usepackage{stmaryrd} 

\usepackage[margin=2.5cm]{geometry}

 \usepackage[unicode,colorlinks,plainpages=false,hyperindex=true,bookmarksnumbered=true,bookmarksopen=false,pdfpagelabels]{hyperref}

 \hypersetup{urlcolor=cyan,linkcolor=blue,citecolor=red,colorlinks=true} 
\usepackage{xcolor}
\usepackage{tikz} 
\usetikzlibrary{arrows,shapes}  

\usepackage{enumitem}

\makeatletter
\renewcommand*{\p@section}{\,}
\renewcommand*{\p@subsection}{\S\,}
\renewcommand*{\p@subsubsection}{\S\,}
\makeatother

\newtheorem{thm}{Theorem}[section]
\newtheorem{cor}[thm]{Corollary}
\newtheorem{lem}[thm]{Lemma}
\newtheorem{prop}[thm]{Proposition}
\newtheorem{conj}[thm]{Conjecture}
\theoremstyle{remark}
\newtheorem{defn}[thm]{Definition}

\newtheorem{rem}[thm]{Remark}

\numberwithin{equation}{section}

\newcommand{\CC}{\ensuremath{\mathbb{C}}}
\newcommand{\N}{\ensuremath{\mathbb{N}}}

\newcommand{\kk}{\ensuremath{\Bbbk}}

\newcommand{\op}{\operatorname{op}}
\newcommand{\ee}{\operatorname{e}}
\newcommand{\out}{\operatorname{out}}
\newcommand{\inn}{\operatorname{inn}}
\newcommand{\du}{\operatorname{d}}
\newcommand{\Dtwo}{\operatorname{D}^2}
\newcommand{\mc}{\operatorname{m}}
\newcommand{\tr}{\operatorname{tr}}
\newcommand{\diag}{\operatorname{diag}}

\newcommand{\Hom}{\operatorname{Hom}}

\newcommand{\Der}{\operatorname{Der}}
\newcommand{\Mat}{\operatorname{Mat}}
\newcommand{\Id}{\operatorname{Id}}
\newcommand{\Gl}{\operatorname{GL}}

\newcommand{\Rep}{\operatorname{Rep}}

\newcommand{\D}{\operatorname{\mathbb{D}er}}

\newcommand{\sgn}{\operatorname{sgn}}
\newcommand{\he}{\hat{e}}
\newcommand\dgal[1]{  \left\{\!\!\!\left\{#1\right\}\!\!\!\right\} }
\newcommand\br[1]{\{ #1 \}} 
 
\newcommand{\lr}[1]{
  \{\mkern-6mu\{#1\}\mkern-6mu\}}
  \newcommand{\PP}{\ensuremath{\mathrm{P}}}
  \newcommand{\PPn}{\ensuremath{\mathrm{P}_n}}
\newcommand{\qP}{\ensuremath{\mathrm{qP}}}
\newcommand{\lfus}[1]{\{\mkern-6mu\{#1\}\mkern-6mu\}_{\mathrm{fus}}}
  
\newcommand{\Ac}{\ensuremath{\mathcal{A}}}
\newcommand{\Bc}{\ensuremath{\mathcal{B}}}

\newcommand{\loc}{\operatorname{loc}}
\newcommand{\fus}{\operatorname{fus}}

\newcommand{\ssep}{\operatorname{sep}}
\newcommand{\Bloc}{\ensuremath{\mathcal{B}_{n}^{\loc}}}
\newcommand{\Afus}{\ensuremath{\mathcal{A}_{n}^{\fus}}}

\begin{document}

\title[Euler continuants in NC quasi-Poisson geometry]{Euler continuants in\\noncommutative quasi-Poisson geometry}

\author{Maxime Fairon}
 \address[Maxime Fairon]{School of Mathematics and Statistics\\ University of Glasgow, University Place\\ Glasgow G12 8QQ, UK}
 \email{Maxime.Fairon@glasgow.ac.uk}
 
 \address{\emph{(Current address)} Department of Mathematical Sciences, Schofield Building\\ Loughborough University, Epinal Way\\ Loughborough LE11 3TU, UK}
 \email{M.Fairon@lboro.ac.uk}

\author{David Fern\'andez}
\address[David Fern\'andez]{Fakult\"at f\"ur Mathematik\\ Universit\"at Bielefeld, Universit\"atsstr. 25
\\
33615 Bielefeld, Germany}
\email{dfernand@math.uni-bielefeld.de}

\address{\emph{(Current address)} DMATH. Universit\'e du Luxembourg\\ Maison du Nombre, 6\\ Avenue de la Fonte
\\
L-4364 Esch-sur-Alzette, Luxembourg}
\email{david.fernandez@uni.lu}

\date{}

\begin{abstract} 
It was established by Boalch that Euler continuants arise as Lie group valued moment maps for a class of wild character varieties described as moduli spaces of points on $\mathbb{P}^1$ by Sibuya. 
Furthermore, Boalch noticed that these varieties are multiplicative analogues of certain Nakajima quiver varieties originally introduced by Calabi, which are attached to the quiver $\Gamma_n$ on two  vertices and $n$ equioriented arrows. 
In this article, we go a step further by unveiling that the Sibuya varieties can be understood using noncommutative quasi-Poisson geometry modeled on the quiver $\Gamma_n$. 
We prove that the Poisson structure carried by these varieties is induced, via the Kontsevich--Rosenberg principle, by an explicit Hamiltonian double quasi-Poisson algebra defined at the level of the quiver $\Gamma_n$ such that its noncommutative multiplicative moment map is given in terms of Euler continuants.
This result generalises the Hamiltonian double quasi-Poisson algebra associated with the quiver $\Gamma_1$ by Van den Bergh. Moreover, using the method of fusion, we prove that the Hamiltonian double quasi-Poisson algebra attached to $\Gamma_n$ admits a factorisation in terms of $n$ copies of the algebra attached to $\Gamma_1$.
\end{abstract}

\addcontentsline{toc}{subsection}{Contents}  

\maketitle

 \setcounter{tocdepth}{2} 
 
\vspace{-0.8cm} 


\vspace{-0.5cm}

\tableofcontents


 \section{Introduction} \label{S:Intro}

\subsection{Euler continuants} 
Fix a string $\mathtt{S}=``x_1\dots x_k"$ of $k\geq 1$ indeterminates, which are not necessarily invertible or commuting. We define the \emph{$k$-th Euler continuant (polynomial)}, denoted $(x_1,\dots,x_k)$, by starting with the product $x_1\cdots x_k$ and then taking the sum of 
all the distinct substrings of $\mathtt{S}$ obtained by removing adjacent pairs $x_\ell x_{\ell+1}$ in all possible ways, with $+1$ assigned to  the empty substring. The first instances of this family were already written in 1764 by Euler \cite[p. 55]{Eu} as 
\begin{align*}
(x_1)&=x_1\,,
\\
(x_1,x_2)&=x_1x_2+1\,,
\\
(x_1,x_2,x_3)&=x_1x_2x_3+x_3+x_1\,,
\\
(x_1,x_2,x_3,x_4)&=x_1x_2x_3x_4+x_1x_2+x_1x_4+x_3x_4+1\,,
\\
(x_1,x_2,x_3,x_4,x_5)&=x_1x_2x_3x_4x_5+x_1x_2x_3+x_1x_2x_5+x_1x_4x_5+x_3x_4x_5+x_1+x_3+x_5\,.
\end{align*}
More succinctly, Euler continuants can be defined by the following recurrence:
\begin{equation}
(\emptyset)=1\,,\quad (x_1)=x_1\,,\quad (x_1,\dots,x_k)=(x_1,\dots,x_{k-1})x_k+(x_1,\dots,x_{k-2}) \quad \text{if } k\geq 2\,.
\label{eq:def-Euler-continuant-intro-1}
\end{equation} 

Originally, Euler introduced this class of polynomials to grasp the numerators and denominators of continued fractions (hence the name `continuants'). 
For example, 
\[
x_1+\frac{1}{x_2}=\frac{(x_1,x_2)}{(x_2)}\,, \quad 
x_1+\frac{1}{x_2+\frac{1}{x_3}}=\frac{(x_1,x_2,x_3)}{(x_2,x_3)}\,, \quad 
x_1+\frac{1}{x_2+\frac{1}{x_3+\frac{1}{x_4}}}=\frac{(x_1,x_2,x_3,x_4)}{(x_2,x_3,x_4)}\,.
\]
Since then,  Euler continuants have appeared in many unexpected areas of research and they have become essential in mathematics. 
In the context of number theory, Euler continuants are closely related to Euclid's algorithm, see e.g. \cite{concrete} for a more detailed account. The number of terms of the $k$-th Euler continuant is the $(k+1)$-th Fibonacci number, while the Catalan numbers count the factorisations of the continuant via triangulations of polygons--- see \cite{Paluba}, where a connection to Loday's free duplicial algebras is also stated.
In knot theory, the Conway polynomial of an oriented two-bridge link is related to Euler continuants \cite{KP}; in contact and symplectic geometry, the braid varieties associated with 2-stranded braids are smooth varieties whose defining equations are closely linked to Euler continuants \cite{CGGS}. 
Another appearance of Euler continuants occurs when we consider a matrix version of the Stern--Brocot tree \cite{concrete}, which is a classical construction of all non-negative fractions whose numerators and denominators are coprime.  
If we let
\[
L= \left(\begin{array}{cc} 1&1\\0&1 \end{array} \right) ,\quad R= \left(\begin{array}{cc} 1&0\\1&1 \end{array} \right),
\]
each node in the matrix Stern--Brocot tree can be represented as a sequence 
\[
R^{c_0}L^{c_1}R^{c_2}L^{c_3}\cdots R^{c_{n-2}}L^{c_{n-1}}\,, \quad \text{ where } c_0,c_{n-1}\geq 0\,,\quad c_2,\dots,c_{n-2}\geq 1\,,
\]
and $n$ is even. Then, it is not difficult to prove that the four entries of the obtained matrix can be rewritten in terms of Euler continuants. 
To show this, it suffices to note that 
\begin{equation} \label{Eq:matrix-perm}
\left(\begin{array}{cc} 1&\alpha \\0&1 \end{array} \right) =  
\left(\begin{array}{cc} 0&1\\1&0 \end{array} \right)   \left(\begin{array}{cc} 0&1\\1&\alpha \end{array} \right) \,, \quad 
\left(\begin{array}{cc} 1&0\\\alpha&1 \end{array} \right) =  
\left(\begin{array}{cc} 0&1\\1&\alpha \end{array} \right) \left(\begin{array}{cc} 0&1\\1&0 \end{array} \right) \,, 
\end{equation}
for any constant $\alpha$, and then use that given $k\geq2$ and constants $x_i$ for $1\leq i \leq k$, we have 
\begin{equation} \label{eq:matrix-Euler-continuant}
 \left(\begin{array}{cc} 0&1\\1&x_k \end{array} \right) 
  \left(\begin{array}{cc} 0&1\\1&x_{k-1} \end{array} \right) \cdots 
\left(\begin{array}{cc} 0&1\\1&x_1 \end{array} \right) 
=
 \left(\begin{array}{cc} (x_{k-1},\ldots,x_2)&(x_{k-1},\ldots,x_1)\\(x_{k},\ldots,x_2)&(x_{k},\ldots,x_1) \end{array} \right)\,.
\end{equation}
This remarkable matrix identity can be proven by induction and it will be crucial in \ref{sec:intro-Boalch}.

Finally, beyond the recurrence \eqref{eq:def-Euler-continuant-intro-1}, an effective way to define Euler continuants uses the determinant of a tridiagonal matrix. For such a matrix, the only nonzero entries are given by the indeterminates $x_i$ along the main diagonal, $-1$ on the first diagonal below it, and $+1$ on the first diagonal above it. 
This approach suggests how to introduce new families of continuants; for instance, by taking determinants of other tridiagonal matrices.
They have appeared in interesting contexts such as Coxeter's frieze patterns, Ptolemy--Pl\"ucker relations and cluster algebras, or the discrete Sturm--Liouville, Hill or Schr\"odinger equations (see \cite{MO} and references therein).

Quite strikingly, Boalch \cite{B18} realised that Euler continuants naturally appear in the setting of certain wild character varieties linked to the work of Sibuya \cite{Si75}. In that case, the continuants can be understood as Lie group valued moment maps \cite{AMM98,AKSM02};  we will explain these links in \ref{sec:intro-Boalch}.
In this article, we are able to further prove that Euler continuants are noncommutative moment maps in the sense of Van den Bergh's noncommutative quasi-Poisson geometry--- see  \ref{sec:intro-results}.
This is another step towards the programme that we have initiated in \cite{FF} which aims at understanding the Poisson geometry of wild character varieties in terms of Hamiltonian double quasi-Poisson algebras attached to quivers.
Our programme is based on two main tools. First, the interpretation of such varieties using graphs/quivers undertaken by Boalch \cite{Bo12}--\cite{BoY}. Second, the application of the Kontsevich--Rosenberg principle \cite{KR00}, whereby one studies a noncommutative structure on an associative algebra whenever it induces the corresponding standard algebro-geometric structures on representation spaces.
Whereas in \cite{FF} we pursued this idea by varying the number of vertices of the quivers under consideration (typically, complete graphs with some additional data), in this article we fix the quivers to have two vertices and we vary the number of arrows between them.

\subsection{Euler continuants as moment maps} \label{sec:intro-Boalch}
Character varieties of Riemann surfaces have become central objects in modern mathematics. This is essentially due to their dual nature: they can be defined as moduli spaces of monodromy data of regular singular connections, or as spaces of representations of the fundamental group.
Partially motivated by works on two-dimensional gauge theory concerning the Atiyah--Bott symplectic form, Boalch started in \cite{Bo01,Bo01a} a groundbreaking programme to investigate the geometry of \emph{wild} character varieties (note this terminology appeared in \cite{Bo14}). 
These spaces generalise character varieties by considering moduli spaces of monodromy data classifying irregular meromorphic connections on bundles over Riemann surfaces; alternatively, wild character varieties parametrise fundamental group representations enriched by adding Stokes data at each singularity.
This led to the construction of many new holomorphic symplectic manifolds \cite{Bo-thesis, Bo01}, which turn out to admit hyperk\"ahler metrics, as shown in \cite{BB}. 
These developments prompted Boalch in \cite{Bo-Duke} to begin an extensive study of wild character varieties from an alternative algebraic perspective. 
Indeed, he realised that such varieties can be constructed as finite dimensional multiplicative symplectic quotients of smooth affine varieties, involving a distinguished holomorphic 2-form. This point of view required the introduction of a complex-analytic version of quasi-Hamiltonian geometry \cite{AMM98}, which builds on the key  operations of fission and fusion. 
Whereas the former consists in breaking the structure group into some relevant subgroups --- see \cite{Bo-Grenoble,Bo14}, the latter provides a way to glue pieces of surfaces together out of building blocks given by conjugacy classes, pairs of pants and the important fission spaces \cite{Bo14}.
In the rest of this subsection, we will come across one of these building blocks: the (reduced) fission space $\Bc^{n+1}$, see \eqref{eq:Boalch-space-1}.

In \cite{B18}, Boalch examined an interesting class of wild character varieties arising as specific moduli spaces of points on $\mathbb{P}^1$; they were originally studied by Sibuya \cite{Si75}, and for certain values they give rise to the prominent gravitational instantons. 
These wild character varieties are multiplicative analogues of a family of Hyperk\"ahler varieties introduced by Calabi \cite{Ca79} in 1979 as higher-dimensional examples of the Eguchi--Hanson spaces. 
 If we denote by 
$\Gamma_n$ the quiver with two vertices $\{1,2\}$ and $n$ arrows $a_i:1\to 2$ (whose double $\overline{\Gamma}_n$ has $n$ extra arrows $b_i:2\to 1$, see Figure \ref{fig:quiver-Gamma-n}),
Calabi's varieties can be described as (Nakajima/additive) quiver varieties attached to $\Gamma_n$. They are obtained by Hamiltonian reduction from the moduli space of representations of $\overline{\Gamma}_n$ with dimension vector $d=(1,1)$.
This prompted Boalch \cite{B18} to make the crucial observation that Sibuya's varieties, as multiplicative analogues of Calabi's varieties, could also be defined using the quiver $\Gamma_n$. This point of view naturally leads to Euler continuants as we explain now.

If $G:=\Gl_2(\CC)$, let $U_{+}$ (resp. $U_-$) be the subgroup of unipotent upper (resp. lower) triangular matrices, $T$ be the maximal diagonal torus of $G$, and we let $G_\circ=U_+TU_-$ denote the subspace of matrices in $G$ admitting a Gauss decomposition. Recall that, for a matrix  $M=(M_{ij})\in G$, we have that $M\in G_\circ$ is equivalent to the condition $M_{22}\neq 0$, in which case we can write that 
\begin{equation}\label{eq:matrix-Gauss-decompo}
M=
 \left(\begin{array}{cc} 1&M_{12}M_{22}^{-1}\\0&1 \end{array} \right) 
 \left(\begin{array}{cc} M_{11}-M_{12}M_{22}^{-1}M_{21}&0\\0&M_{22} \end{array} \right)
\left(\begin{array}{cc} 1&0\\M_{22}^{-1}M_{21}&1 \end{array} \right)  \,.
\end{equation}
Following Boalch \cite{B18} from now on (see also Paluba's thesis \cite{Paluba}, and even more generally \cite{Bo-Duke, Bo14}), we introduce for any $n\geq 1$ the reduced fission space $\Bc^{n+1}\subset (U_-\times U_+)^{n}$, which is the smooth complex variety defined by 
\begin{equation} \label{eq:Boalch-space-1}
\Bc^{n+1}:=\left\{S_{b,i}\in U_-,\,\,S_{a,i}\in U_+ \text{ for }1\leq i \leq n \mid S_{b,n}S_{a,n}\cdots S_{b,1}S_{a,1}\in G_\circ\right\}\,. 
\end{equation}
The condition of admitting a Gauss decomposition has a deep geometric meaning. Indeed, there is a natural action of $T$ on $\Bc^{n+1}$ by simultaneous conjugation, which preserves a $2$-form, and the inverse $\mu_T$ of the diagonal part  of the Gauss decomposition can be interpreted as a Lie group valued moment map, in the sense of \cite{AMM98}. 
Furthermore, since we consider the Gauss decomposition of a product of unipotent matrices, $\mu_T$ takes value in $T\cap \operatorname{SL}_2(\CC)$.
This implies that, if we fix the $T$-valued moment map $\mu_T$ to $t_\gamma:=\diag(\gamma,\gamma^{-1})\in T \cap \operatorname{SL}_2(\CC)$ for generic $\gamma\in \CC^\times$, the corresponding GIT quotient $\mathcal{M}_n(t_\gamma)=\mu_T^{-1}(t_\gamma)/\!/T$ admits a symplectic form. 
The space $\mathcal{M}_n(t_\gamma)$ hence obtained is an example of the Sibuya spaces mentioned earlier.

To relate the construction of $\mathcal{M}_n(t_\gamma)$ to quivers and Euler continuants, we let 
\begin{equation} \label{eq:Boalch-space-2}
 S_{b,i}=\left(\begin{array}{cc} 1&0\\B_i&1 \end{array} \right)\,, \quad 
 S_{a,i}= \left(\begin{array}{cc} 1&A_i\\0&1 \end{array} \right) \,, \quad 1\leq i \leq n\,,
\end{equation}
for constants $A_i,B_i\in \CC$ at a point of $\Bc^{n+1}$. Then, by \eqref{Eq:matrix-perm} and \eqref{eq:matrix-Euler-continuant}, we can see that 
\begin{equation} \label{eq:Boalch-space-3}
\begin{aligned}
 S_{b,n}S_{a,n}\cdots S_{b,1}S_{a,1}
=&
\left(\begin{array}{cc} 0&1\\1&B_n \end{array} \right) \left(\begin{array}{cc} 0&1\\1&A_n \end{array} \right) \ldots 
\left(\begin{array}{cc} 0&1\\1&B_1 \end{array} \right) \left(\begin{array}{cc} 0&1\\1&A_1 \end{array} \right)  \\
=& \left(\begin{array}{cc} (A_{n},\ldots,B_1)&(A_{n},\ldots,A_1)\\(B_{n},\ldots,B_1)&(B_{n},\ldots,A_1) \end{array} \right)\,.
\end{aligned}
\end{equation} 
Thus, the condition that this product belongs to $G_\circ$ is equivalent to requiring the invertibility of the Euler continuant $(B_{n},\ldots,A_1)$ placed in the (2,2)-entry. In view of \eqref{eq:matrix-Gauss-decompo}, we get that the diagonal part of the Gauss decomposition is $\diag(\tilde{s}_n,(B_{n},\ldots,A_1))$, where 
\begin{equation}  \label{eq:Boalch-space-4}
 \tilde{s}_n:=(A_{n},\ldots,B_1)-(A_{n},\ldots,A_1)(B_{n},\ldots,A_1)^{-1}(B_{n},\ldots,B_1)\,.
\end{equation}
By taking the determinant in \eqref{eq:Boalch-space-3} we can find that\footnote{ Here, one uses that the constants $(A_i,B_i)_{i=1}^n$ and the Euler continuants are valued in $\CC$ and hence are commuting. The same form for $\tilde{s}_n$ can be obtained for non-commuting variables $A_i,B_i$, though it requires a different proof.} $\tilde{s}_n=(A_1,B_1,\ldots,A_n,B_n)^{-1}$.
The upshot is that we can reexpress  $\Bc^{n+1}$ and its Lie group valued moment map $\mu_T:\Bc^{n+1}\to T$ in terms of the space 
\begin{equation} \label{eq:Boalch-space-5}
\widehat{\Bc}^{n+1}:=\left\{A_i,B_i\in \CC \text{ for }1\leq i \leq n \mid (A_1,\ldots,B_n)\neq 0\right\}\,, 
\end{equation}
and the map 
\begin{equation} \label{eq:Boalch-space-6}
\widehat{\mu}:\widehat{\Bc}^{n+1}\longrightarrow (\CC^\times)^2\,, \quad 
\{A_i,B_i\}\longmapsto \Big( (B_{n},\ldots,A_1)^{-1} , (A_1,\ldots,B_n)\Big)\,. 
\end{equation}
From this point of view, we can realise $\widehat{\Bc}^{n+1}$ as an affine subspace\footnote{The reader should be warned that Boalch \cite{B18} writes $\widehat{\Bc}^{n+1}=\Rep^\ast(\Gamma_n,(1,1))$ as a space of invertible representations inside $\Rep(\Gamma_n,(1,1))$. In that work, one considers $\Gamma_n$ as a \emph{graph} where each edge represents two arrows going in opposite directions, i.e. this is for us the double quiver $\overline{\Gamma}_n$, so that $\Rep(\Gamma_n,(1,1))$ in \cite{B18} corresponds to $\Rep(\overline{\Gamma}_n,(1,1))$ for us.} inside $\Rep(\overline{\Gamma}_n,(1,1))$, where we assign $A_i,B_i$ to the arrows $a_i,b_i$, respectively. 
Moreover, the subspace $\mu_T^{-1}(t_\gamma)$ needed to define the wild character variety $\mathcal{M}_n(t_\gamma)$ can be identified with 
$\widehat{\mu}^{-1}(\gamma,\gamma^{-1})$, which can be further seen as a moduli space of representations of a noncommutative algebra, the \emph{fission algebra} $\mathcal{F}^{q_\gamma}(\overline{\Gamma}_n)$; see \eqref{Eq:Fission-algebra} with parameter $q_\gamma=(\gamma,\gamma^{-1})$. 
Introduced in \cite{B15}, fission algebras generalise the (deformed) multiplicative preprojective algebras of \cite{CBShaw} and they are linked to the \emph{generalised double affine Hecke algebras} of Etingof, Oblomkov and Rains, as explained in \cite{EOR07}.

The importance of the discussion made so far is that we do not really need to consider the elements $\{A_i,B_i\}$ parametrising $\widehat{\Bc}^{n+1}$ as numbers, but, more generally, we can use matrices $A_i\in \Mat_{d_2\times d_1}(\CC)$, $B_i\in \Mat_{d_1\times d_2}(\CC)$ for any $d_1,d_2\geq 1$. 
Indeed, we get in that way other examples of reduced fission spaces  \cite{Bo14}, see \cite{Paluba}.
Hence, we can work with a space of invertible representations inside $\Rep(\overline{\Gamma}_n,(d_1,d_2))$, rather than $\Rep(\overline{\Gamma}_n,(1,1))$. Therefore, it seems natural to try to understand the structures unveiled by Boalch directly at the level of the quiver $\overline{\Gamma}_n$ by using noncommutative quasi-Poisson geometry \cite{VdB1}. In that case, we define from $\overline{\Gamma}_n$ the \emph{noncommutative} $(2n)$-th Euler continuants
\begin{equation*}
 (b_{n},\ldots,a_1)^{-1},\quad (a_1,\ldots,b_n)\,,
\end{equation*}
which shall play the role of noncommutative moment maps. This will be the aim of this article, whose main achievements will be presented in \ref{sec:intro-results}. 

To gain insight into this noncommutative framework, it may be illuminating to illustrate the case $n=1$. Given two complex vector spaces $V_1,V_2$ of finite dimensions $d_1,d_2$, respectively, we can introduce the Van den Bergh space (see \cite{BK16,Bo14,BEF,Ya})
\begin{equation*} 
\Bc^{\operatorname{VdB}}:=\{A\in \Hom(V_1,V_2),\, B\in \Hom(V_2,V_1)\mid \det(\Id_{V_2}+AB)\neq 0\}\,.
\end{equation*}
Van den Bergh proved in \cite{VdB1,VdB2} that this space is a complex quasi-Hamiltonian manifold in the sense of \cite{AMM98} for the natural action of $H=\Gl(V_1)\times \Gl(V_2)$. Furthermore, its Lie group valued moment map is given by 
\begin{equation*}
\mu^{\operatorname{VdB}}(\{A,B\}) := \left((B,A)^{-1},(A,B)\right)=\left((\Id_{V_1} +BA)^{-1},\Id_{V_2}+AB\right)\in H\,.
\end{equation*}
Note that we easily recover $(\widehat{\Bc}^{2},\widehat{\mu})$ from $(\Bc^{\operatorname{VdB}},\mu^{\operatorname{VdB}})$ when $V_1=V_2=\CC$. 
In the setting of multiplicative quiver varieties, Van den Bergh crucially observed that the pair $\big(\Bc^{\operatorname{VdB}},\mu^{\operatorname{VdB}}\big)$ and its corresponding $2$-form can be understood at the level of the quiver $\Gamma_1$. Indeed, he introduced the notion of a \emph{quasi-bisymplectic algebra} \cite{VdB2} as a noncommutative analogue of a quasi-Hamiltonian manifold, and he showed that an appropriate localisation, denoted $\Ac(\Gamma_1)$, of the path algebra of the double quiver $\overline{\Gamma}_1$ is endowed with such a structure. In fact, Van den Bergh first introduced a ``quasi-Poisson version'' of this result, when he derived in \cite{VdB1} that $\Ac(\Gamma_1)$ is a \emph{Hamiltonian double quasi-Poisson algebra}; this is the noncommutative version of a Hamiltonian quasi-Poisson manifold in the sense of \cite{AKSM02}, see Definition \ref{defn:double-quasi-Hamiltonian}.
The key point is that this structure is such that the second Euler continuants $(b,a)^{-1}$ and $(a,b)$ become noncommutative moment maps. 
Our aim is to generalise this result of Van den Bergh: we want to prove that we can attach to each quiver $\Gamma_n$ an explicit Hamiltonian double quasi-Poisson algebra structure, whose moment map is given in terms of the $(2n)$-th Euler continuants $(a_1,\dots,b_n)$ and $(b_n,\dots,a_1)^{-1}$.

\subsection{Main results}
 \label{sec:intro-results}

We fix a field $\kk$ of characteristic zero. Recall that $\overline{\Gamma}_n$ denotes the quiver on $2$ vertices and $2n$ arrows labelled $\{a_i,b_i\mid 1\leq i\leq n\}$ as in Figure \ref{fig:quiver-Gamma-n}. 
We introduce the \emph{Boalch algebra} $\mathcal{B}(\Gamma_n)$ which is obtained from the path algebra $\kk\overline{\Gamma}_n$ by requiring that the $(2n)$-th Euler continuants $(a_1,\dots,b_n)$ and $ (b_n,\dots,a_1)$ are inverted (see \ref{sec:Euler-continuants-introd} for the definition of Euler continuants in $\kk\overline{\Gamma}_n$). 
The main result of this article is Theorem \ref{Thm:MAIN}: we prove that $\mathcal{B}(\Gamma_n)$ is endowed with a double quasi-Poisson bracket $\lr{-,-}$, explicitly defined in \eqref{eq:Euler-double-bracket}, which is such that $\Phi:= (a_1,\dots,b_n) + (b_n,\dots,a_1)^{-1}$ is a multiplicative moment map. In other words, the triple $\big(\mathcal{B}(\Gamma_n),\lr{-,-},\Phi\big)$ is a Hamiltonian double quasi-Poisson algebra. 
As a consequence of this result, we obtain in Corollary \ref{cor:fission-H0} that the corresponding fission algebra $\mathcal{F}^q(\Gamma_n)$ defined by Boalch \cite{B18} carries a noncommutative Poisson structure, called an $H_0$-Poisson structure, as introduced in \cite{CB11}. Such an $H_0$-Poisson structure induces a Poisson bracket on the corresponding varieties 
\begin{equation*}
\mathcal{M}_n^{\kk}(q):= \Rep\big(\mathcal{F}^q(\Gamma_n),(d_1,d_2)\big)/\!/\big(\Gl(d_1)\times \Gl(d_2)\big)\,,
\end{equation*}
when $\kk$ is algebraically closed. 
If $\kk=\CC$, $(d_1,d_2)=(1,1)$ and $q=(\gamma,\gamma^{-1})$, we recover the wild character variety $\mathcal{M}_n(t_\gamma)$ from \ref{sec:intro-Boalch}. 
In particular, the Poisson bracket hence obtained is nondegenerate and corresponds to Boalch's symplectic form \cite{B18}, see \ref{ss:qbisymp}. We will investigate this nondegeneracy property for a general field $\kk$ in a future work.

It is important to observe that if $n=1$, the Boalch algebra $\Bc(\Gamma_1)$ coincides with Van den Bergh's algebra $\mathcal{A}(\Gamma_1)$ described above, and the Hamiltonian double quasi-Poisson structures are the same. Thus, Theorem \ref{Thm:MAIN} can be regarded as a generalisation of \cite[Theorem 6.5.1]{VdB1}.
We also note that, in the inductive proof of Theorem \ref{Thm:MAIN} carried out in Section \ref{sec:proff-thm-Euler}, formulae \eqref{eq:Convenient-formula-Euler-cont-a} and \eqref{eq:Convenient-formula-Euler-cont-b} play a key role, since they enable us to rewrite Euler continuants in terms of lower ones. To the best of our knowledge, these formulae are new and may be of independent interest. 
After that, using the correspondence between double (quasi-Poisson) brackets and noncommutative bivectors explained in \cite[\S4.2]{VdB1},  we exhibit in Proposition \ref{prop:the-nc-bivector} the bivector $\PP_n$ that gives rise to the double quasi-Poisson bracket from Theorem \ref{Thm:MAIN}. 
As in \cite[Proposition 8.3.1]{VdB2}, we expect that this result will be important in the future to prove that the double quasi-Poisson bracket \eqref{eq:Euler-double-bracket} is nondegenerate.

Furthermore, Boalch observed \cite[Remark 5]{B18} that, at certain values ($n=2$), $\mathcal{M}_n^{\kk}(q)$ is isomorphic to the Flaschka--Newell surface \cite{FN80}. This is an affine cubic surface endowed with a Poisson structure; in the setting of integrable systems, it is closely related to solutions to Painlev\'e II equation (see also \cite{CMR}). Recently, based on \cite{FN}  (see also \cite[\S5]{B18}), Bertola and Tarricone \cite{BT} explicitly wrote the Poisson bracket on the wild character variety $\mathcal{M}_n(t_\gamma)$ (with no restriction on $n$); note that an explicit expression for the corresponding symplectic 2-form is well-known by experts. Whereas we originally obtained the Hamiltonian double quasi-Poisson structure of Theorem \ref{Thm:MAIN} employing entirely noncommutative arguments, in \ref{ss:FN-Poisson} we are able to show that it induces the Poisson bracket due to Flaschka and Newell  
on $\Rep\big(\mathcal{B}(\Gamma_n),(1,1)\big)$  (i.e. before performing quasi-Hamiltonian reduction to end up with $\mathcal{M}_n(t_\gamma)$).
In particular, this shows that our Theorem  \ref{Thm:MAIN} should be regarded as the natural noncommutative counterpart of the well-known commutative theory.

Finally, one of the driving ideas in quasi-Hamiltonian geometry consists of constructing interesting moduli spaces 
from simple pieces by using the operation of \emph{fusion} \cite[\S5]{AKSM02}; 
algebraically, it endows the category of quasi-Hamiltonian manifolds with a symmetric monoidal category structure. As we explain in \ref{sec:ssec-Fusion}, Van den Bergh \cite[\S5.3]{VdB1} unveiled a noncommutative analogue of this method to obtain a double quasi-Poisson bracket and a multiplicative moment map from a Hamiltonian double quasi-Poisson algebra by identifying several idempotents.
The last important result of this article is Theorem \ref{Thm:Factorisation} which states that, after further localisation at $2n-2$ Euler continuants, the structure of Hamiltonian double quasi-Poisson algebra of $\mathcal{B}(\Gamma_n)$ unveiled in Theorem \ref{Thm:MAIN} can be obtained by fusing the idempotents in $n$ copies of $\mathcal{B}(\Gamma_1)$. 
In the simplest case $n=2$, this can be interpreted as a way to get the factorisation 
\begin{equation*}
 (a_1,b_1,a_2,b_2)=(a_1',b_1')(a_2',b_2')
\end{equation*}
of the fourth Euler continuant in terms of a product of second ones through the substitution 
\begin{equation} \label{eq:intro-change-var}
 a_1'=a_1,\,\quad b_1'=b_1\,;\qquad a_2'=a_2+(a_1,b_1)^{-1}a_1,\,\quad b_2'=b_2.
\end{equation}
The proof of Theorem \ref{Thm:Factorisation} is based on an explicit description of the Hamiltonian double quasi-Poisson algebra structure obtained after the change of coordinates \eqref{Eq:Locgen3} that generalises \eqref{eq:intro-change-var} for any $n\geq 2$. 
We also note that this result relies on the curious identity of Lemma \ref{lem:curious-identity}, which commutes the adjacent arrows in a product of Euler continuants.

\medskip

\noindent\textbf{Layout of the article.} In Section \ref{sec:nc-Poisson-geom}, we recall the basics of Van den Bergh's noncommutative quasi-Poisson geometry \cite{VdB1}. We introduce the important notions of double quasi-Poisson brackets and Hamiltonian double quasi-Poisson algebras, as well as the method of fusion. In this context, we also present an example associated with the quiver $\Gamma_1$ by Van den Bergh \cite{VdB1,VdB2}, which is generalised in the rest of this work. In Section \ref{sec:Euler-continuants}, we state as part of Theorem \ref{Thm:MAIN} that we can obtain a Hamiltonian double quasi-Poisson algebra $\Bc(\Gamma_n)$ by localisation of the path algebra $\kk\overline{\Gamma}_n$ of the double of the quiver $\Gamma_n$. This structure is such that its noncommutative moment map is given in terms of Euler continuants. We then write down the corresponding noncommutative bivector $\PP_n$ in Proposition \ref{prop:the-nc-bivector}, and we link our result to the Flaschka--Newell Poisson bracket \cite{FN} in \ref{ss:FN-Poisson}. 
The proof of Theorem \ref{Thm:MAIN} is the subject of Section \ref{sec:proff-thm-Euler}.
Finally, in Section \ref{sec:factorisation-after-localiz} we exhibit a factorisation of (a localisation of) the algebra $\Bc(\Gamma_n)$ in terms of $n$ copies of $\Bc(\Gamma_1)$ through the method of fusion. 

\medskip

\noindent\textbf{Acknowledgements.} M. F. is supported by a Rankin-Sneddon Research Fellowship of the University of Glasgow.  
D. F. is supported by the Alexander von Humboldt Stiftung in the framework of an Alexander von Humboldt professorship endowed by the German Federal Ministry of Education and Research. 
We are deeply grateful to Luis \'Alvarez-C\'onsul, Philip Boalch, Oleg Chalykh, William Crawley-Boevey and Vladimir Rubtsov for useful discussions and their encouragements, and to Thomas Br\"{u}stle for bringing \cite{Br01} to our attention. Special thanks are due to Philip Boalch for suggesting that Proposition \ref{Pr:RepCorrespond} should hold, and to the referees for useful comments.


\section{Noncommutative quasi-Poisson geometry}
\label{sec:nc-Poisson-geom}

\subsection{Hamiltonian double quasi-Poisson algebras}
\label{sec:sec-double-quasi-Ham}

Hereafter, we follow \cite{VdB1,CBEG07,F19}.

\subsubsection{Double derivations} \label{sec:ssec-double-derivations}

We fix a finitely generated associative unital algebra $A$ over a field $\kk$ of characteristic zero, and we use the unadorned notations $\otimes=\otimes_\kk$, $\Hom=\Hom_\kk$ for brevity. The \emph{opposite algebra} and the \emph{enveloping algebra} of $A$ will be denoted $A^\op$ and $A^\ee:=A\otimes A^\op$, respectively. We shall identify the category of $A$-bimodules and the category of (left) $A^\ee$-modules. Note that the underlying $A$-bimodule of $A$ carries two $A$-bimodule structures with the same underlying vector space $A\otimes A$, namely the \emph{outer} and the \emph{inner} $A$-bimodule structures, respectively denoted by $(A\otimes A)_{\out}$ and $(A\otimes A)_{\inn}$; they are given by
\begin{align*}
a_1(a'\otimes a'')a_2&:=(a_1a')\otimes (a''a_2)\quad \text{on }(A\otimes A)_{\out}\,;
\\
a_1*(a'\otimes a'')*a_2&:=(a'a_2)\otimes (a_1a'')\quad \text{on }(A\otimes A)_{\inn}\,,
\end{align*}
for all $a',a'',a_1,a_2\in A$.

Given another unital associative $\kk$-algebra $B$, $A$ is called a \emph{$B$-algebra} if there exists a unit preserving $\kk$-algebra morphism $B\to A$.
Let $M$ be an arbitrary $A$-bimodule, and $\Der_B(A,M)$ be the vector space of $B$-derivations of $A$ into $M$; that is, additive maps $\theta\colon A\to M$ satisfying the Leibniz rule and $\theta(b)=0$, for all $b\in B$.
Following \cite{CQ95}, we define the $A$-bimodule of \emph{noncommutative differential 1-forms} as $\Omega^1_B A:=\operatorname{ker}(m\colon A\otimes_B A\to A)$, where $m$ is the product of $A$. As in the commutative case, it carries a canonical $B$-derivation $\du\colon A\to\Omega^1_B A$, $a\mapsto \du{a}:=a\otimes_B 1-1\otimes_B a$, and the pair $(\Omega^1_B A,\du)$ has a universal property (\cite[\S 2]{Quillen}), which can be rephrased as 
\begin{equation}
\Der_B (A,M)\cong \Hom_{A^\ee}(\Omega^1_B A,M)\,.
\label{eq:univ-property}
\end{equation}

By taking $M=(A\otimes A)_{\out}$ in \eqref{eq:univ-property}, we define the $A$-bimodule of \emph{double derivations}
\[
\D_B A:=\Hom_{A^\ee}(\Omega^1_B A,(A\otimes A)_{\out})\cong \Der_B(A,(A\otimes A)_{\out})\,.
\]
If $\delta\in\D_B A$, we systematically use Sweedler's notation: $\delta\colon A\to A\otimes A$, $a\mapsto \delta(a)^\prime\otimes \delta(a)''$. Note that 
 the $A$-bimodule structure on $\D_B A$ is induced from the inner bimodule structure: $(a_1\delta a_2)(a):=\delta(a)'a_2\otimes a_1\delta(a)^{\prime\prime}$, for all $a,a_1,a_2\in A$.

The $A$-bimodule $\D_BA$ carries a distinguished element, called the \emph{gauge element} $E$, defined by
\begin{equation}
E\colon A\longrightarrow A\otimes A,\quad  a\,\longmapsto\, a\otimes 1-1\otimes a.
\label{eq:gauge-element}
\end{equation}
Note that this element was also introduced in \cite[\S 3]{CBEG07} and denoted by $\Delta$.

Following \cite[\S 3.1]{VdB1}, we define the \emph{algebra of poly-vector fields} as $D_BA=T_A(\D_B A)$, which is the tensor algebra of the $A$-bimodule $\D_B A$ (over $A$). The $n$-th component is the $n$-fold $(D_BA)_n=(\D_B A)^{\otimes_ An}$ for $n\geq 1$.

\subsubsection{Double derivations for quivers} \label{ssec:Quivers}

A \emph{quiver} $Q$ is an oriented graph. We denote its set of vertices by $I$, and its set of arrows (oriented edges) by $Q$. We can define on $Q$ the tail (resp. head) map $t\colon Q\to I$ (resp. $h\colon Q\to I$) assigning to any arrow $a\in Q$ its tail/starting vertex $t(a)\in Q$ (resp. its head/ending vertex $h(a)\in Q$). 
We can then form the double $\overline{Q}$ of $Q$ with the same vertex set $I$ by adding an opposite arrow $a^\ast : h(a)\to t(a)$ for each $a\in Q$. 
We naturally extend $h,t$ to $\overline{Q}$.

Now, given a quiver $Q$ (not necessarily a double quiver), we form the path algebra $\kk Q$ which is the $\kk$-algebra generated by the arrows $a\in Q$ and idempotents $(e_s)_{s\in I}$ labelled by the vertices such that 
\begin{equation}
  a=e_{h(a)}a e_{t(a)}\,, \quad e_s e_t = \delta_{st}\,e_s\,,
  \label{eq:convention-arrows}
\end{equation}
and the product is given by concatenation of paths (if possible; otherwise the idempotent decomposition makes the product vanish). In this article we will regard $\kk Q$ as a $B$-algebra with $B=\oplus_{s\in I}\kk e_s$. 
Then, if $A=\kk Q$, the $A$-bimodule of double derivations $\D_B A$ is generated as an $A$-bimodule by $\big\{\frac{\partial}{\partial a}\mid a\in Q\big\}$, where if $b\in Q$,
\[
\frac{\partial}{\partial a}(b):=\delta_{ab}e_{h(a)}\otimes e_{t(a)}\, .
\] 

\begin{rem}\label{rem:convention-on-arrows}
Note that \eqref{eq:convention-arrows} implies that we read paths from right to left. This is the convention followed by Boalch in \cite{B15,B18} and Crawley-Boevey and Shaw \cite{CBShaw}, but it contrasts with  \cite{VdB1}, where Van den Bergh uses the \emph{opposite} convention (\emph{i.e.}, $a=e_{t(a)}a e_{h(a)}$).  \end{rem}

\subsubsection{Double quasi-Poisson brackets}
Let $A$ be a finitely generated associative unital $B$-algebra.
A \emph{$B$-linear  $n$-bracket} \cite{VdB1} is a map $\lr{-,\cdots,-}\colon A^{\times n}\to A^{\otimes n}$  
which
\begin{enumerate}
 \item [\textup{(a)}]
 is linear in all its arguments;
  \item  [\textup{(b)}]
  is cyclically antisymmetric, i.e. 
$$\tau_{(1\cdots n)}\circ\lr{-,\cdots,-}\circ \tau^{-1}_{(1\cdots n)}=(-1)^{n+1}\lr{-,\cdots,-}$$ 
for the cyclic permutation $\tau_{(1\cdots n)}(a_1\otimes \cdots\otimes a_{n-1}\otimes a_n)=a_n\otimes a_1\otimes\cdots\otimes a_{n-1}$; 
\item [\textup{(c)}]
 vanishes when its last argument (hence any argument) is in the image of $B$;
 \item [\textup{(d)}]
 is a derivation $A\to A^{\otimes n}$ in its last argument for the outer bimodule structure on $A^{\otimes n}$ given by 
  $b(a_1\otimes a_2 \otimes\cdots \otimes a_{n-1}\otimes a_n)c= (ba_1)\otimes a_2 \otimes\cdots \otimes a_{n-1}\otimes (a_nc)$.
\end{enumerate}
We call a 2- and 3- bracket a \emph{double} and \emph{triple} bracket, respectively. For the reader's convenience, we explicitly write the definition of the former. 
\begin{defn}[\cite{VdB1}]
A \emph{$B$-linear double bracket} on the $B$-algebra $A$ is a $\kk$-bilinear map $\lr{-,-}:A\times A \to A \otimes A$, which satisfies for any $a,b,c \in A$,
\begin{align} 
 \lr{a,b}&=-\tau_{(12)}\lr{b,a}\,, &&\text{\emph{(cyclic antisymmetry)}},  \label{Eq:cycanti}
\\
 \lr{a,bc}&=\lr{a,b}c+b\lr{a,c}\,,  &&\text{\emph{(right Leibniz rule)}}, \label{Eq:outder}
\end{align}
and which is such that $\lr{-,e}=0$ for all $e\in B$. 
\label{def:quasi-Poisson -bracket}
\end{defn}
Assuming that \eqref{Eq:cycanti} holds, it is easy to check that \eqref{Eq:outder} is equivalent to 
\begin{equation}\label{Eq:inder}
 \lr{bc,a}=\lr{b,a}\ast c+b\ast\lr{c,a} \,,\qquad \text{(left Leibniz rule)}.
\end{equation}
From the derivation rules \eqref{Eq:outder} and \eqref{Eq:inder}, observe that it suffices to define double brackets on generators of $A$. 
From now on, if the context is clear, by a double bracket we will mean a $B$-linear double bracket.
The following result will be used in \ref{sec:nc-bivector}:

\begin{prop}[\cite{VdB1}, Proposition 4.1.1] \label{prop:VdB-mu}
There exists a well-defined linear map
\begin{align*}
\mu\colon (D_B A)_n\longrightarrow&\, \{B\text{-linear }n\text{-brackets on }A\}
\\
Q\,\,\longmapsto&\, \lr{-,\cdots,-}_Q=\sum^{n-1}_{i=0}(-1)^{(n-1)i}\,\tau^{i}_{(1\cdots n)}\circ\lr{-,\cdots,-}^\sim_Q\circ\tau^{-i}_{(1\cdots n)}\, ,
\end{align*}
where we set for $Q=\delta_1\cdots\delta_n$ and $a_1,\dots,a_n\in A$, 
\[
\lr{a_1,\cdots,a_n}^\sim_Q=\delta_n(a_n)'\delta_1(a_1)''\otimes \delta_1(a_1)'\delta_2(a_2)''\otimes\cdots\otimes\delta_{n-1}(a_{n-1})'\delta_n(a_n)'\, .
\]
\end{prop}
Note that if $\Omega^1_B A$ is a projective $A^\ee$-module--- for instance, if $A$ is the path algebra of a quiver, $\mu$ is an isomorphism (see \cite[Proposition 4.1.2]{VdB1}).
Also, since it will be used in \ref{sec:nc-bivector}, we specialise the previous formula to noncommutative bivectors $Q=\delta_1\delta_2\in (D_B A)_2=(\D_B A) \otimes_A (\D_B A)$:
\begin{equation} \label{Eq:bivector-bracket}
 \lr{a_1,a_2}_{\delta_1 \delta_2}=\delta_2(a_2)' \delta_1(a_1)'' \otimes \delta_1(a_1)' \delta_2(a_2)'' - \delta_1(a_2)' \delta_2(a_1)'' \otimes \delta_2(a_1)' \delta_1(a_2)'' \,,
\end{equation}
for any $a_1,a_2\in A$. 

To develop a noncommutative version of quasi-Poisson geometry, as introduced in \cite{AKSM02}, we need a non-vanishing noncommutative Jacobi identity.
Firstly, given a double bracket $\lr{-,-}$ on $A$, and $a,b,c\in A$, Van den Bergh \cite{VdB1} introduced a suitable extension of the double bracket, given by $\lr{a,b\otimes c}_L:=\lr{a,b}\otimes c\in A^{\otimes 3}$. Next, he defined a natural triple bracket associated with $\lr{-,-}$ by setting
\begin{equation}
\label{Eq:TripBr}
 \lr{a,b,c}:=\lr{a,\lr{b,c}}_L+\tau_{(123)}\lr{b,\lr{c,a}}_L+\tau_{(132)}\lr{c,\lr{a,b}}_L\, .
\end{equation}
Note that  
$\tau_{(123)}(a_1\otimes a_2\otimes a_3)=a_{3}\otimes a_1 \otimes a_{2}$
and
$\tau_{(132)}(a_1\otimes a_2\otimes a_3)=a_2\otimes a_3\otimes a_1$,
for all $a_1,a_2,a_3\in A$. 
Secondly, we assume that the unit in $A$ admits a decomposition $1=\sum_{s\in I} e_s$ in terms of a finite set of orthogonal idempotents, i.e. $|I|\in \N^\times$ and $e_s e_t = \delta_{st} e_s$.  In that case, we view $A$ as a $B$-algebra for $B=\oplus_{s\in I} \kk e_s$. 
Now, using the distinguished double derivation $E$ defined in \eqref{eq:gauge-element}, we have $E_s=e_s E e_s\in D_B A$ and we can apply Proposition \ref{prop:VdB-mu} to $\sum_{s\in I}E_s^3\in (D_B A)_3$ and define the following triple bracket:
\begin{equation}
\begin{aligned}
&\quad \lr{a,b,c}_{\qP}:=\frac{1}{12}\sum_{s\in I} \lr{a,b,c}_{E^3_s}
\\
&=\frac14 \sum_{s\in I} \Big(
c e_s a \otimes e_s b \otimes e_s  - c e_s a \otimes e_s \otimes b e_s - c e_s \otimes a e_s b \otimes e_s 
+ c e_s \otimes a e_s \otimes b e_s \\
&\qquad \quad - e_s a \otimes e_s b \otimes e_s c + e_s a \otimes e_s \otimes b e_s c + e_s \otimes a e_s b \otimes e_s c - e_s \otimes a e_s \otimes b e_s c \Big)\,,
\end{aligned}
\label{eq:triple-bracket-E3}
\end{equation}
for any $a,b,c\in A$.

\begin{defn}[\cite{VdB1}]  \label{def:quasi-Poisson-bracket}
Let $A$ be an associative $B$-algebra endowed with a double bracket $\lr{-,-}$.
We say that the double bracket is \emph{quasi-Poisson} if the triple brackets \eqref{Eq:TripBr} and \eqref{eq:triple-bracket-E3} coincide:
\begin{equation}
  \lr{a,b,c}=\lr{a,b,c}_{\qP}\,, 
   \label{qPabc}
\end{equation}
 for all $a,b,c\in A$. The pair $(A,\lr{-,-})$ is called a \emph{double quasi-Poisson algebra}.
\end{defn}
Note that $\lr{-,-,-}$ and $\lr{-,-,-}_{\qP}$ are triple brackets according to the terminology given at the beginning of this subsection. Thus, since triple brackets are cyclic and enjoy a derivation property, it suffices to check \eqref{qPabc} on generators of $A$.

\subsubsection{Multiplicative moment maps}
In addition to the introduction of a noncommutative version of quasi-Poisson brackets, Van den Bergh adapted the notion of Lie group valued moment maps from \cite{AMM98,AKSM02} as follows:

\begin{defn}[\cite{VdB1}]
Let $(A,\lr{-,-})$ be a double quasi-Poisson algebra over $B=\oplus_{s\in I} \kk e_s$. 
An invertible element $\Phi\in A$ is a \emph{multiplicative moment map}  if it can be decomposed as $\Phi=\sum_{s\in I}\Phi_s$ with $\Phi_s\in e_sAe_s$, and it satisfies that  for all $a\in A$ and $s\in I$ 
\begin{equation} \label{Phim}
 \lr{\Phi_s,a}=\frac12 \Big(ae_s\otimes \Phi_s-e_s \otimes \Phi_s a +  a \Phi_s \otimes e_s-\Phi_s \otimes e_s a\Big)\,.
\end{equation}
Then we call the triple $(A,\lr{-,-},\Phi)$ a \emph{Hamiltonian double quasi-Poisson algebra}\footnote{We use the terminology of \cite{VdB2,FF}. Note that the triple $(A,\lr{-,-},\Phi)$ is also called a \emph{Hamiltonian algebra} in \cite{VdB1,F19,F20}.}. 
\label{defn:double-quasi-Hamiltonian}
\end{defn}
The invertibility of $\Phi$ implies that $\Phi_s^{-1}:=e_s\Phi^{-1}e_s$ is an inverse for  $\Phi_s$ in $e_s A e_s$. 
Then, the identity \eqref{Phim} yields for all $a\in A$ that  
\begin{equation} \label{PhimInv}
 \lr{\Phi_s^{-1},a}=-\frac12 \Big(a \Phi_s^{-1}\otimes e_s-\Phi_s^{-1} \otimes e_s a +  a e_s \otimes \Phi_s^{-1}-e_s \otimes \Phi_s^{-1} a\Big)\,.
\end{equation}

Finally, we can define morphisms between such algebras, which we use in Section \ref{sec:factorisation-after-localiz}.

\begin{defn}[\cite{F20}] \label{def:iso-DQHA}
Let $(A,\lr{-,-},\Phi)$ and $(A',\lr{-,-}',\Phi')$ be two Hamiltonian double quasi-Poisson algebras over $B$. An isomorphism $\psi:A\to A'$ of $B$-algebras is said to be an \emph{isomorphism of Hamiltonian double quasi-Poisson algebras} if it preserves the double quasi-Poisson brackets and the multiplicative moment maps, that is,
\begin{equation}
 \lr{\psi(a_1),\psi(a_2)}'=(\psi\otimes \psi)(\lr{a_1,a_2})\,, \quad \text{ and }\quad \psi(\Phi)=\Phi'\,,
 \end{equation}
for all $a_1,a_2\in A$.
\end{defn}

\subsection{Fusion of Hamiltonian double quasi-Poisson algebras} \label{sec:ssec-Fusion}

Given an algebra $A$, the method of fusion allows to construct a new algebra $A^f$ by identifying two orthogonal idempotents in $A$. More importantly, if $A$ is a Hamiltonian double quasi-Poisson algebra, it is possible to obtain such a structure on $A^f$ after performing fusion. Below, we recall the main results associated with this method, which are due to Van den Bergh \cite[\S 2.5, 5.3]{VdB1} (see also \cite[\S2.1]{Br01} without the perspective of double brackets); 
alternative presentations can be found in \cite{F19,F20}. 
These results are noncommutative analogues of \cite[\S5]{AKSM02}. 

As in \ref{sec:sec-double-quasi-Ham}, we assume\footnote{We can directly have in mind that $B=\oplus_{s\in I} \kk e_s$, though the construction of fusion algebra is more general.} that there exist orthogonal idempotents $e_i,e_j\in B$.  
First, we extend the algebra $A$ along the (ordered) idempotents $(e_i,e_j)$  as 
\begin{equation}
  \bar{A}:=A \ast_{\kk e_i \oplus \kk e_j \oplus \kk \he} \Big(\Mat_2(\kk)\oplus \kk \he\Big ) = A \ast_B \bar{B}\,.
\end{equation}
Here, we have set $\he=1-e_i-e_j$, and $\Mat_2(\kk)$ is the $\kk$-algebra generated by the elements $e_i=e_{ii},e_{ij},e_{ji},e_j=e_{jj}$ subject to the matrix relations  $e_{st}e_{uv}=\delta_{tu}e_{sv}$. 
Second, we can introduce the \emph{fusion algebra} $A^f_{e_j \to e_i}$ obtained by fusing $e_j$ onto $e_i$, which we abbreviate as $A^f$. It is simply defined as  
\begin{equation}
  A^f:=\, \epsilon \bar{A} \epsilon\,, \quad \text{for } \epsilon=1-e_j\,.
\end{equation}
The fusion algebra $A^f$ can be seen as dismissing the  elements of $e_j \bar{A} + \bar{A} e_j$ inside  $\bar{A}$. 
Note that $A^f$ is a $B^f$-algebra, where  $B^f:=\epsilon \bar{B} \epsilon$. 
Furthermore, there exists a projection onto the fusion algebra 
\begin{equation} \label{Eq:prfus}
 A \longrightarrow A^f\,, \quad a \longmapsto a^f:=\epsilon a \epsilon + e_{ij}ae_{ji} + e_{ij}a \epsilon + \epsilon a e_{ji}\,,
\end{equation}
from which we can get a set of generators in $A^f$, which we split into four types. Namely, 
  \begin{align}
   (\text{first type})&\qquad\qquad t^f=t\,, \quad & &\text{for }t \in \epsilon A \epsilon\,, \label{type1}\\
   (\text{second type})&\qquad\qquad u^f=e_{ij}u\,, \quad & &\text{for }u \in e_j A \epsilon\,,\label{type2} \\
   (\text{third type})&\qquad\qquad v^f=v e_{ji}\,, \quad & &\text{for }v \in \epsilon A e_j\,, \label{type3} \\
   (\text{fourth type})&\qquad\qquad w^f=e_{ij} w e_{ji}\,, \quad & &\text{for }w \in e_j A e_j\,. \label{type4}
  \end{align}
  
We assume that $(A,\lr{-,-},\Phi)$ is a Hamiltonian double quasi-Poisson algebra over $B=\oplus_{s\in I} \kk e_s$ from now on.   Then, the double bracket uniquely extends from $A$ to $\bar A$ by requiring it to be $\bar{B}$-linear, and it can then be restricted to $A^f$. We also denote the double bracket hence induced on $A^f$ by $\lr{-,-}$. The image in $A^f$ of the component $\Phi_s$ of the multiplicative moment map is either $\Phi_s^f=\epsilon \Phi_s\epsilon=\Phi_s$ if $s\neq j$, or $\Phi_s^f=e_{ij}\Phi_j e_{ji}$ if $s=j$. However, the data $(\lr{-,-},\Phi^f)$ \emph{does not} turn $A^f$ into a Hamiltonian double quasi-Poisson algebra. 
\begin{lem}\label{lem:fus-double-bracket}
 There exists a unique double bracket $\lfus{-,-}$ on $A^f$ such that for any 
\begin{equation} \label{Eq:lem-fus}
t,\tilde{t} \in \epsilon A \epsilon,\quad  u,\tilde{u} \in e_j A \epsilon,\quad v,\tilde{v} \in \epsilon A e_j, \quad w,\tilde{w} \in e_j A e_j,
\end{equation}
we have that 
\begin{subequations}
\begin{align}
& \lfus{\epsilon t \epsilon , \epsilon \tilde{t} \epsilon}=0\,, \label{tt}\\
&\lfus{\epsilon t \epsilon , e_{ij} u \epsilon}=\frac12 \big( e_i \otimes t e_{ij}u - e_i t  \otimes e_{ij}u\big)\,, \label{tu}\\
&\lfus{\epsilon t \epsilon ,\epsilon v e_{ji} }=\frac12 \left( v e_{ji} t   \otimes e_i - v e_{ji} \otimes t e_i\right)\,, \label{tv}\\
&\lfus{\epsilon t \epsilon , e_{ij} w e_{ji}}=\frac12 \left( e_{ij}we_{ji}t \otimes e_i + e_i \otimes t e_{ij} w e_{ji} - e_{ij} w e_{ji} \otimes t e_i - e_i t \otimes e_{ij} w e_{ji}\right), \label{tw}\\
&\lfus{e_{ij} u \epsilon , e_{ij} \tilde{u} \epsilon}= \frac12 (e_i \otimes e_{ij}u e_{ij}\tilde{u}-e_{ij}\tilde{u} e_{ij}u \otimes e_i) \,, \label{uu}\\
&\lfus{e_{ij} u \epsilon ,\epsilon v e_{ji} }=\frac12(e_{ij}u \otimes e_i ve_{ji}-v e_{ji} \otimes e_{ij}u e_i)\,, \label{uv}\\
&\lfus{e_{ij} u\epsilon , e_{ij} w e_{ji}}=\frac12(e_i \otimes e_{ij}u e_{ij}we_{ji} - e_{ij}w  e_{ji} \otimes e_{ij}u e_i) \,, \label{uw} \\
&\lfus{\epsilon v e_{ji},\epsilon \tilde{v} e_{ji} }=\frac12( \tilde{v} e_{ji}v e_{ji}\otimes e_i - e_i \otimes v e_{ji}\tilde{v} e_{ji})\,, \label{vv}\\
&\lfus{\epsilon v e_{ji} , e_{ij} w e_{ji}}=\frac12(e_{ij}we_{ji}ve_{ji}\otimes e_i - e_i v e_{ji}\otimes e_{ij}we_{ji} ) \,, \label{vw} \\
&\lfus{e_{ij} w e_{ji} , e_{ij} \tilde{w} e_{ji}}=0 \,. \label{ww}
\end{align}
\end{subequations}
\end{lem}
\begin{proof}
Defining a double bracket on $A^f$ is equivalent to specifying the double bracket on a set of generators. But we know that the image in $A^f$ of all the elements of the form \eqref{Eq:lem-fus} provides a set of generators, thus specifying \eqref{tt}--\eqref{ww} and using the cyclic antisymmetry \eqref{Eq:cycanti} is enough to determine uniquely the double bracket $\lfus{-,-}$. 
Since any relation in $A$ can be decomposed in components of $\epsilon A \epsilon$, $e_j A \epsilon$, $\epsilon A e_j$ and $e_j A e_j$ which, in turn, induce relations in $A^f$, this yields a well-defined double bracket $\lfus{-,-}$. 
\end{proof}
The double bracket from Lemma \ref{lem:fus-double-bracket} is due to Van den Bergh \cite[Theorem 5.3.1]{VdB1}; it was introduced to get a Hamiltonian double quasi-Poisson structure on $A^f$. 

\begin{thm}[\cite{VdB1,F19}] \label{Thm:IsoFusqHam}
Let $A^f=A^f_{e_j \to e_i}$ be the fusion algebra. Then $A^f$ is a Hamiltonian double quasi-Poisson algebra, whose double quasi-Poisson bracket is given by
\begin{equation} \label{dgalf}
  \lr{-,-}^f:= \lr{-,-} + \lfus{-,-}\,,
\end{equation}
where $\lr{-,-}$ is induced in $A^f$ by the double quasi-Poisson bracket of $A$, while $\lfus{-,-}$ is defined in Lemma \ref{lem:fus-double-bracket}. 
Its multiplicative moment map is given by 
\begin{equation}
 \Phi^{f\!f}=\Phi^f_i \Phi_j^f+\sum_{s\neq i,j}\Phi_s^f = \Phi_i e_{ij}\Phi_j e_{ji} + \sum_{s\neq i,j} \Phi_s\,.
\end{equation}
\end{thm}
This result was derived in \cite[Theorems 5.3.1 and 5.3.2]{VdB1} under mild assumptions, which were removed in \cite[Theorems 2.14 and 2.15]{F19}.  

Let us note that fusing the two idempotents $e_i,e_j$ in the opposite order yields an isomorphic algebra. 
Indeed, each $a\in A$ has a unique image in both $A^f_{e_j \to e_i}$ and $A^f_{e_i \to e_j}$ through the projection map  \eqref{Eq:prfus}. By mapping one projection onto the other, we get an algebra isomorphism $A^f_{e_j \to e_i}\longrightarrow A^f_{e_i \to e_j}$ which is induced by the identity map on $A$. However, this is \emph{not} an isomorphism of Hamiltonian double quasi-Poisson algebras as in Definition \ref{def:iso-DQHA}. To preserve the  double quasi-Poisson structure, one needs a different isomorphism 
$A^f_{e_j \to e_i}\longrightarrow A^f_{e_i \to e_j}$ that involves (images of) $\Phi_j\in A$, see \cite[\S4.1.1]{F20}.

\subsection{\texorpdfstring{$H_0$}{H_0}-Poisson algebras}
\label{sec:H0-Poisson}
The zeroth Hochschild homology of $A$, denoted $H_0(A)$, is the vector space obtained by identifying all the commutators to zero. This means that $H_0(A)=A/[A,A]$, where $[A,A]\subset A$ is spanned by commutators. We write $\overline{a}$ for the image of $a\in A$ under the natural projection map $A\to A/[A,A]$.

\begin{defn}[\cite{CB11}]\label{defn:H0-Poisson}
Let  $\{-,-\}$ be a Lie bracket on $H_0 (A)$. We say that $A$ is endowed with an \emph{$H_0$-Poisson structure} if, for all $\overline{a}\in H_0(A)$, the map $\{\overline{a},-\}\colon H_0(A)\to H_0(A)$, which is obtained from the Lie bracket, is induced by a derivation $A\to A$.
\label{def:H0-Poisson-str}
\end{defn}

Interestingly, double quasi-Poisson brackets induce $H_0$-Poisson structures as follows. 
By \cite[\S2.4 and \S5.1]{VdB1}, the \emph{associated bracket} $\{-,-\}:=m\circ\lr{-,-}:A\to A$ obtained by multiplication descends to a Lie bracket on $H_0(A)$. Moreover, since the associated bracket on $A$ is a derivation in its second argument, we get an $H_0$-Poisson structure on $A$.  
Furthermore, $H_0$-Poisson structures naturally arise after performing quotients of Hamiltonian double quasi-Poisson algebras  
using the corresponding multiplicative moment map as in Definition \ref{defn:double-quasi-Hamiltonian}. 

\begin{prop}[\cite{VdB1}]
\label{prop:double-quasi-Ham-red}
Let $(A,\lr{-,-},\Phi)$ be a Hamiltonian double quasi-Poisson algebra. Fix $q\in B^\times$ and set $\overline{A}^q=A/(\Phi-q)$. 
Then the associated bracket $\{-,-\}$ descends to an $H_0$-Poisson structure on $\overline{A}^q$.
\end{prop}

\begin{rem}
Proposition \ref{prop:double-quasi-Ham-red} is an analogue in the noncommutative setting of quasi-Hamiltonian reduction \cite{AKSM02}. 
To see this, note that a Hamiltonian double quasi-Poisson algebra $(A,\lr{-,-},\Phi)$ induces a structure of Hamiltonian quasi-Poisson variety on its representation spaces--- see \cite[\S7.12 and \S7.13]{VdB1}. 
This structure further induces a Poisson structure on the GIT quotient obtained from the representation space of $\overline{A}^q=A/(\Phi-q)$. This (commutative) Poisson structure is determined by the (noncommutative) $H_0$-Poisson structure on $\overline{A}^q$ obtained in Proposition \ref{prop:double-quasi-Ham-red}; see \cite[Theorem 4.5]{CB11}. 
\end{rem}

\subsection{Example from a one-arrow quiver}
\label{sec:CBS-multiplicative}

Let $\Gamma_1:=A_2$ be the quiver with one arrow $a:1\to 2$, and $\overline{\Gamma}_1$ be its double with the additional arrow $b=a^\ast:2\to 1$. 
We form the path algebra $\kk\overline{\Gamma}_1$ as explained in \ref{ssec:Quivers}.
Now, let $\Ac(\Gamma_1)$ be the algebra obtained by universal localisation from the set $S=\{1+ab,1+ba \}$. 
This is equivalent to add the local inverse $(e_2+ab)^{-1}$ to $e_2+ab$ in $e_{2}\Ac(\Gamma_1) e_{2}$, as well as to add $(e_1+ba)^{-1}$ to $e_1+ba$ in $e_{1}\Ac(\Gamma_1) e_{1}$. 
Finally, note that $\Ac(\Gamma_1)$ is a $B$-algebra for $B=\kk e_1\oplus \kk e_2$.

Following \cite{CBShaw}, we define the \emph{multiplicative preprojective algebra}\footnote{The definition for an arbitrary quiver depends on an ordering, which is irrelevant in the case of $\Gamma_1$.} of $\Gamma_1$ with parameter $q=(q_1,q_2)\in(\kk^\times)^2$ as the algebra 
\begin{equation}
 \Lambda^q(\Gamma_1):=\Ac(\Gamma_1)/R^q\,,
\end{equation}
where $R^q$ is the ideal generated by 
\begin{equation} \label{Eq:MQV1}
 e_2+ab=q_2 e_2\,, \quad (e_1+ba)^{-1}=q_1 e_1\,.
\end{equation}

\begin{thm}[\cite{VdB1}, Theorem 6.5.1 and Proposition 6.8.1]
\hfill
\begin{enumerate}
\item [\textup{(i)}]
The algebra $\Ac(\Gamma_1)$ possesses a Hamiltonian double quasi-Poisson structure given by the double quasi-Poisson bracket
\begin{align*}
\lr{a,a}&=\;\;\,0\,,\qquad \lr{b,b}=0\,;
\\
\lr{a,b}&=\;\;\,\frac{1}{2}\Big(ba\otimes e_{2}+e_{1}\otimes ab\Big)+e_{1}\otimes e_{2}\,;
\\
\lr{b,a}&=-\frac{1}{2}\Big(e_{2}\otimes ba+ab\otimes e_{1}\Big)-e_{2}\otimes e_{1}\,;
\end{align*}
and by the multiplicative moment map 
\[
\Phi=(e_2+ab)+(e_1+ba)^{-1}\,.
\]
\item [\textup{(ii)}] 
The multiplicative preprojective algebra (with parameter $q$) $\Lambda^q(\Gamma_1)$ is endowed with an $H_0$-Poisson structure, as defined in \cite{CB11}.
\end{enumerate}
\label{tm:VdB-mult-preproj-quasi-Hamilt}
\end{thm}

\begin{rem} \label{rem:VdB-double}
Note that Theorem \ref{tm:VdB-mult-preproj-quasi-Hamilt} can be adapted to an arbitrary quiver $Q$. Indeed, it suffices to use fusion and Theorem \ref{Thm:IsoFusqHam}. We will not need this general result which the reader can find in \cite[Theorem 6.7.1]{VdB1} (keeping in mind Remark \ref{rem:convention-on-arrows}). A particular case of this generalisation is derived in \ref{ss:Afus} for the quiver $Q=\Gamma_n$.
\end{rem}


\section{Euler continuants and Hamiltonian double quasi-Poisson algebras} 
\label{sec:Euler-continuants}

\subsection{Euler continuants with idempotents}
\label{sec:Euler-continuants-introd}
\allowdisplaybreaks

Given an integer $n\geq 1$, let $\Gamma_n$ be the quiver whose set of vertices is $\{1,2\}$, and whose set of arrows is $\{a_i\mid 1\leq i \leq n\}$, where $a_i$  is an arrow from the vertex $1$ to $2$. 
We form the double $\overline{\Gamma}_n$ of $\Gamma_n$ by adding the arrows $\{b_i \mid 1\leq i \leq n\}$, where $b_i$ is an arrow from the vertex $2$ to $1$, see Figure \ref{fig:quiver-Gamma-n}. 

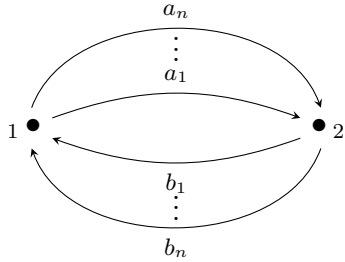
\begin{figure} 
   \begin{tikzpicture}
 \node  (vA) at (-2,0) {${}_1\,\bullet$};
 \node  (vB) at (2,0) {$\bullet\,{}_2$};
  \node  (dot1) at (0,1.15) {$\vdots$};
  \node  (dot2) at (0,-0.95) {$\vdots$};
   \path[->,>=stealth,font=\scriptsize]  
   (vA) edge [bend left=20,looseness=1] node[above] {$a_1$}  (vB) ;
   \path[->,>=stealth,font=\scriptsize]
   (vA) edge [bend left=70,looseness=1] node[above] {$a_n$} (vB) ;
      \path[->,>=stealth,font=\scriptsize]  
   (vB) edge [bend left=20,looseness=1] node[below] {$b_1$}  (vA) ;
   \path[->,>=stealth,font=\scriptsize]
   (vB) edge [bend left=70,looseness=1] node[below] {$b_n$} (vA) ;
   \end{tikzpicture} 
\caption{The quiver $\overline{\Gamma}_n$ with $2n$ arrows is the double of $\Gamma_n$, which is only formed of the $n$ arrows $a_i:1\to 2$.}
\label{fig:quiver-Gamma-n}
\end{figure}

We form the path algebra $\kk \overline{\Gamma}_n$ of $\overline{\Gamma}_n$ as in \ref{ssec:Quivers}. This allows us to see this algebra as being generated by the orthogonal idempotents $e_1,e_2$ and the symbols $\{a_i,b_i\mid 1\!\leq\! i\!\leq\! n\}$ subject to the relations 
\begin{equation}
 e_1+e_2=1\,, \quad a_i=e_2a_ie_1\,, \quad b_i=e_1b_ie_2,\,\, \text{ for }\, 1\leq i\leq n\,.
\end{equation}

Since $a_ib_j$ and $b_ja_i$ are not trivially vanishing in $\kk \overline{\Gamma}_n$, we can slightly adapt the definition of Euler continuants from Section \ref{S:Intro} as follows. We introduce the  first and second Euler continuants by setting 
\begin{equation}
\begin{aligned}
  (a_i)&:=a_i\in e_2(\kk \overline{\Gamma}_n)e_1\,,&\quad (b_i)&:=b_i \in e_1(\kk \overline{\Gamma}_n)e_2\,, \\
  (a_i,b_j)&:= e_2+a_ib_j\in e_2(\kk \overline{\Gamma}_n)e_2\,,&\quad (b_i,a_j)&:= e_1+b_ia_j\in e_1(\kk \overline{\Gamma}_n)e_1\,,
\end{aligned}
\end{equation}
for $i,j=1,\ldots,n$. We also put $(a_i,a_j)=0=(b_i,b_j)$ as $a_ia_j=0=b_ib_j$ in $\kk \overline{\Gamma}_n$. 
 This ensures that the first and second Euler continuants respect the idempotent decomposition, i.e. if $y_1,y_2\in \{a_i,b_i\mid 1\leq i\leq n\}$, we have $(y_1,y_2)=e_{h(y_1)}(y_1,y_2)e_{t(y_2)}$.
Then, given elements $y_1,\ldots,y_k\in \{a_i,b_i\mid 1\leq i\leq n\}$ for $k\geq 3$, we have that the $k$-th Euler continuant is obtained recursively using the rule
\begin{equation}
(y_1,\dots,y_k)=(y_1,\dots,y_{k-1})y_k+(y_1,\dots,y_{k-2})\,.
\label{eq:def-Euler-continuant-very-beginning}
\end{equation} 
We will only consider the Euler continuant $(y_1,\dots,y_k)$ when the elements $\{y_\ell\}$ are alternating with respect to the generators $\{a_i,b_i\mid 1\leq i\leq n\}$; that is if $y_\ell=a_i$ (resp. $y_\ell=b_i$) then $y_{\ell+1}=b_{j}$ (resp. $y_{\ell+1}=a_{j}$) for some $1\leq i,j\leq n$.

\begin{rem}\label{rmk:alt-Euler-cont-def}
Alternatively to \eqref{eq:def-Euler-continuant-intro-1}, observe that  the recursive rule used to define Euler continuants can be expressed as
\begin{equation}
(x_1,\dots,x_k)=x_1(x_2,\dots,x_k)+(x_3,\dots,x_k)\,.
\label{eq:def-Euler-continuant-intro-2}
\end{equation}
Then, when dealing with idempotents, we can also use \eqref{eq:def-Euler-continuant-intro-2} instead of \eqref{eq:def-Euler-continuant-very-beginning} to define the $k$-th Euler continuants for $k\geq 3$.
\end{rem}

We close this subsection by proving several identities involving Euler continuants that will be important for later purposes.

\begin{lem}
For $n\geq 1$, we have  that
\begin{equation} 
(a_1,b_1,\dots, b_{n-1},a_n)=\sum^{n+1}_{\ell=3}\Big[(a_1,\dots,b_{\ell-2})a_{\ell-1}\Big]+a_1\, , \label{eq:Convenient-formula-Euler-cont-a}
\end{equation}
and for $n\geq 2$, we have that
\begin{equation}
(a_1,b_1,\dots,a_n ,b_n)=(a_1,\dots ,b_{n-1})(a_n,b_n)+\sum^n_{\ell=3}\Big[(a_1,\dots,b_{\ell-2})a_{\ell-1}b_n\Big]+a_1b_n\,. \label{eq:Convenient-formula-Euler-cont-b}
\end{equation}
\end{lem}

\begin{proof}
We prove \eqref{eq:Convenient-formula-Euler-cont-a} by induction. The base case occurs when $n=1$, which follows by definition: $(a_1)=a_1$. 
Now, by \eqref{eq:def-Euler-continuant-very-beginning} and the inductive hypothesis, we have that  
\begin{align*}
(a_1,b_1,&\dots ,a_{n-1},b_{n-1},a_n)=(a_1,b_1,\dots ,a_{n-1},b_{n-1})a_n+(a_1,b_1,\dots,a_{n-1})
\\
&=(a_1,b_1,\dots ,a_{n-1},b_{n-1})a_n+\sum^{n}_{\ell=3}\Big[(a_1,\dots,b_{\ell-2})a_{\ell-1}\Big]+a_1
\\
&=\sum^{n+1}_{\ell=3}\Big[(a_1,\dots,b_{\ell-2})a_{\ell-1}\Big]+a_1\, .
\end{align*} 

To prove \eqref{eq:Convenient-formula-Euler-cont-b}, it suffices to use the recursive relation \eqref{eq:def-Euler-continuant-very-beginning} and \eqref{eq:Convenient-formula-Euler-cont-a}: 
\begin{align*}
(a_1,b_1,&\dots,a_n ,b_n)=(a_1,\dots ,a_n)b_n+(a_1,\dots ,b_{n-1})
\\
&=\Big[(a_1,\dots ,b_{n-1})a_n+ (a_1,\dots ,a_{n-1})\Big]b_n+(a_1,\dots ,b_{n-1})
\\
&=(a_1,\dots ,b_{n-1})(a_n,b_n)+ (a_1,\dots ,a_{n-1})b_n
\\
&=(a_1,\dots ,b_{n-1})(a_n,b_n)+\sum^n_{\ell=3}\Big[(a_1,\dots,b_{\ell-2})a_{\ell-1}b_n\Big]+a_1b_n\, .\qedhere
\end{align*} 
\end{proof}

\begin{lem}\label{lem:curious-identity}
For $n\geq 1$, we have
 $$(a_1,b_1,\ldots,b_{n-1},a_{n})(b_{n},a_n,\ldots,b_1,a_1)=(a_1,b_1,\ldots,a_n,b_{n})(a_{n},b_{n-1},\ldots,b_1,a_1)\,.$$ 
\end{lem}
\begin{proof}
 For $n=1$, we can write that
\begin{equation*}
 (a_1)(b_1,a_1)=a_1(e_1+b_1a_1)=(e_2+a_1b_1)a_1=(a_1,b_1)(a_1)\,.
\end{equation*}
Now, using the induction hypothesis, and the recursive relations \eqref{eq:def-Euler-continuant-very-beginning} and \eqref{eq:def-Euler-continuant-intro-2} of the Euler continuants, we have
  \begin{align*}
(a_1,\ldots,&a_{n})(b_{n},\ldots,a_1) \\
&=(a_1,\ldots,b_{n-1})a_{n}(b_{n},\ldots,a_1)+(a_1,\ldots,a_{n-1})(b_{n},\ldots,a_1) \\
&=(a_1,\ldots,b_{n-1})a_{n}(b_{n},\ldots,a_1)+(a_1,\ldots,a_{n-1})\Big[b_{n}(a_{n},\ldots,a_1)+(b_{n-1},\ldots,a_1)\Big] \\
&=(a_1,\ldots,b_{n-1})\Big[a_{n}(b_{n},\ldots,a_1)+(a_{n-1},\ldots,a_1) \Big]+(a_1,\ldots,a_{n-1})b_{n}(a_{n},\ldots,a_1) \\
&=(a_1,\ldots,b_{n-1})(a_{n},b_{n},\ldots,a_1) + (a_1,\ldots,a_{n-1})b_{n}(a_{n},\ldots,a_1) \\
&=(a_1,\ldots,b_{n})(a_{n},\ldots,a_1)\,. \qedhere
  \end{align*} 
\end{proof}

\subsection{The result} \label{ss:Result}

We consider the $(2n)$-th Euler continuants
\begin{equation}\label{Eq:invEuler}
\begin{aligned}
 (a_1,b_1,a_2,b_2,\dots,a_n,b_n)&\in e_2(\kk\overline{\Gamma}_n)e_2\, , 
 \\
 (b_n,a_n,b_{n-1},a_{n-1},\dots,b_1,a_1)&\in e_1(\kk\overline{\Gamma}_n)e_1\, ,
 \end{aligned}
\end{equation}
and we form the \emph{Boalch algebra} $\mathcal{B}(\Gamma_n)$ as the path algebra $\kk\overline{\Gamma}_n$ with $(a_1,\dots,b_n)$ and $ (b_n,\dots,a_1)$ inverted. 
As in \cite[\S6.5]{VdB1}, this means that we introduce elements $A,B$ such that $A=Ae_2=e_2A$, $B=Be_1=e_1B$, and $A(a_1,\dots,b_n)=(a_1,\dots,b_n)A=e_2$ and $B(b_n,\dots,a_1)=(b_n,\dots,a_1)B=e_1$. From now on, we use the notation $(a_1,\dots ,b_n)^{-1}$ and $(b_n,\dots,a_1)^{-1}$ for $A$ and $B$, respectively. 
This entails that $e_1+(a_1,\dots ,b_n)$ and $e_2+(b_n,\dots,a_1)$ are invertible in  $\mathcal{B}(\Gamma_n)$ in the usual sense. 

Next, we follow Boalch \cite[Remark 17]{B18} and define the \emph{fission algebra} 
\begin{equation} \label{Eq:Fission-algebra}
 \mathcal{F}^q(\Gamma_n):=\mathcal{B}(\Gamma_n)/R^q\,, \quad q=(q_1,q_2)\in (\kk^\times)^2\,,
\end{equation}
where $R^q$ is the ideal generated by the two relations 
\begin{equation}
(b_n,\dots,a_1)^{-1}=q_1 e_1\,, \quad 
(a_1,\dots,b_n)=q_2e_2\,.
\end{equation}
 Note that $\mathcal{B}(\Gamma_n)$ is a generalisation of $\mathcal{B}(\Gamma_1)=\Ac(\Gamma_1)$ introduced in \ref{sec:CBS-multiplicative}, while 
 $\mathcal{F}^q(\Gamma_n)$ is a generalisation of the multiplicative preprojective algebra $\mathcal{F}^q(\Gamma_1)=\Lambda^q(\Gamma_1)$ of \cite{CBShaw}.
 
 Finally, let $B=\kk e_1\oplus \kk e_2$, and we introduce a $B$-linear double bracket $\lr{-,-}$ on the arrows of $\overline{\Gamma}_n$ (seen as generators of $\mathcal{B}(\Gamma_n)$) as
\begin{subequations}
\label{eq:Euler-double-bracket}
\begin{align}
 \lr{a_i,a_j}&=
\begin{cases} 
-\frac{1}{2}\big(a_i\otimes a_j+a_j\otimes a_i\big)\,, &\mbox{if } i<j
 \\ 
  \;\;\;  0\, , &\mbox{if } i=j
 \\
 \;\;\;   \frac{1}{2}\big(a_i\otimes a_j+a_j\otimes a_i\big), & \mbox{if } i>j
\end{cases}\,;  \label{eq:Euler-double-bracket.a}
\\
 \lr{b_i,b_j}&=
\begin{cases} 
-\frac{1}{2}\big(b_i\otimes b_j+b_j\otimes b_i\big)\,, &\mbox{if } i<j
 \\ 
  \;\;\;  0\,, &\mbox{if } i=j
 \\
 \;\;\;  \frac{1}{2}\big(b_i\otimes b_j+b_j\otimes b_i\big)\,, & \mbox{if } i>j
\end{cases}\,;  \label{eq:Euler-double-bracket.b}
\\
 \lr{a_i,b_j}&=
\begin{cases} 
 \;\;\;  \frac{1}{2}\big(e_1\otimes a_ib_j+b_ja_i\otimes e_2\big)\,, &\mbox{if } i<j
 \\ 
 \;\;\;  \frac{1}{2}\big(b_ia_i\otimes e_2+e_1\otimes a_ib_i\big)+e_1\otimes e_2\,, &\mbox{if } i=j
 \\
- \frac{1}{2}\big(b_ja_{i}\otimes e_2+e_1\otimes a_{i}b_j\big)-\delta_{i-j,1}(e_1\otimes e_2)\,, &\mbox{if } i>j
\end{cases}\,;   \label{eq:Euler-double-bracket.c}
\\
 \lr{b_i,a_j}&=
\begin{cases} 
   \;\;\;   \frac{1}{2}\big(e_2\otimes b_ia_j+ a_jb_i\otimes e_1\big)+\delta_{j-i,1}(e_2\otimes e_1)\, , &\mbox{if } i<j
 \\
 - \frac{1}{2}\big(e_2\otimes b_ia_i+ a_ib_i\otimes e_1\big)-e_2\otimes e_1\,, &\mbox{if } i=j
\\
-\frac{1}{2}\big(a_jb_i\otimes e_1+ e_2\otimes b_ia_j)\,, &\mbox{if } i>j
\end{cases}\,;    \label{eq:Euler-double-bracket.d}
\end{align}
\end{subequations}
where $i,j\in\{1,\dots,n\}$, and we extend $\lr{-,-}$ to $\mathcal{B}(\Gamma_n)$ by the Leibniz rules \eqref{Eq:outder} and \eqref{Eq:inder}.

It is worth noting how the double bracket can be determined on the inverses of the elements \eqref{Eq:invEuler}, so that it is is uniquely defined on $\mathcal{B}(\Gamma_n)$ and not only on $\kk\overline{\Gamma}_n$. By $B$-linearity $\lr{c,(b_n,\ldots,a_1)(b_n,\ldots,a_1)^{-1}}=\lr{c,e_1}=0$ for any $c\in \mathcal{B}(\Gamma_n)$, hence by \eqref{Eq:outder} 
$$\lr{c,(b_n,\ldots,a_1)^{-1}}=-(b_n,\ldots,a_1)^{-1} \lr{c,(b_n,\ldots,a_1)} (b_n,\ldots,a_1)^{-1}\,.$$
A similar relation holds for $\lr{c,(a_1,\ldots,b_n)^{-1}}$.

\begin{thm} \label{Thm:MAIN}
Given an integer $n\geq 1$, let $\overline{\Gamma}_n$ be the quiver with vertices $\{1,2\}$ and arrows $\{a_i,b_i\mid 1\leq i\leq n\}$ such that $a_i=e_2a_ie_1$ and $b_i=e_1b_ie_2$. Let $B:=\kk e_1\oplus\kk e_2$, and $\kk\overline{\Gamma}_n$ be the path algebra of $\overline{\Gamma}_n$ over $B$. Finally, if $(a_1,\dots,b_n)$ and $ (b_n,\dots,a_1)$ are $(2n)$-th Euler continuants, let $\mathcal{B}(\Gamma_n)$ denote $\kk\overline{\Gamma}_n$ with $(a_1,\dots,b_n)$ and $ (b_n,\dots,a_1)$ inverted.

Let $\lr{-,-}$ be the $B$-linear double bracket defined on $\mathcal{B}(\Gamma_n)$ by \eqref{eq:Euler-double-bracket}. 
In addition, we introduce the following element of  $\mathcal{B}(\Gamma_n)$
\begin{equation}
\Phi:= (a_1,\dots,b_n) + (b_n,\dots,a_1)^{-1}.
\label{eq:mult-moment-map-statement}
\end{equation}
Then the triple $\big (\mathcal{B}(\Gamma_n),\lr{-,-},\Phi\big)$ is a Hamiltonian double quasi-Poisson algebra.
\label{thm:Euler-quasi-Hamiltonian-algebra}
\end{thm}

We postpone the proof of Theorem \ref{thm:Euler-quasi-Hamiltonian-algebra} to Section \ref{sec:proff-thm-Euler}.
We remark that, in the case $n=1$ where $\mathcal{B}(\Gamma_1)=\Ac(\Gamma_1)$, Theorem \ref{thm:Euler-quasi-Hamiltonian-algebra} recovers Van den Bergh's double quasi-Poisson bracket described in Theorem \ref{tm:VdB-mult-preproj-quasi-Hamilt}(i).
Furthermore, using Proposition \ref{prop:double-quasi-Ham-red}, we can obtain the following generalisation of Theorem \ref{tm:VdB-mult-preproj-quasi-Hamilt}(ii).

\begin{cor} \label{cor:fission-H0}
The fission algebra $\mathcal{F}^q(\Gamma_n)$ attached to $\Gamma_n$ carries an $H_0$-Poisson structure.
\end{cor}

Using the Kontsevich--Rosenberg principle,  
Theorem \ref{thm:Euler-quasi-Hamiltonian-algebra} turns the representation space $\Rep\big(\mathcal{B}(\Gamma_n),(d_1,d_2)\big)$ into a Hamiltonian quasi-Poisson space for the canonical action of $H:=\Gl(d_1)\times \Gl(d_2)$. Similarly, Corollary \ref{cor:fission-H0} induces a Poisson structure on the reduced spaces $\Rep\big(\mathcal{F}^q(\Gamma_n),(d_1,d_2)\big)/\!/H$, which coincides with the structure obtained by quasi-Hamiltonian reduction; we refer to \cite[\S2.3]{FF} for a review of this general construction based on \cite{CB11,VdB1}.

As pointed out in the introduction, if $\kk=\CC$ and $(d_1,d_2)=(1,1)$, we are in the situation of the Sibuya spaces \cite{Si75} whose symplectic geometry was studied by Boalch in \cite{B18}. 
In \ref{ss:FN-Poisson}, we give the precise form of the bracket obtained from \eqref{eq:Euler-double-bracket} in this specific case, and we explain its connection to the Flaschka--Newell Poisson bracket \cite{FN}. 
More generally, if $\kk=\CC$ and the dimension vector $(d_1,d_2)$ is arbitrary, we get that the Poisson structure induced on any reduced space $\Rep\big(\mathcal{F}^q(\Gamma_n),(d_1,d_2)\big)/\!/H$ corresponds (on its smooth locus) to the symplectic structure defined by Boalch. This follows from Proposition \ref{Pr:RepCorrespond} presented in \ref{ss:qbisymp}, where we prove that the Hamiltonian quasi-Poisson structure that we obtain on the representation space $\Rep\big(\mathcal{B}(\Gamma_n),(d_1,d_2)\big)$ corresponds to the quasi-Hamiltonian structure introduced by Boalch \cite{Bo-Duke,Bo14} and conveniently studied by Paluba \cite[Ch.~5]{Paluba}.

\subsection{The noncommutative bivector} 
\label{sec:nc-bivector}
 
In the case of Theorem \ref{thm:Euler-quasi-Hamiltonian-algebra}, it is possible to write explicitly the bivector defining the double bracket on the Boalch algebra $\mathcal{B}(\Gamma_n)$ through Proposition \ref{prop:VdB-mu}. As in \ref{ssec:Quivers}, we consider the double derivations 
\begin{equation}\label{Eq:dder-ab}
 \frac{\partial}{\partial a_i}\,,\,\quad  \frac{\partial}{\partial b_i} \,, \quad i=1,\ldots,n\,,
\end{equation}
which are defined for $i,k=1,\ldots,n$  by 
\begin{equation} 
 \frac{\partial a_k}{\partial a_i}=\delta_{ik} \, e_2\otimes e_1\,,\quad \frac{\partial b_k}{\partial a_i}=0\,, \quad 
 \frac{\partial b_k}{\partial b_i}=\delta_{ik}\, e_1\otimes e_2 \,,\,\,\,\frac{\partial a_k}{\partial b_i}=0 \,.
\end{equation}
Then the following result is a direct application of the equivalence given by Proposition \ref{prop:VdB-mu}, since path algebras of quivers are smooth. 
\begin{prop}\label{prop:the-nc-bivector}
 The double quasi-Poisson bracket from Theorem \ref{thm:Euler-quasi-Hamiltonian-algebra} is associated with the following bivector: 
 \begin{equation}
  \begin{aligned} \label{Eq:bivector-P}
 \PPn&=\frac12 \sum_{i>j} \left[ \frac{\partial}{\partial a_i}a_i  \frac{\partial}{\partial a_j} a_j +  \frac{\partial}{\partial a_i} a_j  \frac{\partial}{\partial a_j} a_i \right]   
+\frac12 \sum_{i>j} \left[ \frac{\partial}{\partial b_i}b_i  \frac{\partial}{\partial b_j} b_j +  \frac{\partial}{\partial b_i} b_j  \frac{\partial}{\partial b_j} b_i \right]
 \\
&\quad+\frac12 \sum_{i\neq j} \operatorname{sgn}(j-i)  \left[ \frac{\partial}{\partial a_i}a_i b_j  \frac{\partial}{\partial b_j} + a_i \frac{\partial}{\partial a_i}   \frac{\partial}{\partial b_j} b_j \right] 
-\sum_{i=2}^n \frac{\partial}{\partial a_i}   \frac{\partial}{\partial b_{i-1}} 
\\
&\quad +\frac12 \sum_{i=1}^n\left[ \frac{\partial}{\partial a_i}a_i b_i \frac{\partial}{\partial b_i}
+ a_i\frac{\partial}{\partial a_i} \frac{\partial}{\partial b_i} b_i \right] 
+\sum_{i=1}^n \frac{\partial}{\partial a_i}  \frac{\partial}{\partial b_{i}}\,.
 \end{aligned}
 \end{equation}
\end{prop}
\begin{proof}
 It suffices to show that the double bracket $\lr{-,-}_\PPn$ defined through \eqref{Eq:bivector-bracket} using $\PPn$  gives \eqref{eq:Euler-double-bracket.a}--\eqref{eq:Euler-double-bracket.c} when evaluated on generators. If $i>j$, we note that the only terms contributing to $\lr{a_i,a_j}_\PPn$ come from the first term of  \eqref{Eq:bivector-P} as 
\begin{equation*}
\begin{aligned}
 \frac12 \left(\frac{\partial}{\partial a_i}a_i\right) \left( \frac{\partial}{\partial a_j} a_j\right) 
 + \frac12 \left(\frac{\partial}{\partial a_i} a_j\right)  \left(\frac{\partial}{\partial a_j} a_i \right)\,,
\end{aligned}
\end{equation*}
so that \eqref{Eq:bivector-bracket} yields 
\begin{equation}
\begin{aligned}
\lr{a_i,a_j}_\PPn= \frac12\big(a_j e_1 \otimes a_i e_1 - 0\big) + \frac12\big(a_i e_1 \otimes a_j e_1 - 0\big) = \frac12\Big(a_j \otimes a_i + a_i \otimes a_j\Big) \,.
\end{aligned}
\end{equation}
Clearly, $\lr{a_i,a_i}_\PPn=0$, so by cyclic antisymmetry \eqref{Eq:cycanti} we get that $\lr{a_i,a_j}_\PPn$ coincides with \eqref{eq:Euler-double-bracket.a} for all $1\leq i,j \leq n$. We can show in the same way that  $\lr{b_i,b_j}_\PPn$ coincides with \eqref{eq:Euler-double-bracket.b} for all indices. 

We now compute $\lr{a_i,b_j}_\PPn$. If $i<j$, the only terms that contribute come from the third term of \eqref{Eq:bivector-P} as 
\begin{equation*}
\begin{aligned}
\frac12 \left(\frac{\partial}{\partial a_i}a_i\right) \left(b_j \frac{\partial}{\partial b_j}\right) 
+ \frac12  \left(a_i \frac{\partial}{\partial a_i} \right)  \left( \frac{\partial}{\partial b_j} b_j  \right)\,,
\end{aligned}
\end{equation*}
and we get from \eqref{Eq:bivector-bracket} that 
\begin{equation}
\begin{aligned}
\lr{a_i,b_j}_\PPn= \frac12  e_1 e_1 \otimes a_i b_j  + \frac12 b_j a_i \otimes  e_2 e_2   = \frac12(e_1 \otimes a_i b_j + b_j a_i \otimes e_2) \,.
\end{aligned}
\end{equation}
If $i=j$, we only need the following terms from the last line in  \eqref{Eq:bivector-P}
\begin{equation*}
\begin{aligned}
\frac12 \left(\frac{\partial}{\partial a_i}a_i\right) \left(b_i \frac{\partial}{\partial b_i}\right) 
+ \frac12  \left(a_i \frac{\partial}{\partial a_i} \right)  \left( \frac{\partial}{\partial b_i} b_i  \right)
+ \left( \frac{\partial}{\partial a_i} \right)  \left( \frac{\partial}{\partial b_i}   \right)\,,
\end{aligned}
\end{equation*}
and \eqref{Eq:bivector-bracket} yields 
\begin{equation}
\begin{aligned}
\lr{a_i,b_i}_\PPn= \frac12\Big(e_1 \otimes a_i b_i + b_i a_i \otimes e_2\Big) + e_1 \otimes e_2 \,.
\end{aligned}
\end{equation}
In the case $i>j$, the contributing terms come from  the third and fourth sums in \eqref{Eq:bivector-P} as 
\begin{equation*}
\begin{aligned}
-\frac12 \left(\frac{\partial}{\partial a_i}a_i\right) \left(b_j \frac{\partial}{\partial b_j}\right) 
- \frac12  \left(a_i \frac{\partial}{\partial a_i} \right)  \left( \frac{\partial}{\partial b_j} b_j  \right)
-\delta_{i-j,1} \left( \frac{\partial}{\partial a_i} \right)  \left( \frac{\partial}{\partial b_j}   \right)\,,
\end{aligned}
\end{equation*}
and \eqref{Eq:bivector-bracket} yields 
\begin{equation}
\begin{aligned}
\lr{a_i,b_j}_\PPn= -\frac12\Big(e_1 \otimes a_i b_j + b_j a_i \otimes e_2\Big) -\delta_{i-j,1} e_1 \otimes e_2 \,.
\end{aligned}
\end{equation}
Gathering the three cases, we get that $\lr{a_i,b_j}_\PPn$ coincides with \eqref{eq:Euler-double-bracket.c} for all indices.
\end{proof}

\begin{rem} \label{rem:relation-bivector}
In the case $n=1$, the noncommutative bivector $\PP_1$ can directly be compared with the one of Van den Bergh \cite[Theorem 6.5.1]{VdB1}. Setting $a:=a_1$, $b:=b_1$, we have that modulo graded-commutator $\PP_1$ equals 
\begin{equation}
 \PP_{\operatorname{VdB}}:=\frac12(1+ba)\frac{\partial}{\partial a}\frac{\partial}{\partial b}
 -\frac12 (1+ab)\frac{\partial}{\partial b} \frac{\partial}{\partial a}\,,
\end{equation}
which was constructed by Van den Bergh. It is an easy exercise (see \cite[\S4.1]{VdB1} in full generalities) to check that the double bracket induced by an element $Q\in (D_BA)_2$ by  Proposition \ref{prop:VdB-mu} only depends on $Q$ modulo graded commutators. Thus  $\PP_1$ and $\PP_{\operatorname{VdB}}$ induce the same double quasi-Poisson bracket on $\Bc(\Gamma_1)= \Ac(\Gamma_1)$. 
\end{rem}

\subsection{The Flaschka--Newell Poisson bracket} 
\label{ss:FN-Poisson}

The representation space   of dimension $(1,1)$, denoted $\Rep\big(\Bc(\Gamma_n),(1,1)\big)$, is parametrised by assigning $A_i,B_i\in \CC$ to  the generators $a_i,b_i$ for each $ 1\leq i\leq n$, where we require the invertibility of the (scalar) matrices representing the two elements \eqref{Eq:invEuler}. 
Denoting by $A_i,B_i$ the corresponding evaluation functions on $\Rep\big(\Bc(\Gamma_n),(1,1)\big)$, 
this gives precisely the parametrisation \eqref{eq:Boalch-space-5} of the space $\widehat{\Bc}^{n+1}$ from the introduction. It is endowed with the action of $\CC^\times \times \CC^\times$ through 
\begin{equation*}
 t\cdot \{A_i,B_i\mid 1\leq i\leq n\} = \{t_2A_i t_1^{-1},t_1 B_i t_2^{-1}\mid 1\leq i\leq n\}\,, \quad t=(t_1,t_2)\in \CC^\times \times \CC^\times \,.
\end{equation*}
This is a Hamiltonian quasi-Poisson space using Theorem \ref{thm:Euler-quasi-Hamiltonian-algebra} and \cite[Proposition 7.13.2]{VdB1}, in application of the Kontsevich--Rosenberg principle. Using \cite[Proposition 7.5.1]{VdB1}, the quasi-Poisson bracket is given by\footnote{To get the formulas, it suffices in \eqref{eq:Euler-double-bracket} to replace all elements $a_i,b_j$ by $A_i,B_j$ while $e_1,e_2$ are replaced by $1$, and then we multiply the two components in the tensor product hence obtained.} 
\begin{subequations}
 \begin{align}
\br{A_i,A_j}=&\sgn(i-j) A_iA_j\,, \quad \br{B_i,B_j}=\sgn(i-j)B_iB_j\,, \label{Eq:bracket-rep11-AABB}\\
 \br{A_i,B_j}=&\left\{
 \begin{array}{cc}
  A_iB_j \,, &\mbox{if }i<j\\
  A_iB_i+1\,, &\mbox{if }i=j\\
  -A_iB_j-\delta_{i-j,1}\,, &\mbox{if }i>j
 \end{array}
 \right.\,.  \label{Eq:bracket-rep11-mix}
 \end{align}
\end{subequations}
The corresponding moment map $\widehat{\mu}$ is given by \eqref{eq:Boalch-space-6} and, noting that the variables $A_i,B_i$ commute, it can be simply written as  
\begin{equation} \label{Eq:momap-rep11}
 \widehat{\mu}=(\lambda^{-1},\lambda)\,, \quad \lambda:=(A_1,B_1,\ldots,A_n,B_n)=(B_n,A_n,\ldots,B_1,A_1)\,.
\end{equation}
Since $\lambda$ represents the component $\Phi_2$ of the (noncommutative) moment map from Theorem \ref{thm:Euler-quasi-Hamiltonian-algebra}, we have from \cite[Proposition 7.5.1]{VdB1}, \eqref{eq:moment-map-b} and \eqref{eq:moment-map-a} that 
\begin{equation} \label{Eq:bracket-rep11-lamb}
 \br{\lambda,A_i}=-\lambda A_i\,, \quad \br{\lambda,B_i}=\lambda B_i\,.
\end{equation}

Let us note that since the group $\CC^\times \times \CC^\times$ acting on $\widehat{\Bc}^{n+1}$ is abelian, the quasi-Poisson bracket is in fact a \emph{Poisson bracket}. This will be important in view of the following change of variables, motivated by a standard parametrisation  connected to Stokes matrices \cite{BT,B18}. We set 
\begin{equation}
 A_i=s_{2n+3-2i}\,, \quad B_i=s_{2n+2-2i}\,,
\end{equation}
for $s_j\in \CC$ with $2\leq j\leq 2n+1$. (The parameters $s_k$ with odd/even indices are the $A_i$/$B_i$ respectively, and they appear with decreasing indices.) We also denote by $s_j$ the corresponding evaluation function on $\widehat{\Bc}^{n+1}$. Then, using this change of coordinates, the moment map $\lambda$ in \eqref{Eq:momap-rep11} becomes $\lambda=(s_2,\ldots,s_{2n+1})$.
It is an easy exercise using \eqref{Eq:bracket-rep11-AABB} to see that 
we can write 
\begin{equation*}
 \br{s_k,s_l}=s_k s_l \quad \text{ for }k<l \text{ both odd or both even}\,.
\end{equation*}
We can also note that \eqref{Eq:bracket-rep11-mix}  yields
\begin{equation*}
\br{s_k,s_l}=\left\{ 
\begin{array}{ll}
s_ks_l \,, &\mbox{if } k>l\\
s_ks_l+1\,, &\mbox{if }k=l+1\\
-s_ks_l-1\,, &\mbox{if } k=l-1\\
-s_ks_l\,, &\mbox{if }k<l-1
\end{array}
\right. \quad  \text{ for   $k$ odd and $l$ even}\,.
\end{equation*}
This is equivalent to 
\begin{equation*}
\br{s_k,s_l}=\left\{ 
\begin{array}{ll}
-s_ks_l\,,     &\mbox{if }k<l\\
-s_ks_l-1\,, &\mbox{if }k=l-1\\
s_ks_l+1\,, &\mbox{if }k=l+1\\
s_ks_l\,, &\mbox{if }k>l+1
\end{array}
\right. \quad  \text{ for   $k$ even and $l$ odd}\,.
\end{equation*}
Gathering these identities with the moment map identity \eqref{Eq:bracket-rep11-lamb}, we get that the Poisson bracket is completely determined by 
\begin{subequations}
 \begin{align} 
 \br{s_k,s_l}&=-\delta_{k,l-1}+(-1)^{k-l}s_ks_l\,, \quad \text{ for }1<k<l<2n+2\,, \label{Eq:bracket-svar} \\
 \br{s_k,\lambda}&=(-1)^{k+1}s_k\lambda\,.
\end{align}
\end{subequations}

Using \eqref{eq:Boalch-space-3} and the subsequent Gauss decomposition, we note that we can equivalently parametrise 
$\widehat{\Bc}^{n+1}$ in terms of the elements $\{s_j \mid 2\leq j\leq 2n+1\}$, such that 
\begin{equation}
\begin{aligned}
  &\left(\begin{array}{cc} 1&0\\s_2&1 \end{array} \right) \left(\begin{array}{cc} 1&s_3\\0&1 \end{array} \right) \ldots 
\left(\begin{array}{cc} 1&0\\s_{2n}&1 \end{array} \right) \left(\begin{array}{cc} 1&s_{2n+1}\\0&1 \end{array} \right) \\
=& \left(\begin{array}{cc} 1&(s_3,\ldots,s_{2n+1})\lambda^{-1}\\0&1 \end{array} \right)
\left(\begin{array}{cc} \lambda^{-1}&0\\0&\lambda \end{array} \right)
\left(\begin{array}{cc} 1&0\\ \lambda^{-1} (s_2,\ldots,s_{2n})&1 \end{array} \right)\,.  \label{eq:B-hat-equiv-param}
\end{aligned}
\end{equation}
Consequently, $\widehat{\Bc}^{n+1}$ can be written as the following subspace of $\CC^{2n+2}\times \CC^\times$: 
\begin{equation}
 \left\{ 
  \left(\begin{array}{cc} 1&s_1\\0&1 \end{array} \right)\left(\begin{array}{cc} 1&0\\s_2&1 \end{array} \right) \ldots 
 \left(\begin{array}{cc} 1&s_{2n+1}\\0&1 \end{array} \right)  \left(\begin{array}{cc} 1&0\\s_{2n+2}&1 \end{array} \right)
  \left(\begin{array}{cc} \lambda&0\\0&\lambda^{-1} \end{array} \right) = \Id_2
\right\}\,.
\label{eq:B-hat-s1-s2n+1}
\end{equation}
Indeed, we recover equation \eqref{eq:B-hat-equiv-param} from \eqref{eq:B-hat-s1-s2n+1} because the Gauss decomposition yields 
\begin{equation}
 s_1=- (s_3,\ldots,s_{2n+1})\lambda^{-1}\,, \quad s_{2n+2}=-\lambda^{-1} (s_2,\ldots,s_{2n})\,.
\end{equation}
Using the description of $\widehat{\Bc}^{n+1}$ given in \eqref{eq:B-hat-s1-s2n+1}, Bertola and Tarricone \cite{BT} have recently written the following Poisson bracket 
\begin{subequations}
 \begin{align}
 \br{s_k,s_l}_{FN}=&\delta_{k,l-1}-\frac{\delta_{k,1}\delta_{l,2n+2}}{\lambda^2}+(-1)^{k-l+1}s_ks_l\,, \quad \text{ for }1\leq k<l\leq 2n+2\,, \label{Eq:bracket-FN} \\
 \br{s_k,\lambda}_{FN}=&(-1)^{k}s_k\lambda\,, \label{Eq:bracket-FN-momap}
\end{align}
\end{subequations}
which was originally introduced by Flaschka and Newell \cite{FN}. By comparing the expressions \eqref{Eq:bracket-svar} and \eqref{Eq:bracket-FN} on the generators $\{s_j \mid 2\leq j\leq 2n+1\}$, we directly see that the double quasi-Poisson bracket from Theorem \ref{thm:Euler-quasi-Hamiltonian-algebra} induces the Flaschka--Newell Poisson bracket on the representation space $\Rep\big(\Bc(\Gamma_n),(1,1)\big)$ (up to an irrelevant factor). In particular, \eqref{Eq:bracket-FN-momap} is nothing else than the moment map condition.


\section{Proof of Theorem \ref{thm:Euler-quasi-Hamiltonian-algebra}}
\label{sec:proff-thm-Euler}
\allowdisplaybreaks
 
In order to prove Theorem  \ref{thm:Euler-quasi-Hamiltonian-algebra}, we will show in \ref{sec:proof-Euler-quasi-Poisson} that the double bracket defined by  \eqref{eq:Euler-double-bracket} is quasi-Poisson, i.e. it satisfies \eqref{qPabc}. 
We then prove in \ref{sec:proof-Euler-multiplicative-moment-map} that the components of the moment map $\Phi$ given in \eqref{eq:mult-moment-map-statement} satisfy \eqref{Phim}, which finishes the proof.  
Before tackling these two steps, we will derive some preliminary identities in \ref{sec:preparation-proof-main-theorem}.

\subsection{Preparation for the proof}
\label{sec:preparation-proof-main-theorem}
\allowdisplaybreaks

In this subsection we shall prove two results that we will extensively use in the proof of Theorem \ref{thm:Euler-quasi-Hamiltonian-algebra} below.  
We recall $(a_n,b_n)=e_2+a_nb_n$, and $\ast$ denotes the inner bimodule multiplication from \ref{sec:ssec-double-derivations}. 

\begin{lem}
Consider the double bracket defined in \eqref{eq:Euler-double-bracket}. 
Then the following identities hold:
\begin{enumerate}
\item [\textup{(i)}]
$\lr{(a_n,b_n),a_i}\!=
\begin{cases}
\;\;\,\frac{1}{2}\Big((a_n,b_n)\otimes a_i-e_2\otimes (a_n,b_n)a_i\Big), &\mbox{if } 1\leq i\leq n-1
\\
-\frac{1}{2}\Big((a_n,b_n)\otimes a_n+e_2\otimes (a_n,b_n)a_n\Big), &\mbox{if } i=n
\end{cases};
$
\item [\textup{(ii)}]
$
\lr{(a_n,b_n),b_i}\!=
\begin{cases}
\frac{1}{2}\Big(b_i\otimes (a_n,b_n)-b_i(a_n,b_n)\otimes e_2 \Big), &\!\!\!\mbox{if } 1\leq i\leq n-2
\\
\frac{1}{2}\Big(b_{n-1}\otimes (a_n,b_n)-b_{n-1} (a_n,b_n)\otimes e_2\Big)\!-\!b_n\!\otimes \!e_2, &\!\!\!\mbox{if } i=n-1
\\
\frac{1}{2}\Big(b_n\otimes (a_n,b_n)+b_n(a_n,b_n)\otimes e_2\Big), &\!\!\!\mbox{if } i=n
\end{cases}.
$
\end{enumerate}
\label{lem:technic-lemma}
\end{lem}

\begin{proof}
To prove (i), if $1\leq i\leq n-1$, by \eqref{Eq:inder}, \eqref{eq:Euler-double-bracket.a} and \eqref{eq:Euler-double-bracket.d}, we obtain
\begin{small}
\begin{align*}
\lr{(a_n,b_n),a_i}&=a_n*\lr{b_n,a_i}+\lr{a_n,a_i}*b_n
\\
&=-a_n*\Big(\frac{1}{2}\big(a_ib_n\otimes e_1+e_2\otimes b_na_i\big)\Big)+\Big(\frac{1}{2}\big(a_n\otimes a_i+a_i\otimes a_n\big)\Big)*b_n
\\
&=\frac{1}{2}a_nb_n\otimes a_i-\frac{1}{2}e_2\otimes a_nb_na_i
=\frac{1}{2}\Big((a_n,b_n)\otimes a_i-e_2\otimes (a_n,b_n)a_i\Big)\,,
\end{align*}
\end{small}%
where we used $(a_n,b_n)=a_nb_n+e_2$.
Similarly, in (ii), the reader can prove the identity $2\lr{(a_n,b_n),b_i}=b_i\otimes (a_n,b_n)-b_i(a_n,b_n)\otimes e_2  $, for $1\leq i\leq n-2$.

Now, to prove the second formula in (i), since $\lr{a_n,a_n}=0$, by \eqref{eq:Euler-double-bracket.d} we have
\begin{small}
\begin{align*}
\lr{(a_n,b_n),a_n}&=a_n*\lr{b_n,a_n}
\\
&=-a_n*\Big(\frac{1}{2}\big(e_2\otimes b_na_n+a_nb_n\otimes e_2\big)+e_2\otimes e_1\Big)
\\
&=-\frac{1}{2}\Big( (a_n,b_n)\otimes a_n+e_2\otimes (a_n,b_n)a_n\Big)\, ,
\end{align*}
\end{small}%
where we used again that $(a_n,b_n)=a_nb_n+e_2$. Similarly, the reader can check that $2\lr{(a_n,b_n),b_n}=b_n\otimes (a_n,b_n)+b_n(a_n,b_n)\otimes e_2$ holds in (ii).

Finally, we prove the remaining identity in (ii). 
By \eqref{eq:Euler-double-bracket.b} and \eqref{eq:Euler-double-bracket.c}, 
\begin{small}
\begin{align*}
\lr{(&a_n,b_n),b_{n-1}}=a_n*\lr{b_n,b_{n-1}}+\lr{a_n,b_{n-1}}*b_n
\\
&=a_n*\Big(\frac{1}{2}\big(b_n\otimes b_{n-1}+b_{n-1}\otimes b_n\big)\Big)
-\Big(\frac{1}{2}\big(b_{n-1}a_n\otimes e_2+e_1\otimes a_nb_{n-1}\big)+e_1\otimes e_2\Big)*b_n
\\
&=\frac{1}{2}b_{n-1}\otimes a_nb_n-\frac{1}{2}b_{n-1}a_nb_n\otimes e_2-b_n\otimes e_2
\\
&=\frac{1}{2}\Big(b_{n-1}\otimes (a_n,b_n)-b_{n-1}(a_n,b_n)\otimes e_2\Big)-b_n\otimes e_2\,.\qedhere
\end{align*}
\end{small}
\end{proof}

\begin{lem}
Consider the double bracket defined in \eqref{eq:Euler-double-bracket}. 
Then the following identities hold:
\begin{enumerate}
\item [\textup{(i)}]
$2\lr{(a_1,\dots, b_i),b_j}=b_j(a_1,\dots, b_i)\otimes e_2-b_j\otimes (a_1,\dots, b_i)$, with $1\leq i<j\leq n$;
\item [\textup{(ii)}]
$2\lr{(a_1,\dots, b_i),a_j}=e_2\otimes (a_1,\dots, b_i)a_j- (a_1,\dots, b_i)\otimes a_j$, with $j\geq  i+2$;
\item [\textup{(iii)}]
$\lr{(a_1,\dots, b_i),a_{i+1}}=\frac{1}{2}\Big(e_2\otimes (a_1,\dots, b_i)a_{i+1}- (a_1,\dots, b_i)\otimes a_{i+1}\Big)+e_2\otimes (a_1,\dots, a_i)$, with $1\leq i\leq n-1$.
\end{enumerate}
\label{lem:technic-lemma-no-inductive}
\end{lem}

\begin{proof}
We will prove these formulas by strong induction.
To prove (i), we start noting that by \eqref{eq:Euler-double-bracket.b} and \eqref{eq:Euler-double-bracket.c},
\begin{small}
\begin{equation}
\begin{aligned}
\lr{(a_i,b_i),b_j}&=a_i*\lr{b_i,b_j}+\lr{a_i,b_j}*b_i
\\
&=\frac{1}{2}b_ja_ib_i\otimes e_2-\frac{1}{2}b_j\otimes a_ib_i
=\frac{1}{2}\Big(b_j(a_i,b_i)\otimes e_2-b_j\otimes (a_i,b_i)\Big)\, .
\end{aligned}
\label{eq:prop-technic-euler-i}
\end{equation}
\end{small}%
Then the base case arises when $i=1$. 
Now, if $i\geq 2$, by \eqref{eq:Convenient-formula-Euler-cont-b}, we have
\begin{small}
\begin{align*}
\lr{(a_1,\dots, b_i),&\,b_j}=\dgal{\Big((a_1,\dots ,b_{i-1})(a_i,b_i)+\sum^i_{\ell=3}\Big[(a_1,\dots,b_{\ell-2})a_{\ell-1}b_i\Big]+a_1b_i\Big),b_j}
\\
&=(a_1,\dots,b_{i-1})*\lr{(a_i,b_i),b_j}+\lr{(a_1,\dots,b_{i-1}),b_j}*(a_i,b_i)
\\
&\quad +\sum^i_{\ell=3}\Big[(a_1,\dots,b_{\ell-2})a_{\ell-1}*\lr{b_i,b_j}+(a_1,\dots,b_{\ell-2})*\lr{a_{\ell-1},b_j}*b_i\Big]
\\
&\quad +\sum^i_{\ell=3}\Big[ \lr{(a_1,\dots,b_{\ell-2}),b_j}*a_{\ell-1}b_i\Big]+a_1*\lr{b_i,b_j}+\lr{a_1,b_j}*b_i\,.
\end{align*}
\end{small}%
If we apply \eqref{eq:prop-technic-euler-i} in the first term, and the inductive hypothesis in the second and fifth terms, we obtain:
\begin{small}
\begin{align*}
&=\frac{1}{2}b_j(a_i,b_i)\otimes (a_1,\dots,b_{i-1})-\frac{1}{2}b_j\otimes (a_1,\dots,b_{i-1})(a_i,b_i)
\\
&\quad +\frac{1}{2}\Big(b_j(a_1,\dots,b_{i-1})(a_i,b_i)\otimes e_2-b_j(a_i,b_i)\otimes (a_1,\dots,b_{i-1})\Big)
\\
&\quad -\frac{1}{2}\sum^i_{\ell=3}\Big[b_i\otimes (a_1,\dots,b_{\ell-2})a_{\ell-1}b_j+b_j\otimes (a_1,\dots ,b_{\ell-2})a_{\ell-1}b_i\Big]
\\
&\quad+ \frac{1}{2}\sum^i_{\ell=3}\Big[b_i\otimes (a_1,\dots,b_{\ell-2})a_{\ell-1}b_j+b_ja_{\ell-1}b_i\otimes (a_1,\dots,b_{\ell-2})\Big]
\\
&\quad+ \frac{1}{2}\sum^i_{\ell=3}\Big[b_j (a_1,\dots,b_{\ell-2})a_{\ell-1}b_i\otimes e_2-b_ja_{\ell-1}b_i\otimes (a_1,\dots,b_{\ell-2})\Big]
\\
&\quad -\frac{1}{2}\Big(b_i\otimes a_1b_j+b_j\otimes a_1b_i\Big)+\frac{1}{2}\Big(b_i\otimes a_1b_j+b_ja_1b_i\otimes e_2\Big)
\\
&=\frac{1}{2}\Big(b_j (a_1,\dots,b_{i-1})(a_i,b_i)\otimes e_2+\sum^i_{\ell=3}\Big[b_j(a_1,\dots ,b_{\ell-2})a_{\ell-1}b_i\otimes e_2\Big]+b_ja_1b_i\otimes e_2\Big)
\\
&\quad -\frac{1}{2}\Big(b_j\otimes (a_1,\dots,b_{i-1})(a_i,b_i)+\sum^i_{\ell=3}\Big[b_j\otimes(a_1,\dots ,b_{\ell-2})a_{\ell-1}b_i\Big]+b_j\otimes a_1b_i\Big)
\\
&=\frac{1}{2}\Big(b_j(a_1,\dots,b_i)\otimes e_2-b_j\otimes (a_1,\dots, b_i)\Big)\,,
\end{align*}
\end{small}%
where we used \eqref{eq:Convenient-formula-Euler-cont-b}.

The proof of the formula in (ii) is pretty similar to (i), so it is left to the reader. Nevertheless, the formula in (iii) is slightly more delicate, so we prove it now. If $1\leq i\leq n-1$, note that by \eqref{eq:Euler-double-bracket.a} and \eqref{eq:Euler-double-bracket.d}, we get
\begin{small}
\begin{equation}
\begin{aligned}
\lr{(a_i,b_i),a_{i+1}}&=a_i*\lr{b_i,a_{i+1}}+\lr{a_i,a_{i+1}}*b_{i}
\\
&=a_i*\Big(\frac{1}{2}\big(e_2\otimes b_ia_{i+1}+a_{i+1}b_i\otimes e_1\big)+e_2\otimes e_1\Big)-\frac{1}{2}\Big(a_i\otimes a_{i+1}+a_{i+1}\otimes a_i\Big)*b_i
\\
&=\frac{1}{2}e_2\otimes a_ib_ia_{i+1}+e_2\otimes a_{i}-\frac{1}{2}a_ib_i\otimes a_{i+1}
\\
&=\frac{1}{2}\Big(e_2\otimes (a_i,b_i)a_{i+1}-(a_i,b_i)\otimes a_{i+1}\Big)+e_2\otimes (a_i)\,.
\end{aligned}
\label{eq:prop-technic-euler-iii}
\end{equation}
\end{small}%
In particular, if $i=1$, we have the base case of (iii).

Next, we address the general case by assuming that $i\geq 2$. By \eqref{eq:Convenient-formula-Euler-cont-b},
\begin{small}
\begin{align*}
\lr{(a_1,&\dots, b_i),a_{i+1}}=\dgal{\Big((a_1,\dots ,b_{i-1})(a_i,b_i)+\sum^i_{\ell=3}\Big[(a_1,\dots,b_{\ell-2})a_{\ell-1}b_i\Big]+a_1b_i\Big),a_{i+1}}
\\
&=(a_1,\dots, b_{i-1})*\lr{(a_i,b_i),a_{i+1}}+\lr{(a_1,\dots,b_{i-1}),a_{i+1}}*(a_i,b_i)
\\
&\quad +\sum^i_{\ell=3}\Big[(a_1,\dots,b_{\ell-2})a_{\ell-1}*\lr{b_i,a_{i+1}}+(a_1,\dots,b_{\ell-2})*\lr{a_{\ell-1},a_{i+1}}*b_i\Big]
\\
&\quad +\sum^i_{\ell=3}\Big[\lr{(a_1,\dots,b_{\ell-2}),a_{i+1}}*a_{\ell-1}b_i\Big]
+a_1*\lr{b_i,a_{i+1}}+\lr{a_1,a_{i+1}}*b_i\,.
\end{align*}
\end{small}%
Now, we apply \eqref{eq:prop-technic-euler-iii} in the first summand, and Lemma \ref{lem:technic-lemma-no-inductive}(ii) in the second and fifth summands to obtain
\begin{small}
\begin{align*}
&=\frac{1}{2}\Big(e_2\otimes (a_1,\dots,b_{i-1})(a_i,b_i)a_{i+1}-(a_i,b_i)\otimes (a_1,\dots,b_{i-1})a_{i+1}\Big)+e_2\otimes (a_1,\dots,b_{i-1})a_i
\\
&\; +\frac{1}{2}\Big((a_i,b_i)\otimes (a_1,\dots,b_{i-1})a_{i+1}-(a_1,\dots,b_{i-1})(a_i,b_i)\otimes a_{i+1}\Big)
\\
&\; +\sum^i_{\ell=3}\Big[\frac{1}{2}\Big(e_2\otimes (a_1,\dots,b_{\ell-2})a_{\ell-1}b_ia_{i+1}+a_{i+1}b_i\otimes (a_1,\dots,b_{\ell-2})a_{\ell-1}\Big)+e_2\otimes (a_1,\dots,b_{\ell-2})a_{\ell-1}\Big]
\\
&\; -\frac{1}{2}\sum^i_{\ell=3}\Big[a_{\ell-1}b_i\otimes (a_1,\dots,b_{\ell-2})a_{i+1}+a_{i+1}b_i\otimes (a_1,\dots,b_{\ell-2})a_{\ell-1}\Big]
\\
&\; +\frac{1}{2}\sum^i_{\ell=3}\Big[a_{\ell-1}b_i\otimes (a_1,\dots,b_{\ell-2})a_{i+1}-(a_1,\dots,b_{\ell-2})a_{\ell-1}b_i\otimes a_{i+1}\Big]
\\
&\; +\frac{1}{2}\Big(e_2\otimes a_1b_ia_{i+1}+a_{i+1}b_i\otimes a_1\Big)+e_2\otimes a_1
 -\frac{1}{2}\Big(a_1b_i\otimes a_{i+1}+a_{i+1}b_i\otimes a_1\Big)
\\
&=\frac{1}{2}\Big(e_2\otimes (a_1,\dots,b_{i-1})(a_i,b_i)a_{i+1}+\sum^i_{\ell=3}\Big[e_2\otimes (a_1,\dots,b_{\ell-2})a_{\ell-1}b_ia_{i+1}\Big]+e_2\otimes a_1b_ia_{i+1}\Big)
\\
&\;-\frac{1}{2}\Big((a_1,\dots,b_{i-1})(a_i,b_i)\otimes a_{i+1}+\sum^i_{\ell=3}\Big[(a_1,\dots,b_{\ell-2})a_{\ell-1}b_i\otimes a_{i+1}+a_1b_i\otimes a_{i+1}\Big)
\\
&\;+e_2\otimes (a_1,\dots,b_{i-1})a_i+\sum^i_{\ell=3}\Big[e_2\otimes (a_1,\dots,b_{\ell-2})a_{\ell-1}\Big]+e_2\otimes a_1
\\
&=\frac{1}{2}\Big(e_2\otimes (a_1,\dots,b_i)a_{i+1}-(a_1,\dots,b_i)\otimes a_{i+1}\Big)+e_2\otimes (a_1,\dots,a_i)\,,
\end{align*}
\end{small}%
where in the last identity we used \eqref{eq:Convenient-formula-Euler-cont-b} and \eqref{eq:Convenient-formula-Euler-cont-a}.
\end{proof}

\subsection{The double bracket (\ref{eq:Euler-double-bracket}) is quasi-Poisson}
\label{sec:proof-Euler-quasi-Poisson}
\allowdisplaybreaks

By its definition, it is clear that $\lr{-,-}$, as given in \eqref{eq:Euler-double-bracket}, is a $B$-linear double bracket on $\mathcal{B}(\Gamma_n)$. In order to prove that the pair $(\mathcal{B}(\Gamma_n),\lr{-,-})$ is a double quasi-Poisson algebra, we need to show that $\lr{-,-}$ satisfies \eqref{qPabc} for every triple of generators $\{x_i,y_j,z_k\}$, with $x_i,y_j,z_k\in\{a_\ell,b_{\ell}\mid 1\leq \ell\leq n\}$, by distinguishing every possible case. To do that, we need to attend to the number of $a$'s and $b$'s occurring in such a triple. 
We essentially have four cases\footnote{We warn the reader that, throughout this part of the proof, we will repeatedly deal with expressions such as $(a_i,a_j,b_k)$ as an ordered triple, \emph{not} as an Euler continuant.}: $(a_i,a_j,a_k)$, $(a_i,a_j,b_k)$, $(b_i,b_j,a_k)$ and $(b_i,b_j,b_k)$. Then, we shall keep track of the relations of the indices $i,j,k$ with respect to the natural total order $\leq$ on $\mathbb{N}$. By cyclicity of the triple bracket without loss of generality, we will assume that the distinct element, if any, is in the third position.

Firstly, we note the following general result. 
\begin{lem} \label{Lem-qP}
 Let $A_0$ be an algebra over $B_0$, where $B_0$ contains orthogonal idempotents $e_1,e_2$. Assume that $A_0$ is endowed with a $B_0$-linear double bracket $\lr{-,-}$, and that there exist elements $c_i\in e_2 A_0 e_1$, $i\in \{1,\ldots,n\}$, such that 
 \begin{equation} \label{Eq:cc-Lem-qP}
  \lr{c_i,c_j}=\epsilon_{ij} (c_i \otimes c_j + c_j \otimes c_i)\,, \quad i,j\in \{1,\ldots,n\}\,,
 \end{equation}
with $\epsilon_{ij}\in \kk$. Then $\epsilon_{ij}=-\epsilon_{ji}$, and we have that the quasi-Poisson property  \eqref{qPabc} is satisfied when evaluated on any triples from $\{c_i\}$ if and only if 
\begin{equation} \label{Eq:Cond-Lem-qP}
C_{ijk}:= \epsilon_{ij}\epsilon_{jk}+\epsilon_{jk}\epsilon_{ki}+\epsilon_{ki}\epsilon_{ij}=-\frac14\,,
\end{equation}
for all $i,j,k\in \{1,\ldots,n\}$ with $(i,j,k)\neq (i,i,i)$, i.e. the three indices are not all equal. 
\end{lem}
\begin{proof}
Clearly, the cyclic antisymmetry implies that $\epsilon_{ij}=-\epsilon_{ji}$. 
 The quasi-Poisson property  \eqref{qPabc}  on $c_i,c_j,c_k$  reads:
\begin{equation}
\lr{c_i,c_j,c_k} =\frac{1}{4}\big(c_k\otimes c_i\otimes c_j-c_i\otimes c_j\otimes c_k\big)\,.
\label{eq:RHS-a-a-a-Bis}
\end{equation}
Now, we can compute that 
\begin{align*}
\lr{c_i,\lr{c_j,c_k}}_L &= \epsilon_{jk} [\epsilon_{ij} (c_i \otimes c_j \otimes c_k + c_j \otimes c_i \otimes c_k) 
+ \epsilon_{ik} (c_i \otimes c_k \otimes c_j + c_k \otimes c_i \otimes c_j) ]  \\
\tau_{(123)}\lr{c_j,\lr{c_k,c_i}}_L &= \epsilon_{ki} [\epsilon_{jk} (c_i \otimes c_j \otimes c_k + c_i \otimes c_k \otimes c_j) 
+ \epsilon_{ji} (c_k \otimes c_j \otimes c_i + c_k \otimes c_i \otimes c_j) ] \\
\tau_{(132)}\lr{c_k,\lr{c_i,c_j}}_L &= \epsilon_{ij} [\epsilon_{ki} (c_i \otimes c_j \otimes c_k + c_k \otimes c_j \otimes c_i) 
+ \epsilon_{kj} (c_j \otimes c_i \otimes c_k + c_k \otimes c_i \otimes c_j) ] 
\end{align*}
Using the identity $\epsilon_{ij}=-\epsilon_{ji}$, we obtain 
\begin{equation}
\lr{c_i,c_j,c_k}= C_{ijk} \, (c_i \otimes c_j \otimes c_k - c_k \otimes c_i \otimes c_j)\,,
\end{equation}
which satisfies \eqref{eq:RHS-a-a-a-Bis} if and only if $C_{ijk}=-\frac14$ when $i,j,k$ are not all equal. 
\end{proof}

In our case, we note that if we consider  $A_0=\kk\overline{\Gamma}_n$ (or $\mathcal{B}(\Gamma_n)$ directly), $B_0=B$ and the elements $c_i=a_i$, then \eqref{eq:Euler-double-bracket.a} takes the  form \eqref{Eq:cc-Lem-qP} for 
\begin{equation}
 \epsilon_{ij}=\frac12 \, \operatorname{sgn}(j-i)\,.
\end{equation}
We can then directly verify that the condition \eqref{Eq:Cond-Lem-qP} holds in the cases 
\begin{equation}
 i=j=k,\quad i=j \neq k,\quad i<j<k,\quad i<k<j,\quad i>j>k,\quad i>k>j\,.
\end{equation}
Due to the cyclic antisymmetry \eqref{Eq:cycanti}, these are the only cases to check and we obtain from Lemma \ref{Lem-qP} that \eqref{qPabc} holds on any triple $(a_i,a_j, a_k)$ with $i,j,k\in\{1,\dots, n\}$.

Secondly, we shall show that \eqref{qPabc} holds for every triple $(b_i,b_j,b_k)$, where $1\leq i,j,k\leq n$. 
In fact, this can be obtained by an application of Lemma \ref{Lem-qP}, but we shall use a different method that we will employ below.
As in the proof of \cite[Theorem 6.5.1]{VdB1}, we note the existence of an automorphism of order two on $\kk\overline{\Gamma}_n$, namely $S$,  given by 
\begin{equation}
e_1\longleftrightarrow e_2\,,\quad a_\ell \longleftrightarrow b_{n+1-\ell}\,,\quad b_\ell\longleftrightarrow a_{n+1-\ell}\,,
\label{eq:automorphism-of-order-two}
\end{equation}
for all $1\leq\ell\leq n$. By inspection on \eqref{eq:Euler-double-bracket}, we can see that this automorphism has the effect $\lr{-,-}\mapsto -\lr{-,-}$. Moreover, this automorphism has no effect on the triple bracket $\lr{-,-,-}$ defined by \eqref{Eq:TripBr}; in other words, $\lr{-,-,-}\mapsto \lr{-,-,-}$. Once we have just proved that the identity $4\lr{a_i,a_j,a_k}=a_k\otimes a_i\otimes a_j-a_i\otimes a_j\otimes a_k$ 
is satisfied for all $1\leq i,j,k\leq n$, we can apply the automorphism $S$ to this identity to obtain 
\[
\lr{b_{n+1-i},b_{n+1-j},b_{n+1-k}}=\frac{1}{4}\Big(b_{n+1-k}\otimes b_{n+1-i}\otimes b_{n+1-j}-b_{n+1-i}\otimes b_{n+1-j}\otimes b_{n+1-k}\Big)\,,
\]
for all $1\leq i,j,k\leq n$.
In this way, we can conclude that \eqref{qPabc} also holds for every triple $(b_i,b_j,b_k)$, where $1\leq i,j,k\leq n$.\\

Thirdly, we will prove that the bracket $\lr{-,-}$ defined in \eqref{eq:Euler-double-bracket} satisfies \eqref{qPabc} for triples $(a_i,a_j,b_k)$, where $1\leq i,j,k\leq n$. To achieve this purpose, we shall distinguish the different cases attending to the different values of the indices $i,j,k$. To begin with, note that in this case the right-hand side of \eqref{qPabc} reduces to
\begin{equation}
\frac{1}{4}\Big(b_ka_i\otimes a_j\otimes e_2-e_1\otimes a_i\otimes a_jb_k\Big)\,.
\label{eq:RHS-a-a-b}
\end{equation}

Now, the simplest case occurs when the three indices are equal, namely $i=j=k$. Then, since $\lr{a_i,a_i}=0$, and by \eqref{eq:Euler-double-bracket.a}, \eqref{eq:Euler-double-bracket.c} and \eqref{eq:Euler-double-bracket.d}, at the left-hand side of \eqref{qPabc} we obtain
\begin{small}
\begin{align*}
\lr{a_i,a_i,b_i}&=\lr{a_i,\lr{a_i,b_i}}_L+\tau_{(123)}\lr{a_i,\lr{b_i,a_i}}_L
\\
&=\frac{1}{2}\Big(\lr{a_i,b_i}a_i\otimes e_2\Big)-\frac{\tau_{(123)}}{2}\Big(a_i\lr{a_i,b_i}\otimes e_1\Big)
\\
&=\frac{1}{4}\Big(b_ia_i\otimes a_i\otimes e_2+e_1\otimes a_ib_ia_i\otimes e_2+2e_1\otimes a_i\otimes e_2\Big)
\\
&\quad-\frac{\tau_{(123)}}{4}\Big(a_ib_ia_i\otimes e_2\otimes e_1+a_i\otimes a_ib_i\otimes e_1+2a_i\otimes e_2\otimes e_1\Big)
\\
&=\frac{1}{4}\Big(b_ia_i\otimes a_i\otimes e_2-e_1\otimes a_i\otimes a_ib_i\Big)\,,
\end{align*}
\end{small}%
which coincides with \eqref{eq:RHS-a-a-b}. 

Now, we address the cases when two indices coincide. If $i=j$ and $i<k$, by \eqref{eq:Euler-double-bracket.c} and \eqref{eq:Euler-double-bracket.d},
\begin{small}
\begin{align*}
\lr{a_i,a_i,b_k}&=\lr{a_i,\lr{a_i,b_k}}_L+\tau_{(123)}\lr{a_i,\lr{b_k,a_i}}_L
\\
&=\frac{1}{2}\Big(\lr{a_i,b_k}a_i\otimes e_2\Big)-\frac{\tau_{(123)}}{2}\Big(a_i\lr{a_i,b_k}\otimes e_1\Big)
\\
&=\frac{1}{4}\Big(e_1\otimes a_ib_ka_i\otimes e_2+b_ka_i\otimes a_i\otimes e_2\Big)-\frac{\tau_{(123)}}{4}\Big(a_i\otimes a_ib_k\otimes e_1+a_ib_ka_i\otimes e_2\otimes e_1\Big)
\\
&=\frac{1}{4}\Big(b_ka_i\otimes a_i\otimes e_2-e_1\otimes a_i\otimes a_ib_k\Big)\,,
\end{align*} 
\end{small}%
as we wished. 
Similarly, the reader can check that the four cases occurring when $i=k$ and $j\neq k$, and $j=k$ and $i\neq k$ are completely analogous to the previous ones.

Finally, we address the six possible cases of triples $(a_i,a_j,b_k)$ when $i,j,k$ are all distinct.
Let $(a_i,a_j,b_k)$ with $i<j<k$. Then, by \eqref{eq:Euler-double-bracket.a}, \eqref{eq:Euler-double-bracket.c} and \eqref{eq:Euler-double-bracket.d}, we obtain
\begin{small}
\begin{align*}
\lr{a_i,&a_j,b_k}=\lr{a_i,\lr{a_j,b_k}}_L+\tau_{(123)}\lr{a_j,\lr{b_k,a_i}}_L+\tau_{(132)}\lr{b_k,\lr{a_i,a_j}}_L
\\
&=\frac{1}{2}\Big(\big(b_k\lr{a_i,a_j}+\lr{a_i,b_k}a_j\big)\otimes e_2\Big)
-\frac{\tau_{(123)}}{2}\Big(\big(a_i\lr{a_j,b_k}+\lr{a_j,a_i}b_k\big)\otimes e_1\Big)
\\
&\quad -\frac{\tau_{(132)}}{2}\Big(\lr{b_k,a_i}\otimes a_j+\lr{b_k,a_j}\otimes a_i\Big)
\\
&=-\frac{1}{4}\Big(b_ka_i\otimes a_j\otimes e_2+b_ka_j\otimes a_i\otimes e_2-e_1\otimes a_ib_ka_j\otimes e_2-b_ka_i\otimes a_j\otimes e_2\Big)
\\
&\quad -\frac{\tau_{(123)}}{4}\Big(a_i\otimes a_jb_k\otimes e_1+a_ib_ka_j\otimes e_2\otimes e_1+a_j\otimes a_ib_k\otimes e_1+a_i\otimes a_jb_k\otimes e_1\Big)
\\
&\quad +\frac{\tau_{(132)}}{4}\Big(a_ib_k\otimes e_1\otimes a_j+e_2\otimes b_ka_i\otimes a_j+a_jb_k\otimes e_1\otimes a_i+e_2\otimes b_ka_j\otimes a_i\Big)
\\
&=\frac{1}{4}\Big(b_ka_i\otimes a_j\otimes e_2-e_1\otimes a_i\otimes a_jb_k\Big)\,;
\end{align*}
\end{small}%
consequently, \eqref{qPabc} holds.
The following case occurs when $i<j$, $i<k$ and $j>k$:
\begin{small}
\begin{align*}
&\lr{a_i,a_j,b_k}
\\
&=-\frac{1}{2}\Big(\big(b_k\lr{a_i,a_j}+\lr{a_i,b_k}a_j\big)\otimes e_2\Big)-\frac{\tau_{(123)}}{2}\Big(\big(a_i\lr{a_j,b_k}+\lr{a_j,a_i}b_k
\big)\otimes e_1\Big)
\\
&\quad -\frac{\tau_{(132)}}{2}\Big(\lr{b_k,a_i}\otimes a_j+\lr{b_k,a_j}\otimes a_i\Big)
\\
&=\frac{1}{4}\Big(b_ka_i\otimes a_j\otimes e_2+b_ka_j\otimes a_i\otimes e_2-e_1\otimes a_ib_ka_j\otimes e_2-b_ka_i\otimes a_j\otimes e_2\Big)
\\
&\quad +\frac{\tau_{(123)}}{4}\Big(a_ib_ka_j\otimes e_2\otimes e_1+a_i\otimes a_jb_k\otimes e_1+2\delta_{j-k,1}(a_i\otimes e_2\otimes e_1)-a_j\otimes a_ib_k\otimes e_1
\\
&\qquad \qquad -a_i\otimes a_jb_k\otimes e_1\Big)
\\
&\quad +\frac{\tau_{(132)}}{4}\Big(a_ib_k\otimes e_1\otimes a_j+e_2\otimes b_ka_i\otimes a_j-e_2\otimes b_ka_j\otimes a_i-a_jb_k\otimes e_1\otimes a_i
\\
&\qquad \qquad -2\delta_{j-k,1}(e_2\otimes e_1\otimes a_i)\Big)
\\
&=\frac{1}{4}\Big(b_ka_i\otimes a_j\otimes e_2-e_1\otimes a_i\otimes a_jb_k\Big)\,,
\end{align*}
\end{small}%
which is tantamount to \eqref{eq:RHS-a-a-b}. 
The next case occurs when $i<j$, $i>k$ and $j>k$. It is important to note that with these constrains $j-k\neq 1$, because the condition $j-k=1$ would imply that $0<i-k<1$, which is a contradiction since $i,k\in\mathbb{N}$. Then, by \eqref{eq:Euler-double-bracket.a}, \eqref{eq:Euler-double-bracket.c} and \eqref{eq:Euler-double-bracket.d}, we have
\begin{small}
\begin{align*}
&\lr{a_i,a_j,b_k}
\\
&=-\frac{1}{2}\Big(\big(b_k\lr{a_i,a_j}+\lr{a_i,b_k}a_j\big)\otimes e_2\Big)
+\frac{\tau_{(123)}}{2}\Big(\big(a_i\lr{a_j,b_k}+\lr{a_j,a_i}b_k\big)\otimes e_1\Big)
\\
&\quad -\frac{\tau_{(132)}}{2}\Big(\lr{b_k,a_i}\otimes a_j+\lr{b_k,a_j}\otimes a_i\Big)
\\
&=\frac{1}{4}\Big(b_ka_i\otimes a_j\otimes e_2+b_ka_j\otimes a_i\otimes e_2+b_ka_i\otimes a_j\otimes e_2+e_1\otimes a_ib_ka_j\otimes e_2+2\delta_{i-k,1}(e_1\otimes a_j\otimes e_2)\Big)
\\
&\quad -\frac{\tau_{(123)}}{4}\Big(a_ib_ka_j\otimes e_2\otimes e_1+a_i\otimes a_jb_k\otimes e_1-a_j\otimes a_ib_k\otimes e_1-a_i\otimes a_jb_k\otimes e_1\Big)
\\
&\quad -\frac{\tau_{(132)}}{4}\Big(e_2\otimes b_ka_i\otimes a_j+a_ib_k\otimes e_1\otimes a_j+2\delta_{i-k,1}(e_2\otimes e_1\otimes a_j)+e_2\otimes b_ka_j\otimes a_i
\\
&\qquad \qquad +a_jb_k\otimes e_1\otimes a_i\Big)
\\
&=\frac{1}{4}\Big(b_ka_i\otimes a_j\otimes e_2-e_1\otimes a_i\otimes a_jb_k\Big)\,,
\end{align*}
\end{small}%
which is \eqref{eq:RHS-a-a-b}. 
Next, to prove that the bracket $\lr{-,-}$ as defined in \eqref{eq:Euler-double-bracket} is a double quasi-Poisson bracket, we need to check that \eqref{qPabc} holds for the triple $(a_i,a_j,b_k)$, whose indices are subject to the constrains $i>j$, $j>k$ and $i>k$. Note that $i-k\neq 1$. If not, $0<j-k<1$ that contradicts the fact that $j,k\in\mathbb{N}$. Then,
\begin{small}
\begin{align*}
&\lr{a_i,a_j,b_k}
\\
&=-\frac{1}{2}\Big(\big(b_k\lr{a_i,a_j}+\lr{a_i,b_k}a_j\big)\otimes e_2\Big)+\frac{\tau_{(123)}}{2}\Big(\big(a_i\lr{a_j,b_k}+\lr{a_j,a_i}b_k\big)\otimes e_1\Big)
\\
&\quad +\frac{\tau_{(132)}}{2}\Big(\lr{b_k,a_i}\otimes a_j+\lr{b_k,a_j}\otimes a_i\Big)
\\
&=-\frac{1}{4}\Big(b_ka_i\otimes a_j\otimes e_2+b_ka_j\otimes a_i\otimes e_2-b_ka_i\otimes a_j\otimes e_2-e_1\otimes a_ib_ka_j\otimes e_2\Big)
\\
&\quad - \frac{\tau_{(123)}}{4}\Big(a_ib_ka_j\otimes e_2\otimes e_1+a_i\otimes a_jb_k\otimes e_1+2\delta_{j-k,1}(a_i\otimes e_2\otimes e_1)+a_j\otimes a_ib_k\otimes e_1
\\
&\qquad \qquad +a_i\otimes a_jb_k\otimes e_1\Big)
\\
&\quad +\frac{\tau_{(132)}}{4}\Big(e_2\otimes b_ka_i\otimes a_j+a_ib_k\otimes e_1\otimes a_j+e_2\otimes b_ka_j\otimes a_i+a_jb_k\otimes e_1\otimes a_i
\\
&\qquad \qquad +2\delta_{j-k,1}(e_2\otimes e_1\otimes a_i)\Big)
\\
&=\frac{1}{4}\Big(b_ka_i\otimes a_j\otimes e_2-e_1\otimes a_i\otimes a_jb_k\Big)
\end{align*}
\end{small}%
that is just  the right-hand side of \eqref{qPabc}.
Next, the fifth case that we need to study comes from the triple $(a_i,a_j,b_k)$ such that $i>j$, $j<k$ and $i>k$. Then,
\begin{small}
\begin{align*}
&\lr{a_i,a_j,b_k}
\\
&=\frac{1}{2}\Big(\big(b_k\lr{a_i,a_j}+\lr{a_i,b_k}a_j\big)\otimes e_2\Big)
+\frac{\tau_{(123)}}{2}\Big(\big(a_i\lr{a_j,b_k}+\lr{a_j,a_i}b_k\big)\otimes e_1\Big)
\\
&\quad +\frac{\tau_{(132)}}{2}\Big(\lr{b_k,a_i}\otimes a_j+\lr{b_k,a_j}\otimes a_i\Big)
\\
&=\frac{1}{4}\Big(b_ka_i\otimes a_j\otimes e_2+b_ka_j\otimes a_i\otimes e_2-b_ka_i\otimes a_j\otimes e_2-e_1\otimes a_ib_ka_j\otimes e_2
\\
&\qquad \qquad -2\delta_{i-k,1}(e_1\otimes a_j\otimes e_2)\Big)
\\
&\quad +\frac{\tau_{(123)}}{4}\Big(a_i\otimes a_jb_k\otimes e_1+a_ib_ka_j\otimes e_2\otimes e_1-a_j\otimes a_ib_k\otimes e_1-a_i\otimes a_jb_k\otimes e_1\Big)
\\
&\quad +\frac{\tau_{(132)}}{4}\Big(e_2\otimes b_ka_i\otimes a_j+a_ib_k\otimes e_1\otimes a_j+2\delta_{i-k,1}(e_2\otimes e_1\otimes a_j)-a_jb_k\otimes e_1\otimes a_i
\\
&\qquad\qquad  -e_2\otimes b_ka_j\otimes a_i\Big)
\\
&=\frac{1}{4}\Big(b_ka_i\otimes a_j\otimes e_2-e_1\otimes a_i\otimes a_jb_k\Big)\,,
\end{align*}
\end{small}%
which proves \eqref{qPabc}.
The final case for a triple $(a_i,a_j,b_k)$, with $i,j,k$ distinct occurs when $i>j$, $j<k$ and $i<k$. So, by \eqref{eq:Euler-double-bracket.a}, \eqref{eq:Euler-double-bracket.c} and \eqref{eq:Euler-double-bracket.d}, we can write
\begin{small}
\begin{align*}
&\lr{a_i,a_j,b_k}
\\
&=\frac{1}{2}\Big(\big(b_k\lr{a_i,a_j}+\lr{a_i,b_k}a_j\big)\otimes e_2\Big)
-\frac{\tau_{(123)}}{2}\Big(\big(a_i\lr{a_j,b_k}+\lr{a_j,a_i}b_k\big)\otimes e_1\Big)
\\
&\quad +\frac{\tau_{(132)}}{2}\Big(\lr{b_k,a_i}\otimes a_j+\lr{b_k,a_j}\otimes a_i\Big)
\\
&=\frac{1}{4}\Big(b_ka_i\otimes a_j\otimes e_2+b_ka_j\otimes a_i\otimes e_2+e_1\otimes a_ib_ka_j\otimes e_2+b_ka_i\otimes a_j\otimes e_2\Big)
\\
&\quad -\frac{\tau_{(123)}}{4}\Big(a_i\otimes a_jb_k\otimes e_1+a_ib_ka_j\otimes e_2\otimes  e_1-a_j\otimes a_ib_k\otimes e_1-a_i\otimes a_jb_k\otimes e_1\Big)
\\
&\quad -\frac{\tau_{(132)}}{4}\Big(a_ib_k\otimes e_1\otimes a_j+e_2\otimes b_ka_i\otimes a_j+a_jb_k\otimes e_1\otimes a_i+e_2\otimes b_ka_j\otimes a_i\Big)
\\
&=\frac{1}{4}\Big(b_ka_i\otimes a_j\otimes e_2-e_1\otimes a_i\otimes a_jb_k\Big)\,,
\end{align*}
\end{small}%
as we wanted. So, we can conclude that \eqref{qPabc} holds for every triple $(a_i,a_j,b_k)$, where $1\leq i,j,k\leq n$.\\


Finally, we deal with the case $(b_i,b_j,a_k)$, with $1\leq i,j,k\leq n$. Since we proved that \eqref{qPabc} is fulfilled for every triple $(a_i,a_j,b_k)$, that is, $4\lr{a_i,a_j,b_k}=b_ka_i\otimes a_j\otimes e_2-e_1\otimes a_i\otimes a_jb_k$, we can apply the automorphism $S$ on $\kk\overline{\Gamma}_n$ defined in \eqref{eq:automorphism-of-order-two} to obtain the identity
\[
\lr{b_{n+1-i},b_{n+1-j},a_{n+1-k}}=\frac{1}{4}\Big(a_{n+1-k}b_{n+1-i}\otimes b_{n+1-j}\otimes e_1-e_2\otimes b_{n+1-i}\otimes b_{n+1-j}
a_{n+1-k}\Big)\,,
\]
which holds for all indices $i,j,k\in\{1,\dots,n\}$. Hence, \eqref{qPabc} is also satisfied for a given triple $(b_i,b_j,a_k)$ with $1\leq i,j,k\leq n$.

To sum up, we proved that the pair $(\Bc(\Gamma_n),\lr{-,-})$ (in fact, the pair $(\kk\overline{\Gamma}_n,\lr{-,-})$), where $\lr{-,-}$ was defined in \eqref{eq:Euler-double-bracket}, is a double quasi-Poisson algebra (over $B$).

\subsection{The element (\ref{eq:mult-moment-map-statement}) is a multiplicative moment map}
\label{sec:proof-Euler-multiplicative-moment-map}
\allowdisplaybreaks

In this subsection, we proceed to prove that \eqref{eq:mult-moment-map-statement} is a multiplicative moment map; 
in other words, we show that \eqref{eq:mult-moment-map-statement} satisfies \eqref{Phim} for $s\in\{1,2\}$ and $a_1,\dots ,a_n,b_1,\dots, b_n$, which are the generators of $\Bc(\Gamma_n)$. We write $\Phi=\Phi_1+\Phi_2$, where
\[
\Phi_1=(b_n,\dots,a_1)^{-1}\in e_1 \Bc(\Gamma_n) e_1\,,\quad \Phi_2=(a_1,\dots,b_n)\in e_2 \Bc(\Gamma_n) e_2\,.
\]

Since Euler continuants are defined by the recursive relation \eqref{eq:def-Euler-continuant-very-beginning}, we shall use strong induction. Note that the base case is given by Theorem \ref{tm:VdB-mult-preproj-quasi-Hamilt}(i).
Next, to fix ideas, the inductive hypothesis states that lower Euler continuants satisfy \eqref{Phim}. In particular, using the orthogonality of the idempotents, the inductive hypothesis gives rise to the formulas
\begin{subequations}
\label{eq:inductive-hypothesis}
\begin{align}
\lr{(a_1,\dots b_{\ell}),a_{i}}&=-\frac{1}{2}\Big(e_2\otimes (a_1,\dots b_{\ell})a_i+(a_1,\dots,b_{\ell})\otimes a_i\Big)\, ,\label{eq:inductive-hypothesis.a}
\\
\lr{(a_1,\dots b_{\ell}),b_{i}}&=\;\frac{1}{2}\Big(b_i\otimes (a_1,\dots,b_{\ell})+b_i(a_1,\dots,b_{\ell})\otimes e_2\Big)\, , \label{eq:inductive-hypothesis.b}
\end{align}
\end{subequations}
where $\ell\in\{1,\dots,n-1\}$, and $1\leq i\leq \ell$.

By the very definition of the double quasi-Poisson bracket \eqref{eq:Euler-double-bracket}, we need to distinguish different cases, depending on the use of the inductive hypothesis and Lemma \ref{lem:technic-lemma-no-inductive}.\\

Let $b_i\in\{b_1,\dots,b_n\}$. Then, since $e_ie_j=\delta_{i,j}e_i$ for $i,j\in\{1,2\}$, the identity \eqref{Phim} reduces to
\begin{small}
\begin{equation}
\begin{aligned}
\lr{\Phi_2,b_i}&=\frac{1}{2}\Big(b_i\otimes \Phi_2+b_i\Phi_2\otimes e_2\Big)
\\
&=\frac{1}{2}\Big(b_i\otimes(a_1,\dots ,b_{n-1})(a_n,b_n)+\sum^n_{\ell=3}\Big[b_i\otimes (a_1,\dots,b_{\ell-2})a_{\ell-1}b_n\Big]+b_i\otimes a_1b_n
\\
&\quad + b_i(a_1,\dots,b_{n-1})(a_n,b_n)\otimes e_2+\sum^n_{\ell=3}\Big[b_i(a_1,\dots,b_{\ell-2})a_{\ell-1}b_n\otimes e_2\Big]+b_ia_1b_n\otimes e_2\Big)\,,
\end{aligned}
\label{eq:moment-map-b}
\end{equation}
\end{small}%
where in the last identity we applied \eqref{eq:Convenient-formula-Euler-cont-b}.

To begin with, we shall check \eqref{Phim} for $b_1$. Using the left Leibniz rule \eqref{Eq:inder} in the first argument (that makes the inner action appear) and separating the case $\ell=3$ in the sum to keep track of $\lr{a_2,b_1}$, we can write
\begin{small}
\begin{align*}
\lr{&\Phi_2,b_1}=\lr{(a_1,\dots,b_n),b_1}
\\
&=\dgal{\Big((a_1,\dots ,b_{n-1})(a_n,b_n)+\sum^n_{\ell=3}\Big[(a_1,\dots,b_{\ell-2})a_{\ell-1}b_n\Big]+a_1b_n\Big),b_1}
\\
&=(a_1,\dots,b_{n-1})*\lr{(a_n,b_n),b_1}+\lr{(a_1,\dots ,b_{n-1}),b_1}*(a_n,b_n)
\\
&\quad  +\sum^n_{\ell=3}\Big[(a_1,\dots, b_{\ell-2})a_{\ell-1}*\lr{b_n,b_1}\Big]
+(a_1,b_1)*\lr{a_2,b_1}*b_n
\\
 &\quad +\sum^{n}_{\ell=4}\Big[(a_1,\dots,b_{\ell-2})*\lr{a_{\ell-1},b_1}*b_n \Big] +\sum^n_{\ell=3}\Big[\lr{(a_1,\dots,b_{\ell-2}),b_1}*a_{\ell-1}b_n\Big]
\\
&\quad + a_1*\lr{b_n,b_1}+\lr{a_1,b_1}*b_n\,.
\end{align*}
\end{small}%
By Lemma \ref{lem:technic-lemma}(ii), the inductive hypothesis \eqref{eq:inductive-hypothesis.b}, and \eqref{eq:Euler-double-bracket}, we have 
\begin{small}
\begin{align*}
&=\frac{1}{2}\Big(b_1\otimes (a_1,\dots,b_{n-1})(a_n,b_n)-b_1(a_n,b_n)\otimes (a_1,\dots,b_{n-1})\Big)
\\
&\quad +\frac{1}{2}\Big(b_1(a_n,b_n)\otimes (a_1,\dots,b_{n-1})+b_1(a_1,\dots,b_{n-1})(a_n,b_n)\otimes e_2\Big)
\\
&\quad +\frac{1}{2}\sum^n_{\ell=3}\Big[b_n\otimes (a_1,\dots,b_{\ell-2})a_{\ell-1}b_1+b_1\otimes (a_1,\dots,b_{\ell-2})a_{\ell-1}b_n\Big]
\\
&\quad -\frac{1}{2}\Big(b_1a_2b_n\otimes (a_1,b_1)+b_n\otimes (a_1,b_1)a_2b_1\Big)-b_n\otimes (a_1,b_1)
\\
&\quad -\frac{1}{2}\sum^n_{\ell=4}\Big[b_1a_{\ell-1}b_n\otimes (a_1,\dots,b_{\ell-2})+b_n\otimes (a_1,\dots,b_{\ell-2})a_{\ell-1}b_1\Big]
\\
&\quad +\frac{1}{2}\sum^n_{\ell=3}\Big[b_1a_{\ell-1}b_n\otimes (a_1,\dots,b_{\ell-2})+b_1(a_1,\dots,b_{\ell-2})a_{\ell-1}b_n\otimes e_2\Big]
\\
&\quad +\frac{1}{2}\Big(b_n\otimes a_1b_1+b_1\otimes a_1b_n\Big)+\frac{1}{2}\Big(b_1a_1b_n\otimes e_2+b_n\otimes a_1b_1\Big)+b_n\otimes e_2
\\
&=\frac{1}{2}\Big(b_1\otimes (a_1,\dots,b_{n-1})(a_n,b_n)+\sum^n_{\ell=3}\Big[b_1\otimes (a_1,\dots,b_{\ell-2})a_{\ell-1}b_n\Big]+b_1\otimes a_1b_n\Big)
\\
&\quad +\frac{1}{2}\Big(b_1(a_1,\dots,b_{n-1})(a_n,b_n)\otimes e_2+\sum^n_{\ell=3}\Big[b_1(a_1,\dots,b_{\ell-2})a_{\ell-1}b_n\otimes e_2\Big]+b_1a_1b_n\otimes e_2\Big]
\\
&\quad -b_n\otimes (a_1,b_1)+b_n\otimes a_1b_1+b_n\otimes e_2
\\
&=\frac{1}{2}\big(b_1\otimes \Phi_2+b_1\Phi_2\otimes e_2\big)\,,
\end{align*}
\end{small}%
where the last identity is due to \eqref{eq:Convenient-formula-Euler-cont-b}.

Now, we need to check \eqref{Phim} for an arrow $b_k$, where $k\in \{2,\dots,n-2\}$ is fixed. Then, 
\begin{small}
\begin{align*}
\lr{&\Phi_2,b_k}=\dgal{\Big((a_1,\dots ,b_{n-1})(a_n,b_n)+\sum^n_{\ell=3}\Big[(a_1,\dots,b_{\ell-2})a_{\ell-1}b_n\Big]+a_1b_n\Big),b_k}
\\
&=(a_1,\dots,b_{n-1})*\lr{(a_n,b_n),b_k}+\lr{(a_1,\dots,b_{n-1}),b_k}*(a_n,b_n)
\\
&\quad +\sum^n_{\ell=3}\Big[(a_1,\dots,b_{\ell-2})a_{\ell-1}*\lr{b_n,b_k}\Big]+\sum^k_{\ell=3}\Big[(a_1,\dots,b_{\ell-2})*\lr{a_{\ell-1},b_k}*b_n\Big]
\\
&\quad +(a_1,\dots,b_{k-1})*\lr{a_k,b_k}*b_n+(a_1,\dots,b_k)*\lr{a_{k+1},b_k}*b_n
\\
&\quad + \sum^n_{\ell=k+3}\Big[(a_1,\dots,b_{\ell-2})*\lr{a_{\ell-1},b_k}*b_n\Big] +\sum^{k+1}_{\ell=3}\Big[\lr{(a_1,\dots,b_{\ell-2}),b_k}*a_{\ell-1}b_n\Big]
\\
&\quad +\sum^n_{\ell=k+2}\Big[\lr{(a_1,\dots,b_{\ell-2}),b_k}*a_{\ell-1}b_n\Big]
 +a_1*\lr{b_n,b_k}+\lr{a_1,b_k}*b_n\,.
\end{align*}
\end{small}%
Applying Lemma \ref{lem:technic-lemma}(ii) in the first summand, the inductive hypothesis \eqref{eq:inductive-hypothesis.b} in the second and ninth summands and Lemma \ref{lem:technic-lemma-no-inductive}(i) in the eighth one, we obtain:
\begin{small}
\begin{align*}
&=\frac{1}{2}\Big(b_k\otimes (a_1,\dots,b_{n-1})(a_n,b_n)-b_k(a_n,b_n)\otimes (a_1,\dots,b_{n-1})\Big)
\\
 &\quad +\frac{1}{2}\Big(b_k(a_n,b_n)\otimes (a_1,\dots,b_{n-1})+b_k(a_1,\dots,b_{n-1})(a_n,b_n)\otimes e_2\Big)
\\
 &\quad +\frac{1}{2}\sum^n_{\ell=3}\Big[b_n\otimes (a_1,\dots,b_{\ell-2})a_{\ell-1}b_k+b_k\otimes (a_1,\dots,b_{\ell-2})a_{\ell-1}b_n\Big]
\\
&\quad +\frac{1}{2}\sum^k_{\ell=3}\Big[b_ka_{\ell-1}b_n\otimes (a_1,\dots,b_{\ell-2})+b_n\otimes (a_1,\dots,b_{\ell-2})a_{\ell-1}b_k\Big]
\\
&\quad  +\frac{1}{2}\Big(b_ka_kb_n\otimes (a_1,\dots ,b_{k-1}) +b_n\otimes (a_1,\dots,b_{k-1})a_kb_k\Big)+b_n\otimes (a_1,\dots,b_{k-1})
\\
&\quad-\frac{1}{2}\Big(b_ka_{k+1}b_n\otimes (a_1,\dots ,b_k)+b_n\otimes (a_1,\dots,b_k)a_{k+1}b_k\Big)-b_n\otimes (a_1,\dots ,b_k)
\\
&\quad  -\frac{1}{2}\sum^n_{\ell=k+3}\Big[b_ka_{\ell-1}b_n\otimes (a_1,\dots,b_{\ell-2})+b_n\otimes (a_1,\dots,b_{\ell-2})a_{\ell-1}b_k\Big]
\\
&\quad +\frac{1}{2}\sum^{k+1}_{\ell=3}\Big[b_k(a_1,\dots, b_{\ell-2})a_{\ell-1}b_n\otimes e_2-b_ka_{\ell-1}b_n\otimes (a_1,\dots,b_{\ell-2})\Big]
\\
&\quad +\frac{1}{2}\sum^{n}_{\ell=k+2}\Big[b_k(a_1,\dots, b_{\ell-2})a_{\ell-1}b_n\otimes e_2+b_ka_{\ell-1}b_n\otimes (a_1,\dots,b_{\ell-2})\Big]
\\
&\quad +\frac{1}{2}\Big(b_n\otimes a_1b_k+b_k\otimes a_1b_n\Big)+\frac{1}{2}\Big(b_ka_1b_n\otimes e_2+b_n\otimes a_1b_k\Big).
\end{align*}
\end{small}%
At this point, if in the third line we rewrite  
\begin{equation*}
\begin{split}
\sum^n_{\ell=3}\Big[b_n\otimes (a_1,\dots,b_{\ell-2})a_{\ell-1}b_k\Big]=\sum^k_{\ell=3}\Big[b_n\otimes (a_1,\dots,b_{\ell-2})a_{\ell-1}b_k\Big]+b_n\otimes (a_1,\dots,b_{k-1})a_kb_k
\\
+b_n\otimes (a_1,\dots,b_{k})a_{k+1}b_k+\sum^n_{\ell=k+3}\Big[b_n\otimes (a_1,\dots,b_{\ell-2})a_{\ell-1}b_k\Big]\,,
\end{split}
\end{equation*}
we can cancel some terms, obtaining
\begin{small}
\begin{align*}
&=\frac{1}{2}\Big(b_k\otimes (a_1,\dots,b_{n-1})(a_n,b_n)+\sum^n_{\ell=3}\Big[b_k\otimes (a_1,\dots,b_{\ell-2})a_{\ell-1}b_n\Big]+b_k\otimes a_1b_n\Big)
\\
&\quad +\frac{1}{2}\Big(b_k(a_1,\dots,b_{n-1})(a_n,b_n)\otimes e_2+\sum^n_{\ell=3}\Big[b_k(a_1,\dots,b_{\ell-2})a_{\ell-1}b_n\otimes e_2\Big]+b_ka_1b_n\otimes e_2\Big)
\\
&\quad -b_n\otimes (a_1,,\dots,b_k)+b_n\otimes (a_1,\dots,b_{k-1})(a_k,b_k)+\sum^k_{\ell=3}\Big[b_n\otimes (a_1,\dots,b_{\ell-2})a_{\ell-1}b_k\Big]+b_n\otimes a_1b_k
\\
&=\frac{1}{2}\big(b_k\otimes \Phi_2+b_k\Phi_2\otimes e_2\big)\,,
\end{align*}
\end{small}%
where in the last identity we used \eqref{eq:Convenient-formula-Euler-cont-b}.

Now, we shall show the validity of \eqref{Phim} for $b_{n-1}$. By \eqref{eq:Convenient-formula-Euler-cont-a} and the left Leibniz rule \eqref{Eq:inder} of the double bracket, we can write
\begin{small}
\begin{align*}
&\lr{\Phi_2,b_{n-1}}=\dgal{\Big((a_1,\dots,b_{n-1})(a_n,b_n)+\sum^n_{\ell=3}\Big[(a_1,\dots,b_{\ell-2})a_{\ell-1}b_n\Big]+a_1b_n\Big),b_{n-1}}
\\
&=(a_1,\dots,b_{n-1})*\lr{(a_n,b_n),b_{n-1}}+\lr{(a_1,\dots,b_{n-1}),b_{n-1}}*(a_n,b_n)
\\
&\quad +\sum^n_{\ell=3}\Big[(a_1,\dots,b_{\ell-2})a_{\ell-1}*\lr{b_n,b_{n-1}}\Big]
\\
&\quad +\sum^{n-1}_{\ell=3}\Big[(a_1,\dots,b_{\ell-2})*\lr{a_{\ell-1},b_{n-1}}*b_n\Big]+(a_1,\dots,b_{n-2})*\lr{a_{n-1},b_{n-1}}*b_n
\\
&\quad +\sum^n_{\ell=3}\Big[\lr{(a_1,\dots,b_{\ell-2}),b_{n-1}}*a_{\ell-1}b_n\Big]+a_1*\lr{b_n,b_{n-1}}+\lr{a_1,b_{n-1}}*b_n\,.
\end{align*}
\end{small}%
Applying Lemma \ref{lem:technic-lemma}(ii) in the first term, the inductive hypothesis in the second one, and Lemma \ref{lem:technic-lemma-no-inductive}(i), we obtain
\begin{small}
\begin{align*}
&=\frac{1}{2}\Big(b_{n-1}\otimes (a_1,\dots,b_{n-1})(a_n,b_n)-b_{n-1}(a_n,b_n)\otimes (a_1,\dots,b_{n-1})\Big)-b_n\otimes (a_1,\dots,b_{n-1})
\\
&\quad  +\frac{1}{2}\Big(b_{n-1}(a_n,b_n)\otimes (a_1,\dots,b_{n-1})+b_{n-1}(a_1,\dots,b_{n-1})(a_n,b_n)\otimes e_2\Big)
\\
&\quad +\frac{1}{2}\sum^n_{\ell=3}\Big[b_n\otimes (a_1,\dots,b_{\ell-2})a_{\ell-1}b_{n-1}+b_{n-1}\otimes (a_1,\dots,b_{\ell-2})a_{\ell-1}b_n\Big]
\\
&\quad +\frac{1}{2}\sum^{n-1}_{\ell=3}\Big[b_n\otimes (a_1,\dots,b_{\ell-2})a_{\ell-1}b_{n-1}+b_{n-1}a_{\ell-1}b_n\otimes (a_1,\dots,b_{\ell-2})\Big]
\\
&\quad + \frac{1}{2}\Big(b_{n-1}a_{n-1}b_n\otimes (a_1,\dots,b_{n-2})+b_n\otimes (a_1,\dots,b_{n-2})a_{n-1}b_{n-1}\Big)+b_n\otimes (a_1,\dots,b_{n-2})
\\
&\quad +\frac{1}{2}\sum^n_{\ell=3}\Big[b_{n-1}(a_1,\dots,b_{\ell-2})a_{\ell-1}b_n\otimes e_2-b_{n-1}a_{\ell-1}b_n\otimes (a_1,\dots,b_{\ell-2})\Big]
\\
&\quad +\frac{1}{2}\Big(b_n\otimes a_1b_{n-1}+b_{n-1}\otimes a_1b_n\Big)+\frac{1}{2}\Big(b_n\otimes a_1b_{n-1}+b_{n-1}a_1b_n\otimes e_2\Big)\,.
\end{align*}
\end{small}%
We can cancel some terms to obtain
\begin{small}
\begin{equation*}
\begin{split}
&=\frac{1}{2}\Big(b_{n-1}\otimes(a_1,\dots,b_{n-1})(a_n,b_n)+\sum^n_{\ell=3}\Big[b_{n-1}\otimes (a_1,\dots,b_{\ell-2})a_{\ell-1}b_n\Big]+b_{n-1}\otimes a_1b_n\Big)
\\
&\quad +\frac{1}{2}\Big(b_{n-1}(a_1,\dots,b_{n-1})(a_n,b_n)\otimes e_2
+\sum^n_{\ell=3}\Big[b_{n-1}(a_1,\dots,b_{\ell-2})a_{\ell-1}b_n\otimes e_2
+b_{n-1}a_1b_n\otimes e_2\Big)
\\
&\quad -b_n\otimes (a_1,\dots,b_{n-1})+b_n\otimes (a_1,\dots,b_{n-2})(a_{n-1},b_{n-1})+\sum^{n-1}_{\ell=3}\Big[b_n\otimes (a_1,\dots,b_{\ell-2})a_{\ell-1}b_{n-1}\Big]
\\
&\quad +b_n\otimes a_1b_{n-1}
\\
&=\frac{1}{2}\big(b_{n-1}\otimes \Phi_2+b_{n-1}\Phi_2\otimes e_2\big)\,,
\end{split}
\end{equation*}
\end{small}%
which is tantamount to \eqref{eq:moment-map-b}.

Finally, we shall prove that \eqref{Phim} holds for $b_n$ via \eqref{eq:moment-map-b}. Then,
\begin{small}
\begin{align*}
&\lr{\Phi_2,b_{n}}=\dgal{\Big((a_1,\dots,b_{n-1})(a_n,b_n)+\sum^n_{\ell=3}\Big[(a_1,\dots,b_{\ell-2})a_{\ell-1}b_n\Big]+a_1b_n\Big),b_{n}}
\\
&=(a_1,\dots,b_{n-1})*\lr{(a_n,b_n),b_{n}}+\lr{(a_1,\dots,b_{n-1}),b_{n}}*(a_n,b_n)
\\
&\quad +\sum^{n}_{\ell=3}\Big[(a_1,\dots,b_{\ell-2})*\lr{a_{\ell-1},b_{n}}*b_n+\lr{(a_1,\dots,b_{\ell-2}),b_{n}}*a_{\ell-1}b_n\Big]+\lr{a_1,b_{n}}*b_n\,,
\end{align*}
\end{small}%
where we used that $\lr{b_n,b_n}=0$. Now, by Lemma \ref{lem:technic-lemma}(ii) and Lemma \ref{lem:technic-lemma-no-inductive}(i), we have
\begin{small}
\begin{align*}
&=\frac{1}{2}\Big(b_n\otimes (a_1,\dots,b_{n-1})(a_n,b_n)+b_n(a_n,b_n)\otimes (a_1,\dots, b_{n-1})\Big)
\\
&\quad +\frac{1}{2}\Big(b_n(a_1,\dots,b_{n-1})(a_n,b_n)\otimes e_2-b_n(a_n,b_n)\otimes (a_1,\dots,b_{n-1})\Big)
\\
&\quad +\frac{1}{2}\sum^n_{\ell=3}\Big[b_n\otimes (a_1,\dots,b_{\ell-2})a_{\ell-1}
b_n+b_na_{\ell-1}b_n\otimes (a_1,\dots,b_{\ell-2})\Big]
\\
&\quad +\frac{1}{2}\sum^n_{\ell=3}\Big[b_n(a_1,\dots,b_{\ell-2})a_{\ell-1}b_n\otimes e_2-b_na_{\ell-1}b_n\otimes (a_1,\dots,b_{\ell-2})\Big]
\\
&\quad +\frac{1}{2}\Big(b_n\otimes a_1b_n+b_na_1b_n\otimes e_2\Big)
\\
&=\frac{1}{2}\big(b_{n}\otimes \Phi_2+b_{n}\Phi_2\otimes e_2\big)\,,
\end{align*}
\end{small}%
where, once again, in the last identity we applied \eqref{eq:Convenient-formula-Euler-cont-b}.
To sum up, so far, we proved that \eqref{Phim} holds for $\Phi_2$ and $b_i\in\{b_1,\dots,b_n\}$.\\


Next, let $a_i\in\{a_1,\dots,a_n\}$. Then, since $e_ie_j=\delta_{i,j}e_i$ for $i,j\in\{1,2\}$, \eqref{Phim} reduces to
\begin{equation}
\begin{aligned}
\lr{&\Phi_2,a_i}=-\frac{1}{2}\Big(e_2\otimes \Phi_2a_i+\Phi_2\otimes a_i\Big)
\\
&=-\frac{1}{2}\Big(e_2\otimes(a_1,\dots ,b_{n-1})(a_n,b_n)a_i+ \sum^n_{\ell=3} \Big[e_2\otimes(a_1,\dots,b_{\ell-2})a_{\ell-1}b_na_i\Big]+e_2\otimes a_1b_na_i
\\
&\quad + (a_1,\dots,b_{n-1})(a_n,b_n)\otimes a_i+\sum^n_{\ell=3}\Big[(a_1,\dots,b_{\ell-2})a_{\ell-1}b_n\otimes a_i\Big]+a_1b_n\otimes a_i\Big)\,,
\end{aligned}
\label{eq:moment-map-a}
\end{equation}
where we used \eqref{eq:Convenient-formula-Euler-cont-b}.

Now, applying the left Leibniz rule \eqref{Eq:inder} and using that $\lr{a_1,a_1}=0$, we can write
\begin{small}
\begin{align*}
&\lr{\Phi_2,a_1}=\dgal{\Big((a_1,\dots,b_{n-1})(a_n,b_n)+\sum^n_{\ell=3}\Big[(a_1,\dots,b_{\ell-2})a_{\ell-1}b_n\Big]+a_1b_n\Big),a_1}
\\
&=(a_1,\dots,b_{n-1})*\lr{(a_n,b_n),a_1}+\lr{(a_1,\dots, b_{n-1}),a_1}*(a_n,b_n)
\\
&\quad+ \sum^n_{\ell=3}\Big[(a_1,\dots,b_{\ell-2})a_{\ell-1}*\lr{b_n,a_1}
+(a_1,\dots,b_{\ell-2})*\lr{a_{\ell-1},a_1}*b_n\Big]
\\
&\quad + \sum^n_{\ell=3}\Big[ \lr{(a_1,\dots,b_{\ell-2}),a_1}*a_{\ell-1}b_n\Big]
+a_1*\lr{b_n,a_1}\,.
\end{align*}
\end{small}%
Using Lemma \ref{lem:technic-lemma}(i) in the first term and the inductive hypothesis \eqref{eq:inductive-hypothesis.a} in the second and fifth terms, 
\begin{small}
\begin{align*}
&=\frac{1}{2}\Big((a_n,b_n)\otimes (a_1,\dots,b_{n-1})a_1-e_2\otimes (a_1,\dots,b_{n-1})(a_n,b_n)a_1\Big)
\\
&\quad -\frac{1}{2}\Big((a_n,b_n)\otimes (a_1,\dots,b_{n-1})a_1+(a_1,\dots,b_{n-1})(a_n,b_n)\otimes a_1\Big)
\\
&\quad -\frac{1}{2}\sum^n_{\ell=3}\Big[a_1b_n\otimes (a_1,\dots,b_{\ell-2})a_{\ell-1}+e_2\otimes (a_1,\dots,b_{\ell-2})a_{\ell-1}b_na_1\Big]
\\
&\quad  +\frac{1}{2}\sum^n_{\ell=3}\Big[a_{\ell-1}b_n\otimes (a_1,\dots,b_{\ell-2})a_1+a_1b_n\otimes (a_1,\dots,b_{\ell-2})a_{\ell-1}\Big]
\\
&\quad -\frac{1}{2}\sum^n_{\ell=3}\Big[a_{\ell-1}b_n\otimes (a_1,\dots,b_{\ell-2})a_1+(a_1,\dots,b_{\ell-2})a_{\ell-1}b_n\otimes a_1 \Big]
\\
&\quad 
-\frac{1}{2}\Big(a_1b_n\otimes a_1+e_2\otimes a_1b_na_1\Big)
\\
&=-\frac{1}{2}\Big(e_2\otimes \Phi_2a_1+\Phi_2\otimes a_1\Big)\,,
\end{align*}
\end{small}%
where we used \eqref{eq:Convenient-formula-Euler-cont-b} in the last identity. Thus, \eqref{Phim} holds for $a_1$.

Next, we need to prove first that \eqref{Phim} holds for $a_2,\dots, a_{n-2}$. Since the proof is similar, we are going to fix $k\in\{2,\dots, n-2\}$, and we shall show that \eqref{Phim} is satisfied for the arrow $a_k$. By \eqref{eq:Convenient-formula-Euler-cont-b},
\begin{small}
\begin{align*}
\lr{\Phi_2,a_k}&=\dgal{\Big((a_1,\dots, b_{n-1})(a_n,b_n)+\sum^n_{\ell=3}\Big[(a_1,\dots,b_{\ell-2})a_{\ell-1}b_n\Big]+a_1b_n\Big),a_k}
\\
&=(a_1,\dots,b_{n-1})*\lr{(a_n,b_n),a_k}+\lr{(a_1,\dots,b_{n-1}),a_k}*(a_n,b_n)
\\
&\quad +\sum^n_{\ell=3}\Big[(a_1,\dots,b_{\ell-2})a_{\ell-1}*\lr{b_n,a_k}\Big]
 +\sum^k_{\ell=3}\Big[(a_1,\dots,b_{\ell-2})*\lr{a_{\ell-1},a_k}*b_n\Big]
\\
&\quad +\sum^n_{\ell=k+2}\Big[(a_1,\dots,b_{\ell-2})*\lr{a_{\ell-1},a_k}*b_n\Big]
 +\sum^k_{\ell=3}\Big[\lr{(a_1,\dots,b_{\ell-2}),a_k}*a_{\ell-1}b_n\Big]
\\
&\quad
+\lr{(a_1,\dots,b_{k-1}),a_k}*a_kb_n
+\sum^n_{\ell=k+2}\Big[\lr{(a_1,\dots,b_{\ell-2}),a_k}*a_{\ell-1}b_n\Big]
\\
&\quad
 +a_1*\lr{b_n,a_k}+\lr{a_1,a_k}*b_n\,,
\end{align*}
\end{small}%
where we used that $\lr{a_k,a_k}=0$. Now, we use Lemma \ref{lem:technic-lemma}(i) in the first summand, the inductive hypothesis in the second and eighth terms, Lemma \ref{lem:technic-lemma-no-inductive}(ii) in the sixth term and Lemma \ref{lem:technic-lemma-no-inductive}(iii) in the seventh summand.
\begin{small}
\begin{align*}
&=\frac{1}{2}\Big((a_n,b_n)\otimes (a_1,\dots,b_{n-1})a_k-e_2\otimes (a_1,\dots,b_{n-1})(a_n,b_n)a_k\Big)
\\
&\quad -\frac{1}{2}\Big((a_1,\dots,b_{n-1})(a_n,b_n)\otimes a_k+(a_n,b_n)\otimes (a_1,\dots,b_{n-1})a_k\Big)
\\
&\quad -\frac{1}{2}\sum^n_{\ell=3}\Big[a_kb_n\otimes (a_1,\dots,b_{\ell-2})a_{\ell-1}+e_2\otimes (a_1,\dots,b_{\ell-2})a_{\ell-1}b_na_k\Big]
\\
&\quad -\frac{1}{2}\sum^k_{\ell=3}\Big[a_{\ell-1}b_n\otimes (a_1,\dots,b_{\ell-2})a_k+a_kb_n\otimes (a_1,\dots,b_{\ell-2})a_{\ell-1}\Big]
\\
&\quad +\frac{1}{2}\sum^n_{\ell=k+2}\Big[a_{\ell-1}b_n\otimes (a_1,\dots,b_{\ell-2})a_k+a_kb_n\otimes (a_1,\dots,b_{\ell-2})a_{\ell-1}\Big]
\\
&\quad +\frac{1}{2}\sum^k_{\ell=3}\Big[a_{\ell-1}b_n\otimes (a_1,\dots,b_{\ell-2})a_k-(a_1,\dots,b_{\ell-2})a_{\ell-1}b_n\otimes a_k\Big]
\\
&\quad +\frac{1}{2}\Big(a_kb_n\otimes (a_1,\dots,b_{k-1})a_k-(a_1,\dots,b_{k-1})a_kb_n\otimes a_k\Big)+a_kb_n\otimes (a_1,\dots,a_{k-1})
\\
&\quad-\frac{1}{2}\sum^n_{\ell=k+2}\Big[a_{\ell-1}b_n\otimes (a_1,\dots,b_{\ell-2})a_k+(a_1,\dots,b_{\ell-2})a_{\ell-1}b_n\otimes a_k\Big]
\\
&\quad -\frac{1}{2}\Big(a_kb_n\otimes a_1+e_2\otimes a_1b_na_k\Big)-\frac{1}{2}\Big(a_1b_n\otimes a_k+a_kb_n\otimes a_1\Big)\,.
\end{align*}
\end{small}%
Some terms cancel each other, giving
\begin{small}
\begin{align*}
&=-\frac{1}{2}\Big(e_2\otimes (a_1,\dots,b_{n-1})(a_n,b_n)a_k+\sum^n_{\ell=3}\Big[e_2\otimes (a_1,\dots,b_{\ell-2})a_{\ell-1}b_na_k\Big]+e_2\otimes a_1b_na_k\Big)
\\
&\quad -\frac{1}{2}\Big((a_1,\dots,b_{n-1})(a_n,b_n)\otimes a_k+\sum^n_{\ell=3}\Big[(a_1,\dots,b_{\ell-2})a_{\ell-1}b_n\otimes a_k\Big]+a_1b_n\otimes a_k\Big)
\\
&\quad +a_kb_n\otimes (a_1,\dots,a_{k-1})-\sum^k_{\ell=3}\Big[a_kb_n\otimes (a_1,\dots,b_{\ell-2})a_{\ell-1}\Big]-a_kb_n\otimes a_1
\\
&=-\frac{1}{2}\Big(e_2\otimes \Phi_2a_k+\Phi_2\otimes a_k\Big)\,.
\end{align*}
\end{small}%
Consequently, we can conclude that \eqref{Phim} holds for $a_2,\dots,a_{n-2}$.

Now, we proceed to prove that $a_{n-1}$ satisfies \eqref{Phim}, that is, we need to show that $2\lr{\Phi_2,a_{n-1}}=-\big(e_2\otimes \Phi_2a_{n-1}+\Phi_2\otimes a_{n-1}\big)$. By \eqref{eq:Convenient-formula-Euler-cont-b}, 
\begin{small}
\begin{align*}
\lr{\Phi_2,a_{n-1}}&=\dgal{\Big((a_1,\dots,b_{n-1})(a_n,b_n)+\sum^n_{\ell=3}\Big[(a_1,\dots,b_{\ell-2})a_{\ell-1}b_n\Big]+a_1b_n\Big),a_{n-1}}
\\
&=(a_1,\dots,b_{n-1})*\lr{(a_n,b_n),a_{n-1}}+\lr{(a_1,\dots,b_{n-1}),a_{n-1}}*(a_n,b_n)
\\
&\quad +\sum^n_{\ell=3}\Big[(a_1,\dots,b_{\ell-2})a_{\ell-1}*\lr{b_n,a_{n-1}}\Big]+\sum^{n-1}_{\ell=3}\Big[(a_1,\dots,b_{\ell-2})*\lr{a_{\ell-1},a_{n-1}}*b_n\Big]
\\
&\quad+ \sum^{n-1}_{\ell=3}\Big[\lr{(a_1,\dots,b_{\ell-2}),a_{n-1}}*a_{\ell-1}b_n\Big]+\lr{(a_1,\dots,b_{n-2}),a_{n-1}}*a_{n-1}b_n
\\
&\quad 
+ a_1*\lr{b_n,a_{n-1}}+\lr{a_1,a_{n-1}}*b_n\,,
\end{align*}
\end{small}%
where we used that $\lr{a_{n-1},a_{n-1}}=0$. Next, we apply Lemma \ref{lem:technic-lemma}(i) to the first summand, the inductive hypothesis to the second summand, Lemma \ref{lem:technic-lemma-no-inductive}(ii) to the fifth term, and Lemma \ref{lem:technic-lemma-no-inductive}(iii) in the sixth one.
\begin{small}
\begin{align*}
&=\frac{1}{2}\Big((a_n,b_n)\otimes (a_1,\dots,b_{n-1})a_{n-1}-e_2\otimes (a_1,\dots,b_{n-1})(a_n,b_n)a_{n-1}\Big)
\\
&\quad -\frac{1}{2}\Big((a_n,b_n)\otimes (a_1,\dots,b_{n-1})a_{n-1}+(a_1,\dots,b_{n-1})(a_n,b_n)\otimes a_{n-1}\Big)
\\
&\quad -\frac{1}{2}\sum^{n}_{\ell=3}\Big[a_{n-1}b_n\otimes (a_1,\dots,b_{\ell-2})a_{\ell-1}+e_2\otimes (a_1,\dots,b_{\ell-2})a_{\ell-1}b_na_{n-1}\Big]
\\
&\quad -\frac{1}{2}\sum^{n-1}_{\ell=3}\Big[a_{\ell-1}b_n\otimes (a_1,\dots,b_{\ell-2})a_{n-1}+a_{n-1}b_n\otimes (a_1,\dots, b_{\ell-2})a_{\ell-1}\Big]
\\
&\quad +\frac{1}{2}\sum^{n-1}_{\ell=3}\Big[a_{\ell-1}b_n\otimes (a_1,\dots,b_{\ell-2})a_{n-1}-(a_1,\dots,b_{\ell-2})a_{\ell-1}b_n\otimes a_{n-1}\Big]
\\
&\quad +\frac{1}{2}\Big(a_{n-1}b_n\otimes (a_1,\dots,b_{n-2})a_{n-1}-(a_1,\dots,b_{n-2})a_{n-1}b_n\otimes a_{n-1}\Big)+a_{n-1}b_n\otimes (a_1,\dots, a_{n-2})
\\
&\quad -\frac{1}{2}\Big(a_{n-1}b_n\otimes a_1+e_2\otimes a_1b_na_{n-1}\Big)-\frac{1}{2}\Big(a_1b_n\otimes a_{n-1}+a_{n-1}b_n\otimes a_1\Big)\,.
\end{align*}
\end{small}%
If we cancel some terms and use both \eqref{eq:Convenient-formula-Euler-cont-b} and \eqref{eq:Convenient-formula-Euler-cont-a}, we obtain
\begin{small}
\begin{align*}
&=-\frac{1}{2}\Big(e_2\otimes (a_1,\dots,b_{n-1})(a_n,b_n)a_{n-1}
+\sum^n_{\ell=3}\Big[e_2\otimes(a_1,\dots,b_{\ell-2})a_{\ell-1}b_na_{n-1}\Big]
+e_2\otimes a_1b_na_{n-1} \Big)
\\
&\quad -\frac{1}{2} \Big( (a_1,\dots,b_{n-1})(a_n,b_n)\otimes a_{n-1}
+\sum^n_{\ell=3}\Big[(a_1,\dots,b_{\ell-2})a_{\ell-1}b_n\otimes a_{n-1}\Big]
+ a_1b_n\otimes a_{n-1} \Big)
\\
&\quad +a_{n-1}b_n\otimes (a_1,\dots,a_{n-2})-\sum^{n-1}_{\ell=3}\Big[a_{n-1}b_n\otimes (a_1,\dots,b_{\ell-2})a_{\ell-1}\Big]- a_{n-1}b_n\otimes a_1
\\
&=-\frac{1}{2}\Big(e_2\otimes \Phi_2a_{n-1}+\Phi_2\otimes a_{n-1}\Big)\,.
\end{align*}
\end{small}%


Finally, we check \eqref{Phim} for $a_{n}$ via \eqref{eq:moment-map-a}. Once again, by \eqref{eq:Convenient-formula-Euler-cont-b},
\begin{small}
\begin{align*}
&\lr{\Phi_2,a_{n}}=\dgal{\Big((a_1,\dots,b_{n-1})(a_n,b_n)+\sum^n_{\ell=3}\Big[(a_1,\dots,b_{\ell-2})a_{\ell-1}b_n\Big]+a_1b_n\Big),a_{n}}
\\
&=(a_1,\dots,b_{n-1})*\lr{(a_n,b_n),a_n}+\lr{(a_1,\dots,b_{n-1}),a_n}*(a_n,b_n)
\\
&\quad  +\sum^n_{\ell=3}\Big[(a_1,\dots,b_{\ell-2})a_{\ell-1}*\lr{b_n,a_n}+(a_1,\dots,b_{\ell-2})*\lr{a_{\ell-1},a_n}*b_n\Big]
\\
&\quad + \sum^n_{\ell=3}\Big[\lr{(a_1,\dots,b_{\ell-2}),a_n}*a_{\ell-1}b_n\Big]+ a_1*\lr{b_n,a_n}+\lr{a_1,a_n}*b_n\,.
\end{align*}
\end{small}%
Now, applying Lemma \ref{lem:technic-lemma}(i) to the first summand, Lemma \ref{lem:technic-lemma-no-inductive}(iii) in the second term, 
and Lemma \ref{lem:technic-lemma-no-inductive}(ii) in the fifth term, we get
\begin{small}
\begin{align*}
&=-\frac{1}{2}\Big(e_2\otimes (a_1,\dots ,b_{n-1})(a_n,b_n)a_n+(a_n,b_n)\otimes (a_1,\dots,b_{n-1})a_n\Big)
\\
&\quad +\frac{1}{2}\Big((a_n,b_n)\otimes (a_1,\dots,b_{n-1})a_n-(a_1,\dots,b_{n-1})(a_n,b_n)\otimes a_n\Big)+(a_n,b_n)\otimes (a_1,\dots,a_{n-1})
\\
&\quad -\sum^n_{\ell=3}\Big[\frac{1}{2}\big(e_2\otimes (a_1,\dots,b_{\ell-2})a_{\ell-1}b_na_n+a_nb_n\otimes (a_1,\dots,b_{\ell-2})a_{\ell-1}\big)+e_2\otimes (a_1,\dots,b_{\ell-2})a_{\ell-1}\Big]
\\
&\quad -\frac{1}{2}\sum^n_{\ell=3}\Big[a_{\ell-1}b_n\otimes (a_1,\dots,b_{\ell-2})a_n+a_nb_n\otimes (a_1,\dots,b_{\ell-2})a_{\ell-1}\Big]
\\
&\quad +\frac{1}{2}\sum^n_{\ell=3}\Big[a_{\ell-1}b_n\otimes (a_1,\dots,b_{\ell-2})a_n-(a_1,\dots,b_{\ell-2})a_{\ell-1}b_n\otimes a_n\Big]
\\
&\quad -\frac{1}{2}\Big(e_2\otimes a_1b_na_n+a_nb_n\otimes a_1\Big)-e_2\otimes a_1
 -\frac{1}{2}\Big(a_1b_n\otimes a_n+a_nb_n\otimes a_1\Big)\,.
\end{align*}
\end{small}%
Using that $(a_n,b_n)=a_nb_n+e_2$, \eqref{eq:Convenient-formula-Euler-cont-b}, and \eqref{eq:Convenient-formula-Euler-cont-a}, we finally obtain
\begin{small}
\begin{align*}
&=-\frac{1}{2}\Big(e_2\otimes (a_1,\dots,b_{n-1})(a_n,b_n)a_n+\sum^n_{\ell=3}\Big[e_2\otimes(a_1,\dots,b_{\ell-2})a_{\ell-1}b_na_n\Big]+e_2\otimes a_1b_na_n\Big)
\\
&\quad -\frac{1}{2} \Big( (a_1,\dots,b_{n-1})(a_n,b_n)\otimes a_n+\sum^n_{\ell=3}\Big[(a_1,\dots,b_{\ell-2})a_{\ell-1}b_n\otimes a_n\Big]+a_1b_n\otimes a_n\Big)
\\
&\quad+ (a_n,b_n)\otimes (a_1,\dots,a_{n-1})-\sum^n_{\ell=3}\Big[(a_n,b_n)\otimes (a_1,\dots,b_{\ell-2})a_{\ell-1}\Big]-(a_n,b_n)\otimes a_1
\\
&=-\frac{1}{2}\Big(e_2\otimes \Phi_2a_{n}+\Phi_2\otimes a_{n}\Big)\,.
\end{align*}
\end{small}

To sum up, we proved that $\Phi_2=(a_1,\dots, b_n)$ satisfies \eqref{Phim} in $\Bc(\Gamma_n)$. 
Now, we need to prove that $\Phi_1=(b_n,\dots, a_1)^{-1}\in e_1 \Bc(\Gamma_n) e_1$ satisfies \eqref{Phim} as well. The automorphism $S$ of order two (see \eqref{eq:automorphism-of-order-two}) has the effect $\lr{-,-}\mapsto -\lr{-,-}$ and $(a_1,\dots,b_n)\mapsto (b_n,\dots, a_1)$. By applying $S$ to the just proved identity $2\lr{(a_1,\dots, b_n),a_i}=-\big(e_2\otimes \Phi_2a_{i}+\Phi_2\otimes a_{i}\big)$---see \eqref{eq:moment-map-a}, for $1\leq i\leq n$, we obtain
\[
\lr{(b_n,\dots, a_1),b_{n+1-i}}=\frac{1}{2}\Big(e_1\otimes (b_n,\dots,a_1)b_{n+1-i}+(b_n,\dots, a_1)\otimes b_{n+1-i}\Big)\,,
\]
for $1\leq i\leq n$. Then, by \eqref{Eq:inder} and noting that $\Phi_1(b_n,\dots,a_1)=(b_n,\dots,a_1)\Phi_1=e_1$, we finally obtain
\begin{align*}
\lr{\Phi_1,b_{n+1-i}}&=-\Phi_1*\lr{(b_n,\dots, a_1),b_{n+1-i}}*\Phi_1
\\
&=-\Phi_1*\Big(\frac{1}{2}\big(e_1\otimes (b_n,\dots,a_1)b_{n+1-i}+(b_n,\dots, a_1)\otimes b_{n+1-i}\big)\Big)*\Phi_1
\\
&=-\frac{1}{2}\Big(\Phi_1\otimes b_{n+1-i}+e_1\otimes \Phi_1b_{n+1-i}\Big)\,,
\end{align*}
for all $1\leq i\leq n$, which is tantamount to \eqref{Phim}, since $b_{n+1-i}e_1=0$. Similarly, applying the automorphism $S$ to \eqref{eq:moment-map-b}, we get the identity
\[
-\lr{(b_n,\dots, a_1),a_{n+1-i}}=\frac{1}{2}\Big(a_{n+1-i}\otimes (b_n,\dots, a_1)+a_{n+1-i}(b_n,\dots, a_1)\otimes e_1\Big)\,.
\]
By \eqref{Eq:inder} and the fact $e_1a_{n+1-i}=0$, we have
\begin{align*}
\lr{\Phi_1,a_{n+1-i}}&=-\Phi_1*\lr{(b_n,\dots, a_1),a_{n+1-i}}*\Phi_1
\\
&=\Phi_1*\Big(\frac{1}{2}\big(a_{n+1-i}\otimes (b_n,\dots, a_1)+a_{n+1-i}(b_n,\dots, a_1)\otimes e_1\big)\Big)*\Phi_1
\\
&\frac{1}{2}\Big(a_{n+1-i}\Phi_1\otimes e_1+a_{n+1-i}\otimes \Phi_1\Big)\,,
\end{align*}
as we wished. So, we conclude that \eqref{Phim} holds for $\Phi_1$; therefore $\Phi$, as defined in \eqref{eq:mult-moment-map-statement}, is a multiplicative moment map for the double quasi-Poisson bracket \eqref{eq:Euler-double-bracket}.
This finishes the proof of Theorem \ref{thm:Euler-quasi-Hamiltonian-algebra}.


\section{Factorisation after localisation at several Euler continuants} 
\label{sec:factorisation-after-localiz}

\subsection{The result}
Recall from \ref{ss:Result} that $\Bc(\Gamma_n)$ denotes the algebra obtained from $\kk\overline{\Gamma}_n$ after  localisation at the Euler continuants \eqref{Eq:invEuler}. We now consider the algebra $\Bloc$ obtained by further localising $\Bc(\Gamma_n)$ at the elements 
\begin{equation}\label{Eq:invLoc}
\begin{aligned}
  &(a_1,b_1)\,,\ldots\,,(a_1,b_1,\dots,a_{n-1},b_{n-1})\in e_2(\kk\overline{\Gamma}_n)e_2\,, \\
  &(b_1,a_1)\,,\ldots,(b_{n-1},a_{n-1}\,,\dots,b_1,a_1)\in e_1(\kk\overline{\Gamma}_n)e_1\,.
\end{aligned}
\end{equation}
In other words, $\Bloc$ is obtained from $\kk \overline{\Gamma}_n$ by requiring that the following elements are inverted (in the sense of \ref{ss:Result}):
\begin{equation} \label{Eq:invLocB}
  (a_1,b_1,\dots,a_k,b_k)\in e_2(\kk\overline{\Gamma}_n)e_2\,,\quad 
  (b_k,a_k,\dots,b_1,a_1)\in e_1(\kk\overline{\Gamma}_n)e_1\,,\quad k=1,\ldots,n\,.
\end{equation}
The  structure on $\Bc(\Gamma_n)$ given in Theorem \ref{Thm:MAIN} descends to $\Bloc$ by localisation. 
\begin{thm} \label{Thm:Factorisation}
 There exists an isomorphism of Hamiltonian double quasi-Poisson algebras between $\Bloc$ endowed with the structure induced by localisation of $\Bc(\Gamma_n)$, and the fusion algebra $\Afus$ obtained by identifying the idempotents in $n$ copies of $\Bc(\Gamma_1)$ endowed with the structure from Theorem \ref{Thm:MAIN}. 
 \end{thm}

The algebra $\Afus$ is obtained by fusion which, as we recalled in \ref{sec:ssec-Fusion}, is a method to get a double quasi-Poisson bracket and a multiplicative moment map from a Hamiltonian double quasi-Poisson algebra by  identifying several idempotents. We describe the structure of $\Afus$ in \ref{ss:Afus}. In particular, we will notice in Remark \ref{Rem:RelVDB} that the algebra $\Afus$ is, by construction, an example of Hamiltonian double quasi-Poisson algebra associated with the quiver $\Gamma_n$ (where we see $\overline{\Gamma}_n$ as its double quiver for $b_k=a_k^\ast$) by Van den Bergh, see Remark \ref{rem:VdB-double}. Thus, the next result is a direct consequence of Theorem \ref{Thm:Factorisation}. 

\begin{cor}\label{coro:coro-factorisation}
After localisation at the Euler continuants \eqref{Eq:invLocB}, 
the Boalch algebra $\Bc(\Gamma_n)$ associated with the quiver $\Gamma_n$ is isomorphic as a Hamiltonian double quasi-Poisson algebra to the algebra $\Ac(\Gamma_n)$ associated with the quiver $\Gamma_n$ by Van den Bergh.
\end{cor}

\begin{rem} \label{Rem:factoris}
Theorem \ref{Thm:Factorisation} and Corollary \ref{coro:coro-factorisation}  are motivated by the fact that, at the level of representation spaces, $\Rep(\Bc(\Gamma_n),(d_1,d_2))$ can be seen as a ``partial compactification'' of   $\Rep(\Ac(\Gamma_n),(d_1,d_2))$. We are grateful to Philip Boalch for mentioning this geometric property which can be found in \cite[\S5.3--5.5]{Paluba}.
\end{rem}

\subsection{The algebra \texorpdfstring{$\Afus$}{Afus}} \label{ss:Afus}

We consider $n$ copies of the algebra $\Bc(\Gamma_1)$, that we denote by $\Ac(\Gamma_1)$ for the rest of this subsection, since the Hamiltonian double quasi-Poisson structures of Theorems \ref{tm:VdB-mult-preproj-quasi-Hamilt} and \ref{thm:Euler-quasi-Hamiltonian-algebra} (with $n=1$) coincide. 
We respectively denote the two generators $a_1,b_1$ in the $\ell$-th copy as $c_\ell,d_\ell$, and the two idempotents $e_1,e_2$ in the $\ell$-th copy as $e_{1,\ell},e_{2,\ell}$. In particular, the continuants $(c_\ell,d_\ell)$ and $(d_\ell,c_\ell)$ are inverted in each copy. We can naturally see $\oplus_{\ell=1}^n \Ac(\Gamma_1)$ as a Hamiltonian double quasi-Poisson algebra over 
$\tilde{B}$, where $\tilde B = (\oplus_{\ell=1}^n \kk e_{1,\ell}) \bigoplus  (\oplus_{\ell=1}^n \kk e_{2,\ell})$; see \cite[Remark 2.13]{F19}. 

Next, we identify idempotents in  $\oplus_{\ell=1}^n \Ac(\Gamma_1)$. We let $e_1:=e_{1,1}$ and $e_2:=e_{2,1}$, and we inductively perform fusion (\ref{sec:ssec-Fusion}) of the idempotents $e_{1,\ell}$ onto $e_1$ and $e_{2,\ell}$ onto $e_2$ for $\ell=2,\ldots,n$. 
The algebra $\Afus$ obtained in this way can be described as being generated by the  orthogonal idempotents $e_1,e_2$ satisfying $e_1+e_2=1$, together with the elements 
\begin{equation} \label{Eq:Fusgen1}
 c_1=e_2c_1 e_1\,,\quad \ldots\,,\quad c_n=e_2c_n e_1\,;\quad d_1=e_1 d_1 e_2\,,\quad \ldots\,,\quad  d_n=e_1d_n e_2\,,
\end{equation}
and the following elements are inverted:  
\begin{equation} \label{Eq:invFus}
  (c_\ell,d_\ell)=e_2+c_\ell d_\ell \in e_2\Afus e_2\,,\quad 
  (d_\ell ,c_\ell)=e_1+d_\ell c_\ell\in e_1\Afus e_1\,,\quad \ell=1,\ldots,n\,.
\end{equation}
The choice of order in which we perform fusion gives a Hamiltonian double quasi-Poisson algebra structure on $\Afus$ by \cite[\S5.3]{VdB1}.
\begin{prop} \label{Pr:Fus-qHam}
The Hamiltonian double quasi-Poisson structure on $\Afus$ is given by 
the $B$-linear double bracket $\lr{-,-}^{(n)}$ which is defined on generators by  
\begin{subequations}
\label{eq:Fus-double-bracket}
\begin{align}
 \lr{c_i,c_j}^{(n)}&=
\begin{cases} 
-\frac{1}{2}\big(c_i\otimes c_j+c_j\otimes c_i\big)\,, &\mbox{if } i<j
 \\ 
  \;\;\;  0\,, &\mbox{if } i=j
 \\
 \;\;\;   \frac{1}{2}\big(c_i\otimes c_j+c_j\otimes c_i\big)\,, & \mbox{if } i>j
\end{cases};  \label{eq:Fus-double-bracket.a}
\\
 \lr{d_i,d_j}^{(n)}&=
\begin{cases} 
-\frac{1}{2}\big(d_i\otimes d_j+d_j\otimes d_i\big)\,, &\mbox{if } i<j
 \\ 
  \;\;\;  0\,, &\mbox{if } i=j
 \\
 \;\;\;  \frac{1}{2}\big(d_i\otimes d_j+d_j\otimes d_i\big)\,, & \mbox{if } i>j
\end{cases};  \label{eq:Fus-double-bracket.b}
\\
 \lr{c_i,d_j}^{(n)}&=
\begin{cases} 
 \;\;\;  \frac{1}{2}\big(e_1\otimes c_id_j+d_jc_i\otimes e_2\big)\,, &\mbox{if } i<j
 \\ 
 \;\;\;  \frac{1}{2}\big(e_1\otimes c_id_i+ d_ic_i\otimes e_2\big)+e_1\otimes e_2\,, &\mbox{if } i=j
 \\
- \frac{1}{2}\big(e_1\otimes c_id_j+d_jc_i\otimes e_2\big)\,, &\mbox{if } i>j
\end{cases};   \label{eq:Fus-double-bracket.c}
\end{align}
\end{subequations}
and which we extend to $\Afus$ by the Leibniz rule and cyclic antisymmetry. The multiplicative moment map $\Phi^{(n)}=\Phi_1^{(n)}+\Phi_2^{(n)}$ is given by 
\begin{equation} \label{Eq:Fus-momap}
 \Phi_1^{(n)}=(e_1+d_1 c_1)^{-1}\cdots (e_1+d_n c_n)^{-1},\,\quad
 \Phi_2^{(n)}=(e_2+c_1 d_1)\cdots(e_2+c_n d_n)\,.
\end{equation}
\end{prop}
\begin{proof}
 The Hamiltonian double quasi-Poisson structure on $\oplus_{\ell=1}^n \Ac(\Gamma_1)$ is given by the double quasi-Poisson bracket $\lr{-,-}^{\ssep}$ satisfying for $1\leq i,j\leq n$
 \begin{equation} \label{Eq:dbrSep}
 \begin{aligned}
    &\lr{c_i,c_j}^{\ssep}=0\,, \quad\lr{d_i,d_j}^{\ssep}=0\,, \\ 
  &\lr{c_i,d_j}^{\ssep}=\delta_{ij} \,e_{1,i}\otimes e_{2,i}+\frac12 \delta_{ij} (e_{1,i}\otimes c_i d_i + d_i c_i \otimes e_{2,i})\,,
 \end{aligned}
 \end{equation}
 and the multiplicative moment map $\Phi^{\ssep}$ for 
 \begin{equation}
  \Phi^{\ssep}=\sum_{\ell=1}^n (\Phi_{1,\ell}+\Phi_{2,\ell})\,, \quad \Phi_{1,\ell}:=(e_{1,\ell}+d_\ell c_\ell)^{-1},\quad  \Phi_{2,\ell}:=e_{2,\ell}+c_\ell d_\ell\,.
 \end{equation}
We inductively perform fusion of $e_{1,\ell}$ onto $e_1=e_{1,1}$, and of $e_{2,\ell}$ onto $e_2=e_{2,1}$ for $\ell=2,\ldots,n$. When such fusions have been performed for all $\ell=1,\ldots,k$, $2\leq k\leq n$, we have by Theorem \ref{Thm:IsoFusqHam} that the multiplicative moment map $\Phi^{(k)}$ becomes 
 \begin{equation}
 \begin{aligned}
  & \Phi^{(k)}=\Phi_1^{(k)}+ \Phi_2^{(k)} +\sum_{\ell=k+1}^n (\Phi_{1,\ell}+\Phi_{2,\ell})\,, \\
  &\Phi_{1}^{(k)}=(e_{1}+d_1 c_1)^{-1}\cdots (e_{1}+d_k c_k)^{-1},\quad \Phi_{2}^{(k)}=(e_{2}+c_1 d_1)\cdots (e_{2}+c_k d_k)\,.
 \end{aligned}
 \end{equation}
In particular, $\Phi^{(n)}_1=e_1 \Phi^{(n)}e_1$ and $\Phi^{(n)}_2=e_2\Phi^{(n)}e_2$ are the components of the multiplicative moment map on $\Afus$ by construction, and this is just \eqref{Eq:Fus-momap}.  
 
Using again  Theorem \ref{Thm:IsoFusqHam}, we get that the double quasi-Poisson bracket $\lr{-,-}^{(k)}$ obtained after fusing together the first $k$ copies of $\Ac(\Gamma_1)$ in $\oplus_{\ell=1}^n \Ac(\Gamma_1)$  is defined by adding an extra ``fusion'' double bracket $\lr{-,-}^{(k)}_{\fus}$ to the previously obtained one $\lr{-,-}^{(k-1)}$ (where $\lr{-,-}^{(0)}=\lr{-,-}^{\ssep}$). 
It can be  read from Lemma \ref{lem:fus-double-bracket} that the fusions of  $e_{1,k}$ onto $e_1$ and of $e_{2,k}$ onto $e_2$ result in the additional fusion double bracket $\lr{-,-}^{(k)}_{\fus}$ which satisfies the following relations 
\begin{subequations}
\begin{align}
  \lr{c_i,c_k}^{(k)}_{\fus}=&-\frac12 (c_i \otimes c_k + c_k \otimes c_i)\,, \quad i<k\,, \label{Eq:cckfus} \\
  \lr{d_i,d_k}^{(k)}_{\fus}=&-\frac12 (d_i \otimes d_k + d_k \otimes d_i)\,, \quad i<k\,,  \\
   \lr{c_i,d_k}^{(k)}_{\fus}=&\;\;\;\;\,\frac12 (e_1 \otimes c_i d_k + d_k c_i \otimes e_2)\,, \quad i<k\,, \\
  \lr{d_i,c_k}^{(k)}_{\fus}=&\;\;\;\;\,\frac12 (e_1 \otimes c_i d_k + d_k c_i \otimes e_2)\,, \quad i<k\,, \\
  \lr{f,f'}^{(k)}_{\fus}=&\;\;\;\;\, 0\,, \quad \text{if either }f,f'\in \{c_i,d_i \mid i \neq k\} \text{ or if }f,f'\in \{c_k,d_k\}\,.
\end{align}
\end{subequations}
Let us derive \eqref{Eq:cckfus}. When we fuse $e_{1,k}$ onto $e_1$, according to the terminology of \ref{sec:ssec-Fusion} the element $c_i$ (for $i<k$) is a generator of first type, and $c_k$ is a generator of third type. Hence this fusion results in the following added term in $\lr{c_i,c_k}^{(k)}_{\fus}$ by \eqref{tv}: 
\begin{equation}
 \frac12 (c_k c_i \otimes e_1 - c_k \otimes c_i e_1 ) = -\frac12 c_k\otimes c_i \,.
 \label{eq:aux-fusion-1}
\end{equation}
When we fuse $e_{2,k}$ onto $e_2$, the element $c_i$ (for $i<k$)  is again of first type, but $c_k$ is of second type, so we add the following term in $\lr{c_i,c_k}^{(k)}_{\fus}$ by \eqref{tu}: 
\begin{equation}
 \frac12 (e_1  \otimes c_i c_k - e_1 c_i \otimes c_k) = -\frac12 c_i \otimes c_k \,.
 \label{eq:aux-fusion-2}
\end{equation}
Putting together \eqref{eq:aux-fusion-1} and \eqref{eq:aux-fusion-2}, we get the two terms in \eqref{Eq:cckfus}. 

To conclude, note that the double bracket $\lr{-,-}^{(n)}$ on $\Afus$ is obtained by adding all the double brackets $\lr{-,-}_{\fus}^{(\ell)}$ for $2\leq \ell \leq n$ to \eqref{Eq:dbrSep}. It is easy to see that it coincides with \eqref{eq:Fus-double-bracket}. 
\end{proof}

\begin{rem} \label{Rem:RelVDB}
We can relate Proposition \ref{Pr:Fus-qHam} to Remark \ref{rem:VdB-double} as follows.  
Recall that the Hamiltonian double quasi-Poisson structure on $\Ac(\Gamma_1)$ is the  structure\footnote{\label{Footnote:Conv}As pointed out in Remark \ref{rem:convention-on-arrows}, we use the \emph{opposite} convention to \cite{VdB1} in order to write arrows/paths.} constructed by Van den Bergh on $\kk\overline{\Gamma}_1$ localised at $e_2+ab$ and $e_1+ba$. 
 
Denote by $\Gamma_n$ the quiver consisting of $n$ arrows $c_\ell:1\to 2$, and denote by $d_\ell=c_\ell^\ast:2\to 1$ the corresponding arrows in the double $\overline{\Gamma}_n$. For a fixed ordering of the arrows starting at each vertex, Van den Bergh obtained a Hamiltonian double quasi-Poisson structure on $\Ac(\Gamma_n)$, which is defined as the algebra $\kk \overline{\Gamma}_n$ localised at all the 
$e_2+c_\ell d_\ell$ and $e_1+d_\ell c_\ell$. This structure is obtained by fusion, starting with $n$ copies of $\kk\overline{\Gamma}_1$. 
By construction, it was noticed by Van den Bergh that the order in which we perform fusion can be encoded in an ordering taken at each vertex on the arrows which start at that vertex. In our case the Hamiltonian double quasi-Poisson structure for the orderings
\begin{equation}
 c_1<c_2<\ldots <c_n \quad \text{at the vertex }1\,, \qquad 
 d_1<d_2< \ldots <d_n \quad \text{at the vertex }2\,,
\end{equation}
is precisely the result obtained in Proposition \ref{Pr:Fus-qHam}. 
The comparison is easily made using the explicit form of Van den Bergh's double bracket written in \cite[Theorem 3.3]{F19} (keeping in mind Footnote \ref{Footnote:Conv}). Note that, up to isomorphism, Van den Bergh's Hamiltonian double quasi-Poisson algebra structure associated with $\Gamma_n$ only depends on the quiver seen as an undirected graph \cite[Theorem 4.12]{F20}.
\end{rem}
 
\begin{rem}
Fix $q=(q_1,q_2)\in (\kk^\times)^2$.  
Building on the previous remark, the quotient algebra $\Afus/(\Phi^{(n)}-q_1e_1-q_2 e_2)$ (for $\Phi^{(n)}=\Phi^{(n)}_1+\Phi^{(n)}_2$ defined in \eqref{Eq:Fus-momap}) is the multiplicative preprojective algebra 
$\Lambda^q(\Gamma_n)$  of the quiver $\Gamma_n$ as introduced by Crawley-Boevey and Shaw \cite{CBShaw}. 
Furthermore, by Theorem \ref{Thm:Factorisation}, we obtain that the fission algebra $\mathcal{F}^q(\Gamma_n)$ attached to $\Gamma_n$ is, after localisation at the elements \eqref{Eq:invLoc}, isomorphic to $\Lambda^q(\Gamma_n)$.  
\end{rem}

\subsection{Alternative description of the algebra \texorpdfstring{$\Bloc$}{Aloc}}  \label{ss:Aloc}

By construction, $\Bloc$ is generated by the  orthogonal idempotents $e_1,e_2$ satisfying $e_1+e_2=1$, together with the elements 
\begin{equation} \label{Eq:Locgen1}
 a_1=e_2a_1 e_1\,,\quad \ldots \,,\quad\, a_n=e_2a_n e_1\,;\quad b_1=e_1 b_1 e_2\,,\quad \ldots\,,\quad\,  b_n=e_1b_n e_2\,,
\end{equation}
while the elements in \eqref{Eq:invLocB} are inverted. By localisation, the double quasi-Poisson bracket and the multiplicative moment map can be written as in Theorem  \ref{Thm:MAIN} in terms of the generators \eqref{Eq:Locgen1}. 

Let us introduce the following elements of $\Bloc$:
\begin{equation}\label{Eq:Locgen2}
\begin{aligned} 
  b_\ell'&=b_\ell\,, & &\text{for }1\leq \ell\leq n\,, \\
 a_1'&=a_1\,, &
&a_\ell'=a_\ell+(a_1,\ldots,b_{\ell-1})^{-1}(a_1,\ldots,a_{\ell-1})\,,\quad \text{for }2\leq \ell\leq n\,.
\end{aligned}
\end{equation}
Though the formulae are cumbersome, we can see by induction that the elements \eqref{Eq:Locgen1} can be expressed in terms of the elements $\{a_\ell',b_\ell'\}$ in \eqref{Eq:Locgen2}. Hence the elements $\{a_\ell',b_\ell'\}$ can be used as a set of generators for $\Bloc$
(where the elements \eqref{Eq:invLocB} written in terms of $\{a_\ell',b_\ell'\}$ are inverted). 
We can also write 
\begin{equation} \label{Eq:Locgen3}
 a_\ell'=a_\ell+(a_{\ell-1},\ldots,a_{1}) (b_{\ell-1},\ldots,a_{1})^{-1}\,,\quad \text{for }2\leq \ell\leq n\,,
\end{equation}
due to Lemma \ref{lem:curious-identity}.

\begin{prop} \label{Pr:Alt-qHam}
 In terms of the generators $\{a_k',b_k'\mid 1\leq k \leq n\}$, the Hamiltonian double quasi-Poisson algebra structure on $\Bloc$ is given by 
the $B$-linear double bracket 
\begin{subequations}
\label{eq:Alt-double-bracket} 
\begin{align}
 \lr{a_i',a_j'}&=
\begin{cases} 
-\frac{1}{2}\big(a_i'\otimes a_j'+a_j'\otimes a_i'\big)\,, &\mbox{if } i<j
 \\ 
  \;\;\;  0\,, &\mbox{if } i=j
 \\
 \;\;\;   \frac{1}{2}\big(a_i'\otimes a_j'+a_j'\otimes a_i'\big)\,, & \mbox{if } i>j
\end{cases};  \label{eq:Alt-double-bracket.a}
\\
 \lr{b_i',b_j'}&=
\begin{cases} 
-\frac{1}{2}\big(b_i'\otimes b_j'+b_j'\otimes b_i'\big)\,, &\mbox{if } i<j
 \\ 
  \;\;\;  0\,, &\mbox{if } i=j
 \\
 \;\;\;  \frac{1}{2}\big(b_i'\otimes b_j'+b_j'\otimes b_i'\big)\,, & \mbox{if } i>j
\end{cases};  \label{eq:Alt-double-bracket.b}
\\
 \lr{a_i',b_j'}&=
\begin{cases} 
 \;\;\;  \frac{1}{2}\big(e_1\otimes a_i'b_j'+b_j'a_i'\otimes e_2\big)\,, &\mbox{if } i<j
 \\ 
 \;\;\;  \frac{1}{2}\big(e_1\otimes a_i'b_i'+ b_i'a_i'\otimes e_2\big)+e_1\otimes e_2\,, &\mbox{if } i=j
 \\
- \frac{1}{2}\big(e_1\otimes a_i'b_j'+b_j'a_i'\otimes e_2\big)\,, &\mbox{if } i>j
\end{cases};   \label{eq:Alt-double-bracket.c}
\end{align}
\end{subequations}
and the multiplicative moment map $\Phi=\Phi_1+\Phi_2$ for 
\begin{equation} \label{Eq:Alt-momap}
 \Phi_1=(e_1+b_1' a_1')^{-1}\cdots (e_1+b_n'a_n')^{-1},\qquad 
 \Phi_2=(e_2+a_1'b_1')\cdots(e_2+a_n' b_n')\,.
\end{equation}
\end{prop}
The proof of Proposition \ref{Pr:Alt-qHam} is presented in \ref{ss:Proof-Alt-qHam}.

\subsection{Proof of Theorem \ref{Thm:Factorisation}}
\label{sec:proof-of-thm-factorisation}

We consider the generators introduced in \ref{ss:Afus} for $\Afus$ and \ref{ss:Aloc} for $\Bloc$. We first note that we have a $B$-linear isomorphism given by 
\begin{equation} \label{eq:isomorphism-quasi-Ham-psi}
 \psi:\Bloc\to \Afus\,,\qquad \psi(a_\ell')=c_\ell\,,\,\quad \psi(b_\ell')=d_\ell\,,
\end{equation}
for $\ell=1,\ldots,n$.
To see that this map is well-defined, note that the invertibility of the elements \eqref{Eq:invLocB} in $\Bloc$ is equivalent to the invertibility of the elements $(a_\ell',b_\ell')=e_2+a_\ell'b_\ell'$ and $(b_\ell',a_\ell')=e_1+b_\ell' a_\ell'$ in $\Bloc$ for $1\leq \ell \leq n$ due to \eqref{Eq:contInd}, which we obtain as part of the proof of Proposition \ref{Pr:Alt-qHam}.   
But these elements are precisely mapped onto the invertible elements $(c_\ell,d_\ell)$ and $(d_\ell,c_\ell)$ in $\Afus$. 

Next, by comparing Propositions \ref{Pr:Fus-qHam} and \ref{Pr:Alt-qHam}, we can see that 
\begin{equation}
\begin{aligned}
& \lr{\psi(a_i'),\psi(a_j')}^{(n)}=\psi^{\otimes 2}(\lr{a_i',a_j'})\,, 
\\
&\lr{\psi(b_i'),\psi(b_j')}^{(n)}=\psi^{\otimes 2}(\lr{b_i',b_j'})\,, 
\\
&   \lr{\psi(a_i'),\psi(b_j')}^{(n)}=\psi^{\otimes 2}(\lr{a_i',b_j'})\,, \quad 
\end{aligned}
\end{equation}
for any $i,j=1,\ldots,n$. We thus get that, for all $a,b\in \Bloc$,
\begin{equation}
 \lr{\psi(a),\psi(b)}^{(n)}=\psi^{\otimes 2}(\lr{a,b})\,.
\end{equation}
It is also clear that $\psi(\Phi_1)=\Phi_1^{(n)}$ and $\psi(\Phi_2)=\Phi_2^{(n)}$, so $\psi$ is an isomorphism of Hamiltonian double quasi-Poisson algebras.

 \subsection{Proof of Proposition \ref{Pr:Alt-qHam}} \label{ss:Proof-Alt-qHam}

\subsubsection{Preparation} 

Recall that the double quasi-Poisson bracket on $\Bloc$ takes the form \eqref{eq:Euler-double-bracket} on the elements $\{a_i,b_i\}$ since it is obtained by localisation of $\Bc(\Gamma_n)$. 

\begin{lem} \label{Lem:prep-1}
 The double bracket on $\Bc(\Gamma_n)$ is such that
 \begin{subequations}
\label{eq:Lem-prep-1}
\begin{align}
\lr{(a_1,\dots, b_i),b_j}&=
\begin{cases} 
\;\;\; \frac12\Big(b_j(a_1,\dots, b_i)\otimes e_2-b_j\otimes (a_1,\dots, b_i)\Big) \,, &\mbox{if } i<j,  \\ 
\;\;\;\frac12\Big(b_j(a_1,\dots, b_i)\otimes e_2+b_j\otimes (a_1,\dots, b_i)\Big) \,,  &\mbox{if } i\geq j ,
\end{cases}  \label{eq:Lem-prep-1.a}
\\
\lr{(a_1,\dots, b_i),a_j}&=
\begin{cases} 
\;\;\; \frac12 \Big(e_2\otimes (a_1,\dots, b_i)a_j- (a_1,\dots, b_i)\otimes a_j\Big)\,, &\mbox{if } i<j\!-\!1 , \\ 
\;\;\; \frac{1}{2}\Big(e_2\otimes (a_1,\dots, b_i)a_{j}- (a_1,\dots, b_i)\otimes a_{j}\Big)& \\
\qquad\; +\,e_2\otimes (a_1,\dots, a_i)\,, &\mbox{if } i=j\!-\!1 ,\\
-\frac12 \Big(e_2\otimes (a_1,\dots, b_i)a_j+ (a_1,\dots, b_i)\otimes a_j\Big)\,, & \mbox{if } i\geq j.
\end{cases}  \label{eq:Lem-prep-1.b}
\end{align}
\end{subequations}
 In particular, these identities hold in $\Bloc$.
\end{lem}
\begin{proof}
 If $i<j$, these relations have been derived in Lemma \ref{lem:technic-lemma-no-inductive}. 
 
 If $i\geq j$, note by inspecting Theorem \ref{Thm:MAIN} that the double brackets involving the elements $\{a_\ell,b_\ell\mid 1\leq \ell \leq i\}$ of $\Bc(\Gamma_n)$ are exactly the same as the double brackets in $\Bc(\Gamma_i)$. 
 But since $(a_1,\dots, b_i)$ is the component in $e_2 \Bc(\Gamma_i) e_2$ of the multiplicative moment map of $\Bc(\Gamma_i)$, the claim follows\footnote{To be precise, the claim follows in $\Bc(\Gamma_i)$, where $(a_1,\dots, b_i)$ is invertible. But it suffices to note that our proof of the multiplicative moment map property for $(a_1,\dots, b_n)\in \Bc(\Gamma_n)$ in \ref{sec:proof-Euler-multiplicative-moment-map} does not require the invertibility of $(a_1,\dots, b_n)$, i.e. it holds in $\kk \overline{\Gamma}_n$. Thus, taking $n=i$, we can indeed use \eqref{Phim}.} from \eqref{Phim}. 
\end{proof}

We compute double brackets with the $(2i-1)$-th Euler continuant $(a_1,b_1,\ldots,a_{i-1},b_{i-1},a_i)$. By \eqref{eq:Convenient-formula-Euler-cont-a}, we can use the decomposition 
\begin{equation} \label{eq:Convenient-v2-a}
 (a_1,\ldots,a_i)=\sum_{\ell=2}^i (a_1,\ldots,b_{\ell-1})a_\ell\,+a_1\,.
\end{equation}
(The formula holds for $i=1$ since the sum is over an empty set, hence we omit it).

\begin{lem} \label{Lem:prep-2}
 The double bracket on $\Bc(\Gamma_n)$ is such that
 \begin{equation*}
\begin{aligned}
\lr{(a_1,\dots, a_i),b_j}&=
\begin{cases} 
\frac12\Big(b_j(a_1,\dots, a_i)\otimes e_2+ e_1 \otimes  (a_1,\dots, a_i)b_j\Big)\, , &\mbox{if } i<j,  \\ 
\frac12\Big(b_i(a_1,\dots, a_i)\otimes e_2+ e_1 \otimes  (a_1,\dots, a_i)b_i\Big) & \\
\qquad +\, e_1 \otimes (a_1,\ldots,b_{i-1}) \,, &\mbox{if } i=j,  \\ 
\frac12\Big(b_j(a_1,\dots, a_i)\otimes e_2 - e_1 \otimes  (a_1,\dots, a_i)b_j\Big) \,,  &\mbox{if } i> j ,
\end{cases}  
\end{aligned}
\end{equation*}
 In particular, these identities hold in $\Bloc$.
\end{lem}
\begin{proof}
 If $i< j$, by \eqref{eq:Lem-prep-1.a}, we have 
\begin{small}
\begin{align*}
\lr{(a_1,\dots, a_i),b_j}=&  
\sum_{\ell=2}^i \Big[\lr{(a_1,\ldots,b_{\ell-1}),b_j}\ast a_\ell + (a_1,\ldots,b_{\ell-1})\ast \lr{a_\ell,b_j}\Big]\,+\lr{a_1,b_j} \\
=& \frac12\sum_{\ell=2}^i \Big(b_j(a_1,\dots, b_{\ell-1})\otimes e_2-b_j\otimes (a_1,\dots, b_{\ell-1})\Big) \ast a_\ell \\
&+\frac12\sum_{\ell=2}^i (a_1,\ldots,b_{\ell-1})\ast  \Big(e_1 \otimes a_\ell b_j + b_j a_\ell \otimes e_2\Big) 
+\frac12  (e_1 \otimes a_1 b_j + b_j a_1 \otimes e_2) \\
=&\frac12b_j\left[\sum_{\ell=2}^i (a_1,\ldots,b_{\ell-1})a_\ell\,+a_1\right]\otimes e_2
+\frac12 e_1 \otimes \left[\sum_{\ell=2}^i (a_1,\ldots,b_{\ell-1})a_\ell\,+a_1\right] b_j\,, 
 \end{align*}
\end{small}%
which is the desired result by \eqref{eq:Convenient-v2-a}. 

If $i=j=1$, this is just $\lr{a_1,b_1}$ \eqref{eq:Euler-double-bracket.c}. 
If $i=j$ is distinct from $1$, we use again \eqref{eq:Lem-prep-1.a} to get
\begin{small}
\begin{align*}
&\lr{(a_1,\dots, a_i),b_i}
\\
=&  \sum_{\ell=2}^i \Big[\lr{(a_1,\ldots,b_{\ell-1}),b_i}\ast a_\ell + (a_1,\ldots,b_{\ell-1})\ast \lr{a_\ell,b_i}\Big]\,+\lr{a_1,b_i} \\
=& \frac12\sum_{\ell=2}^i \Big(b_i(a_1,\dots, b_{\ell-1})\otimes e_2-b_i\otimes (a_1,\dots, b_{\ell-1})\Big) \ast a_\ell \\
&+\frac12\sum_{\ell=2}^{i-1} (a_1,\ldots,b_{\ell-1})\ast  \Big(e_1 \otimes a_\ell b_i + b_i a_\ell \otimes e_2\Big) \\
&+\frac12 (a_1,\ldots,b_{i-1})\ast \Big(e_1 \otimes a_i b_i + b_i a_i \otimes e_2 + 2 e_1 \otimes e_2\Big)
+\frac12  (e_1 \otimes a_1 b_j + b_j a_1 \otimes e_2) \\
=&\frac12b_i\left[\sum_{\ell=2}^i (a_1,\ldots,b_{\ell-1})a_\ell\,+a_1\right]\otimes e_2
+\frac12 e_1 \otimes \left[\sum_{\ell=2}^i (a_1,\ldots,b_{\ell-1})a_\ell\,+a_1\right] b_i
+ e_1 \otimes (a_1,\ldots,b_{i-1}), 
 \end{align*}
\end{small}%
which is the claimed result by \eqref{eq:Convenient-v2-a}. 

If $i>j$, we only prove the result for $j\neq 1$, and leave the easier case $j=1$ to the reader. 
We can use \eqref{eq:Lem-prep-1.a} one last time to obtain 
\begin{small}
\begin{align*}
\lr{(a_1,\dots, a_i),b_j}=
&  \sum_{\ell=2}^i \Big[\lr{(a_1,\ldots,b_{\ell-1}),b_j}\ast a_\ell + (a_1,\ldots,b_{\ell-1})\ast \lr{a_\ell,b_j}\Big]\,+\lr{a_1,b_j} \\
=& \frac12\sum_{\ell=2}^j \Big(b_j(a_1,\dots, b_{\ell-1})\otimes e_2-b_j\otimes (a_1,\dots, b_{\ell-1})\Big) \ast a_\ell \\
&+\frac12\sum_{\ell=j+1}^i \Big(b_j(a_1,\dots, b_{\ell-1})\otimes e_2 + b_j\otimes (a_1,\dots, b_{\ell-1})\Big) \ast a_\ell \\
&+\frac12\sum_{\ell=2}^{j-1} (a_1,\ldots,b_{\ell-1})\ast  \Big(e_1 \otimes a_\ell b_j + b_j a_\ell \otimes e_2\Big) \\
&+\frac12 (a_1,\ldots,b_{j-1})\ast \Big(e_1 \otimes a_j b_j + b_j a_j \otimes e_2 + 2 e_1 \otimes e_2\Big) \\
&-\frac12 (a_1,\ldots,b_{j})\ast \Big(e_1 \otimes a_{j+1} b_j + b_j a_{j+1} \otimes e_2 + 2 e_1 \otimes e_2\Big) \\
&-\frac12\sum_{\ell=j+2}^{i} (a_1,\ldots,b_{\ell-1})\ast  \Big(e_1 \otimes a_\ell b_j + b_j a_\ell \otimes e_2\Big)  
+\frac12  (e_1 \otimes a_1 b_j + b_j a_1 \otimes e_2) \\
=&\frac12b_j\left[\sum_{\ell=2}^i (a_1,\ldots,b_{\ell-1})a_\ell\,+a_1\right]\otimes e_2
-\frac12 e_1 \otimes \left[\sum_{\ell=2}^i (a_1,\ldots,b_{\ell-1})a_\ell\,+a_1\right] b_j \\
&+ e_1 \otimes \Big[ (a_1,\ldots,b_{j-1}) + (a_1,\ldots,a_{j})b_j - (a_1,\ldots, b_j) \Big], 
 \end{align*}
\end{small}%
which is the claimed result
 if we use \eqref{eq:Convenient-v2-a} as well as the recursive relation for $(a_1,\ldots, b_j)$, so that the last term disappears. 
\end{proof}

\begin{lem} \label{Lem:prep-3}
 The double bracket on $\Bc(\Gamma_n)$ is such that
 \begin{equation*}
\begin{aligned}
\lr{(a_1,\dots, a_i),a_j}&=
\begin{cases} 
\;\;\,  \frac12\Big(a_j \otimes (a_1,\dots, a_i) - (a_1,\dots, a_i) \otimes a_j\Big)\, , &\mbox{if } i\geq j,  \\ 
-\frac12\Big(a_j \otimes (a_1,\dots, a_i) + (a_1,\dots, a_i) \otimes a_j\Big)\, , &\mbox{if } i=j-1,\, j\neq 1.  \\ 
\end{cases}  
\end{aligned}
\end{equation*}
 In particular, these identities hold in $\Bloc$.
\end{lem}
\begin{proof}
 We first assume that $i>j>1$, and we leave the case $i>j$ with $j=1$ to the reader. Using \eqref{eq:Convenient-v2-a} and \eqref{eq:Lem-prep-1.b}, we can write $\lr{(a_1,\dots, a_i),a_j}$ as 
\begin{small}
\begin{align*}
  \sum_{\ell=2}^i \Big[\lr{(a_1,&\ldots,b_{\ell-1}),a_j}\ast a_\ell + (a_1,\ldots,b_{\ell-1})\ast \lr{a_\ell,a_j}\Big]\,+\lr{a_1,a_j} \\
=& \frac12\sum_{\ell=2}^{j-1} \Big(e_2 \otimes (a_1,\dots, b_{\ell-1}) a_j - (a_1,\dots, b_{\ell-1})\otimes a_j \Big) \ast a_\ell \\
&+\frac12 \Big(e_2 \otimes (a_1,\dots, b_{j-1}) a_j - (a_1,\dots, b_{j-1})\otimes a_j + 2 e_2 \otimes (a_1,\ldots,a_{j-1}) \Big) \ast a_j \\
&-\frac12\sum_{\ell=j+1}^i \Big(e_2 \otimes (a_1,\dots, b_{\ell-1}) a_j + (a_1,\dots, b_{\ell-1})\otimes a_j \Big) \ast a_\ell \\
 &-\frac12\sum_{\ell=2}^{j-1} (a_1,\ldots,b_{\ell-1})\ast  \Big( a_\ell \otimes a_j + a_j \otimes a_\ell\Big) \\
&+\frac12\sum_{\ell=j+1}^{i} (a_1,\ldots,b_{\ell-1})\ast  \Big( a_\ell \otimes a_j + a_j \otimes a_\ell\Big) 
-\frac12  (a_1 \otimes a_j + a_j \otimes a_1) \\
 =&\frac12a_j\otimes \left[\sum_{\ell=2}^i (a_1,\ldots,b_{\ell-1})a_\ell\,+a_1\right]
 -\frac12  \left[\sum_{\ell=2}^i (a_1,\ldots,b_{\ell-1})a_\ell\,+a_1\right] \otimes a_j \\
 &+ a_j \otimes \left[ (a_1,\ldots,a_{j-1}) -\sum_{\ell=2}^{j-1} (a_1,\ldots,b_{\ell-1})a_\ell -a_1 \right] \\
=& \frac12\Big(a_j \otimes (a_1,\dots, a_i) - (a_1,\dots, a_i) \otimes a_j\Big)\,.
 \end{align*}
\end{small}

 Next, let $i=j$. If $j=1$, $\lr{a_1,a_1}=0$ and we are done. If $j\neq 1$, we compute $\lr{(a_1,\dots, a_j),a_j}$ as in the previous case and find 
\begin{small}
\begin{align*}
  \sum_{\ell=2}^j \Big[\lr{(a_1,&\ldots,b_{\ell-1}),a_j}\ast a_\ell + (a_1,\ldots,b_{\ell-1})\ast \lr{a_\ell,a_j}\Big]\,+\lr{a_1,a_j} \\
=& \frac12\sum_{\ell=2}^{j-1} \Big(e_2 \otimes (a_1,\dots, b_{\ell-1}) a_j - (a_1,\dots, b_{\ell-1})\otimes a_j \Big) \ast a_\ell \\
&+\frac12 \Big(e_2 \otimes (a_1,\dots, b_{j-1}) a_j - (a_1,\dots, b_{j-1})\otimes a_j + 2 e_2 \otimes (a_1,\ldots,a_{j-1}) \Big) \ast a_j \\
 &-\frac12\sum_{\ell=2}^{j-1} (a_1,\ldots,b_{\ell-1})\ast  \Big( a_\ell \otimes a_j + a_j \otimes a_\ell\Big) 
-\frac12  (a_1 \otimes a_j + a_j \otimes a_1) \\
 =&\frac12a_j\otimes \left[(a_1,\dots, b_{j-1}) a_j + 2 (a_1,\ldots,a_{j-1})-\sum_{\ell=2}^{j-1} (a_1,\ldots,b_{\ell-1})a_\ell\,-a_1\right] \\
& -\frac12  \left[\sum_{\ell=2}^j (a_1,\ldots,b_{\ell-1})a_\ell\,+a_1\right] \otimes a_j \\
=&\frac12a_j\otimes \Big[(a_1,\dots, b_{j-1}) a_j +  (a_1,\ldots,a_{j-1}) \Big] 
-\frac12  (a_1,\dots, a_j) \otimes a_j \\
=& \frac12\Big(a_j \otimes (a_1,\dots, a_j) - (a_1,\dots, a_j) \otimes a_j\Big)\,,
 \end{align*}
\end{small}%
where we needed to use the recursive relation for $(a_1,\ldots,a_{j})$ in the last line.  
 
 Finally, for $i=j-1$ with $j\neq 1$, we get for $\lr{(a_1,\dots, a_{j-1}),a_j}$ that 
\begin{small}
\begin{align*}
 \sum_{\ell=2}^{j-1} \Big[\lr{(a_1,&\ldots,b_{\ell-1}),a_j}\ast a_\ell + (a_1,\ldots,b_{\ell-1})\ast \lr{a_\ell,a_j}\Big]\,+\lr{a_1,a_j} \\
=& \frac12\sum_{\ell=2}^{j-1} \Big(e_2 \otimes (a_1,\dots, b_{\ell-1}) a_j - (a_1,\dots, b_{\ell-1})\otimes a_j \Big) \ast a_\ell \\
 &-\frac12\sum_{\ell=2}^{j-1} (a_1,\ldots,b_{\ell-1})\ast  \Big( a_\ell \otimes a_j + a_j \otimes a_\ell\Big) 
-\frac12  (a_1 \otimes a_j + a_j \otimes a_1) \\
 =&-\frac12a_j\otimes \Big[ \sum_{\ell=2}^{j-1} (a_1,\ldots,b_{\ell-1})a_\ell\,+ a_1\Big] 
 -\frac12  \left[\sum_{\ell=2}^{j-1} (a_1,\ldots,b_{\ell-1})a_\ell\,+a_1\right] \otimes a_j \\
=& - \frac12\Big(a_j \otimes (a_1,\dots, a_{j-1}) + (a_1,\dots, a_{j-1}) \otimes a_j\Big)\,.\qedhere
 \end{align*}
\end{small}%
\end{proof}

We now perform computations in $\Bloc$ involving the elements $\{a_j'\mid1\leq j\leq n\}$ for which we use the definition in \eqref{Eq:Locgen2}. 
\begin{lem} \label{Lem:prep-4}
 The double bracket on $\Bloc$ is such that
 \begin{equation*}
\begin{aligned}
\lr{a_i',a_j}&=
\begin{cases} 
\frac12\Big(a_i' \otimes a_j + a_j \otimes a_i' \Big)\, , &\mbox{if } i>  j,  \\ 
\frac12\Big(a_j' \otimes a_j + a_j \otimes a_j' \Big) - a_j' \otimes a_j' \,, &\mbox{if } i=j.  \\ 
\end{cases}  
\end{aligned}
\end{equation*}
\end{lem}
\begin{proof}
 We will repeatedly use \eqref{eq:Euler-double-bracket.a}, \eqref{eq:Lem-prep-1.b} and Lemma \ref{Lem:prep-3}. If $i>j>1$, we have that $\lr{a_i',a_j}$ equals  
 \begin{small}
\begin{align*}
& \dgal{\Big(a_i+(a_1,\ldots,b_{i-1})^{-1}(a_1,\ldots,a_{i-1})\Big),a_j} \\
=& \frac12 (a_i \otimes a_j + a_j \otimes a_i) 
+\frac12 (a_1,\ldots,b_{i-1})^{-1}\ast \Big(a_j \otimes (a_1,\ldots,a_{i-1}) - (a_1,\ldots,a_{i-1}) \otimes a_j  \Big) \\
&-\!\frac12 (a_1,\ldots,b_{i-1})^{-1}\!\!\ast \!\big(\!- e_2\! \otimes \!(a_1,\ldots,b_{i-1})a_j \!- (a_1,\ldots,b_{i-1}) \!\otimes\! a_j  \big) \!\ast\! (a_1,\ldots,b_{i-1})^{-1}(a_1,\ldots,a_{i-1}) \\
=&\frac12 \Big[a_i+(a_1,\ldots,b_{i-1})^{-1}(a_1,\ldots,a_{i-1})\Big]\otimes a_j + \frac12 a_j \otimes \Big[a_i+(a_1,\ldots,b_{i-1})^{-1}(a_1,\ldots,a_{i-1})\Big]\\
=& \frac12\Big(a_i' \otimes a_j + a_j \otimes a_i' \Big) \,.
 \end{align*}
\end{small}%
If $i=j=1$, we have $\lr{a_1',a_1}=\lr{a_1,a_1}=0$. If $i=j$ is distinct from $1$,  we have that $\lr{a_j',a_j}$ can be computed as follows 
 \begin{small}
\begin{align*}
 &\dgal{\Big(a_j+(a_1,\ldots,b_{j-1})^{-1}(a_1,\ldots,a_{j-1})\Big),a_j} \\
=& 
-\frac12 (a_1,\ldots,b_{j-1})^{-1}\ast \Big(a_j \otimes (a_1,\ldots,a_{j-1}) + (a_1,\ldots,a_{j-1}) \otimes a_j  \Big) \\
&-\frac12 (a_1,\ldots,b_{j-1})^{-1}\ast \Big( e_2 \otimes (a_1,\ldots,b_{j-1})a_j - (a_1,\ldots,b_{j-1}) \otimes a_j \\
&\qquad \qquad \qquad \qquad \qquad \quad + 2 e_2 \otimes (a_1,\ldots,a_{j-1})\Big) \ast (a_1,\ldots,b_{j-1})^{-1}(a_1,\ldots,a_{j-1}) \\
=&-\frac12 (a_1,\ldots,b_{j-1})^{-1}(a_1,\ldots,a_{j-1})\otimes a_j - \frac12 a_j \otimes (a_1,\ldots,b_{j-1})^{-1}(a_1,\ldots,a_{j-1})\\
&-(a_1,\ldots,b_{j-1})^{-1}(a_1,\ldots,a_{j-1}) \otimes (a_1,\ldots,b_{j-1})^{-1}(a_1,\ldots,a_{j-1}) \\
=& -\frac12\Big((a_j' - a_j) \otimes a_j + a_j \otimes (a_j'-a_j) \Big) - (a_j'-a_j)\otimes (a_j'-a_j) \\
=&\frac12\Big(a_j' \otimes a_j + a_j \otimes a_j' \Big) - a_j' \otimes a_j' \,.\qedhere
 \end{align*}
\end{small}%
\end{proof}

\subsubsection{Proof of Proposition \ref{Pr:Alt-qHam}} 

We show that we can write $\Phi$ as \eqref{Eq:Alt-momap}, then we show that the double bracket takes the form \eqref{eq:Alt-double-bracket} on the generators $\{ a_\ell',b_\ell'\}$.

\textbf{Step 1: Rewriting $\Phi$.}  
For $1\leq \ell \leq n$, we prove that, 
\begin{equation} \label{Eq:contInd}
\begin{aligned}
 (a_1,\ldots,b_\ell)&=(e_2+a_1'b_1')\cdots(e_2+a_\ell' b_\ell')\,, 
 \\
 (b_\ell,\ldots,a_1)&= (e_1+b_\ell'a_\ell')\cdots (e_1+b_1' a_1')\,.
 \end{aligned}
\end{equation}
This gives in particular that $\Phi$ takes the form \eqref{Eq:Alt-momap} when $\ell=n$. We prove \eqref{Eq:contInd} by induction. The case $\ell=1$ is obvious. If we assume that \eqref{Eq:contInd} holds for $\ell-1$, we have 
\begin{equation*}
 \begin{aligned}
  (a_1,\ldots,b_\ell)=&(a_1,\ldots,b_{\ell-1})(e_2+a_\ell b_\ell) + (a_1,\ldots,a_{\ell-1})b_\ell\\
  =&(a_1,\ldots,b_{\ell-1}) (e_2+\underbrace{[a_\ell+(a_1,\ldots,b_{\ell-1})^{-1}(a_1,\ldots,a_{\ell-1})]}_{=\,a_\ell'\text{ by }\eqref{Eq:Locgen2}}b_\ell)\\
  =&(e_2+a_1'b_1')\cdots(e_2+a_{\ell-1}' b_{\ell-1}')(e_2+a_\ell' b_\ell')\,, \\
 (b_\ell,\ldots,a_1)=&(e_1+b_\ell a_\ell) (b_{\ell-1},\ldots,a_1) + (b_{\ell-1},\ldots,a_1)\\
 =&(e_1+b_\ell \underbrace{[a_\ell+(a_{\ell-1},\ldots,a_1)(b_{\ell-1},\ldots,a_1)^{-1}]}_{=\,a_\ell'\text{ by }\eqref{Eq:Locgen3}})    (b_{\ell-1},\ldots,a_1) \\
 =&(e_1+b_\ell'a_\ell')(e_1+b_{\ell-1}'a_{\ell-1}')\cdots (e_1+b_1'a_1')\,.
 \end{aligned}
\end{equation*}

\textbf{Step 2: The double bracket $\lr{b_i',b_j'}$.}   It is clear that \eqref{eq:Alt-double-bracket.b} holds since it is just \eqref{eq:Euler-double-bracket.b} because by definition $b'_\ell=b_\ell$. 

\textbf{Step 3: The double bracket $\lr{a_i',b_j'}$.} We have to check \eqref{eq:Alt-double-bracket.c}. It directly holds for $i=1$ by \eqref{eq:Euler-double-bracket.c} since $\lr{a_1',b_j'}=\lr{a_1,b_j}$. So, hereafter, we assume $i\neq 1$. Then,
we need to compute 
\begin{equation*}
\begin{aligned}
  \lr{a_i',b_j'}=&\dgal{\Big(a_i+(a_1,\ldots,b_{i-1})^{-1}(a_1,\ldots,a_{i-1})\Big), b_j} \\
  =& \lr{a_i,b_j} + (a_1,\ldots,b_{i-1})^{-1}\ast \lr{(a_1,\ldots,a_{i-1}) ,b_j} \\
  &- (a_1,\ldots,b_{i-1})^{-1}\ast \lr{(a_1,\ldots,b_{i-1}),b_j} \ast (a_1,\ldots,b_{i-1})^{-1} (a_1,\ldots,a_{i-1}) \,,
\end{aligned}
\end{equation*}
and this can be obtained from \eqref{eq:Euler-double-bracket.c} and Lemmae \ref{Lem:prep-1} and \ref{Lem:prep-2}. If $i<j$, we find 
\begin{small}
\begin{align*}
\lr{a_i',b_j'}=&\frac12 (e_1 \otimes a_ib_j+b_j a_i \otimes e_2)  \\ 
&+ \frac12  \Big(b_j (a_1,\ldots,a_{i-1}) \otimes (a_1,\ldots,b_{i-1})^{-1} + e_1 \otimes (a_1,\ldots,b_{i-1})^{-1}(a_1,\ldots,a_{i-1}) b_j \Big) \\
&-\frac12  \Big(b_j(a_1,\ldots,a_{i-1}) \otimes (a_1,\ldots,b_{i-1})^{-1} - b_j (a_1,\ldots,b_{i-1})^{-1}(a_1,\ldots,a_{i-1}) \otimes e_2 \Big) \\
=&\frac12 e_1 \otimes \Big[a_i+(a_1,\ldots,b_{i-1})^{-1}(a_1,\ldots,a_{i-1})\Big] b_j \\
&+\frac12\, b_j \Big[a_i+(a_1,\ldots,b_{i-1})^{-1}(a_1,\ldots,a_{i-1}) \Big] \otimes e_2\\
=& \frac12 (e_1 \otimes a_i' b_j' + b_j' a_i'\otimes e_2)\,,
\end{align*}
\end{small}%
as expected. The proof is basically the same for $i=j$ as we add a term $+e_1 \otimes e_2$ coming from $\lr{a_i,b_i}$. 

If $i=j+1$, we get
\begin{small}
\begin{align*}
\lr{a_{j+1}',b_j'}
=&-\frac12 (e_1 \otimes a_{j+1}b_j+b_j a_{j+1} \otimes e_2) - e_1 \otimes e_2  \\ 
&+ \frac12  \Big(b_j (a_1,\ldots,a_{j}) \otimes (a_1,\ldots,b_{j})^{-1} + e_1 \otimes (a_1,\ldots,b_{j})^{-1}(a_1,\ldots,a_{j}) b_j\Big) \\ 
&+ e_1 \otimes (a_1,\ldots,b_j)^{-1} (a_1,\ldots,b_{j-1})\\
&-\frac12  \Big(b_j(a_1,\ldots,a_{j}) \otimes (a_1,\ldots,b_{j})^{-1} + b_j (a_1,\ldots,b_{j})^{-1}(a_1,\ldots,a_{j}) \otimes e_2 \Big) \\
=&-\frac12 e_1 \!\otimes \!\Big[a_{j+1}+(a_1,\ldots,b_{j})^{-1}(a_1,\ldots,a_{j})\Big] b_j  \\
&-\frac12\, b_j \Big[a_{j+1}+(a_1,\ldots,b_{j})^{-1}(a_1,\ldots,a_{j}) \Big]\! \otimes e_2\\
&+ e_1 \otimes (a_1,\ldots,b_{j})^{-1}\Big[ (a_1,\ldots,a_{j}) b_j + (a_1,\ldots,b_{j-1}) - (a_1,\ldots,b_{j})\Big] \\
=&- \frac12 (e_1 \otimes a_i' b_j' + b_j' a_i'\otimes e_2)\,,
\end{align*}
\end{small}%
where we used the recursive relation for $(a_1,\ldots, b_j)$. 

If $i>j+1$, we finally obtain 
\begin{small}
\begin{align*}
\lr{a_i',b_j'}
=&-\frac12 (e_1 \otimes a_ib_j+b_j a_i \otimes e_2)  \\ 
&+ \frac12  \Big(b_j (a_1,\ldots,a_{i-1}) \otimes (a_1,\ldots,b_{i-1})^{-1} - e_1 \otimes (a_1,\ldots,b_{i-1})^{-1}(a_1,\ldots,a_{i-1}) b_j \Big) \\
&-\frac12  \Big(b_j(a_1,\ldots,a_{i-1}) \otimes (a_1,\ldots,b_{i-1})^{-1} + b_j (a_1,\ldots,b_{i-1})^{-1}(a_1,\ldots,a_{i-1}) \otimes e_2 \Big) \\
=&-\frac12 e_1 \!\otimes \Big[a_i+(a_1,\ldots,b_{i-1})^{-1}(a_1,\ldots,a_{i-1})\Big] b_j\\
&-\frac12\, b_j \Big[a_i+(a_1,\ldots,b_{i-1})^{-1}(a_1,\ldots,a_{i-1}) \Big]\! \otimes e_2\\
=&- \frac12 (e_1 \otimes a_i' b_j' + b_j' a_i'\otimes e_2)\,,
\end{align*}
\end{small}%
as we wanted.

\textbf{Step 4: Interlude (some double brackets involving $a_i'$).} 
For the sake of clarity, let us rewrite the following two identities which were obtained in Lemma \ref{Lem:prep-4} and in \eqref{eq:Alt-double-bracket.b} with $b_j=b_j'$ (the latter has just been checked): 
\begin{subequations} \label{Eq:a-prim}
 \begin{align}
   \lr{a_i',b_j}=&-\frac12 \big(e_1 \otimes a_i' b_j + b_j a_i' \otimes e_2\big)\,, &  &\text{ for } i>j\,, \label{Eq:a-prim-1}\\
    \lr{a_i',a_j}=&\frac12 \big(a_i' \otimes a_j + a_j \otimes a_i'\big)\,, & &\text{ for } i>j\,. \label{Eq:a-prim-2}
 \end{align}
\end{subequations}

\begin{lem} \label{Lem:a-prim-1}
 The double bracket on $\Bloc$ is such that 
 \begin{equation}
  \lr{a_i',(a_j,b_j)}= \frac12 \Big( a_i' \otimes (a_j,b_j) - (a_j,b_j) a_i' \otimes e_2 \Big)\,, \quad \text{ for } i>j\,.
 \end{equation}
\end{lem}
\begin{proof}
 Using \eqref{Eq:a-prim}, we get for $i>j$ that 
\begin{small}
\begin{align*}
\lr{a_i',(a_j,b_j)}=& \lr{a_i',a_jb_j} \\
=& \frac12 (a_i' \otimes a_j + a_j \otimes a_i')b_j -\frac12 a_j(e_1 \otimes a_i' b_j + b_j a_i' \otimes e_2) \\
=& \frac12  a_i' \otimes (e_2+a_j b_j) -\frac12 (e_2+a_j b_j) a_i' \otimes e_2\,.\qedhere
\end{align*}
\end{small}
\end{proof}

\begin{lem} \label{Lem:a-prim-2}
 The double bracket on $\Bloc$ is such that 
 \begin{equation}
  \lr{a_i',(a_1,\ldots,b_j)}= \frac12 \Big( a_i' \otimes (a_1,\ldots,b_j) - (a_1,\ldots,b_j) a_i' \otimes e_2 \Big)\,, \quad \text{ for } i>j\,.
 \end{equation}
\end{lem}
\begin{proof}
 We prove the result by induction on $j$. Note that the case $j=1$ is directly obtained from Lemma \ref{Lem:a-prim-1}. If $i>j>1$, by \eqref{eq:Convenient-formula-Euler-cont-b},
 \begin{small}
\begin{align*}
\lr{a_i',(a_1,\ldots,b_j)}=\lr{a_i',(a_1,\ldots,b_{j-1})(a_j,b_j)} + \sum_{\ell=2}^{j-1}\lr{a_i',(a_1,\ldots,b_{\ell-1})a_\ell b_j } + \lr{a_i',a_1 b_j}
\end{align*}
\end{small}%
which can be computed using the induction hypothesis together with Lemma \ref{Lem:a-prim-1} and \eqref{Eq:a-prim} as 
 \begin{small}
\begin{align*}
=&\frac12 \Big( a_i' \otimes (a_1,\ldots,b_{j-1}) - (a_1,\ldots,b_{j-1}) a_i' \otimes e_2 \Big) (a_j,b_j) \\
&+\frac12 (a_1,\ldots,b_{j-1})\Big( a_i' \otimes (a_j,b_j) - (a_j,b_j) a_i' \otimes e_2 \Big) \\
&+\frac12 \sum_{\ell=2}^{j-1} \Big( a_i' \otimes (a_1,\ldots,b_{\ell-1}) - (a_1,\ldots,b_{\ell-1}) a_i' \otimes e_2 \Big)    a_\ell b_j  \\
&+\frac12 \sum_{\ell=2}^{j-1} (a_1,\ldots,b_{\ell-1})  \Big(a_i' \otimes a_\ell + a_\ell \otimes a_i'\Big)   b_j
- \frac12 \sum_{\ell=2}^{j-1} (a_1,\ldots,b_{\ell-1}) a_\ell \Big(e_1 \otimes a_i' b_j + b_j a_i' \otimes e_2\Big) \\
&+ \frac12 (a_i' \otimes a_1 + a_1 \otimes a_i')   b_j - \frac12 a_1 (e_1 \otimes a_i' b_j + b_j a_i' \otimes e_2) \\
=&\frac12 a_i' \otimes \Big[(a_1,\ldots,b_{j-1})(a_j,b_j)  + \sum_{\ell=2}^{j-1} (a_1,\ldots,b_{\ell-1})a_\ell b_j + a_1 b_j  \Big] \\
&-\frac12 \Big[ (a_1,\ldots,b_{j-1})(a_j,b_j)  + \sum_{\ell=2}^{j-1} (a_1,\ldots,b_{\ell-1})a_\ell b_j + a_1 b_j \Big] a_i' \otimes e_2\,,
\end{align*}
\end{small}%
and this is the claimed result due to \eqref{eq:Convenient-formula-Euler-cont-b}. 
\end{proof}

\begin{lem} \label{Lem:a-prim-3}
 The double bracket on $\Bloc$ is such that 
 \begin{equation}
  \lr{a_i',(a_1,\ldots,a_j)}= \frac12 \Big( a_i' \otimes (a_1,\ldots,a_j) + (a_1,\ldots,a_j) \otimes a_i' \Big)\,, \quad \text{ for } i>j\,.
 \end{equation}
\end{lem}
\begin{proof}
We prove the result by induction on $j$, and the case $j=1$ is a special case of \eqref{Eq:a-prim-2}. 
We can use \eqref{eq:Convenient-v2-a} to write 
 \begin{small}
\begin{align*}
 \lr{a_i',(a_1,\ldots,a_j)}=\sum_{\ell=2}^{j}\lr{a_i',(a_1,\ldots,b_{\ell-1})a_\ell} + \lr{a_i',a_1}\,,
\end{align*}
\end{small}%
and due to \eqref{Eq:a-prim-2} and Lemma \ref{Lem:a-prim-2} we find 
 \begin{small}
\begin{align*}
=&\frac12 \sum_{\ell=2}^{j}  \Big( a_i' \otimes (a_1,\ldots,b_{\ell-1}) - (a_1,\ldots,b_{\ell-1}) a_i' \otimes e_2 \Big) a_\ell \\
&+\frac12\sum_{\ell=2}^{j}   (a_1,\ldots,b_{\ell-1}) \Big( a_i' \otimes a_\ell + a_\ell \otimes a_i' \Big) 
+ \Big( a_i' \otimes a_1 + a_1 \otimes a_i' \Big) \\
=&\frac12 a_i' \otimes \Big[\sum_{\ell=2}^{j}(a_1,\ldots,b_{\ell-1}) a_\ell + a_1 \Big]
+ \frac12 \Big[\sum_{\ell=2}^{j}(a_1,\ldots,b_{\ell-1}) a_\ell + a_1 \Big] \otimes a_i'\,,
\end{align*}
\end{small}%
which is the desired result by \eqref{eq:Convenient-v2-a}. 
\end{proof}

\textbf{Step 5: The double bracket $\lr{a_i',a_j'}$.}
It remains to prove that $\lr{a_i',a_j'}$ takes the form \eqref{eq:Alt-double-bracket.a}, and it suffices to do so for $i\geq j$ due to the cyclic antisymmetry. 

Firstly, we consider the case when $i>j$. We write $a_j'$ using \eqref{Eq:Locgen2}. The case $j=1$ is just \eqref{Eq:a-prim-2}, while for $i>j>1$ we have that 
 \begin{small}
\begin{align*}
\lr{a_i',a_j'}=\dgal{a_i',\Big(a_j+(a_1,\ldots,b_{j-1})^{-1}(a_1,\ldots,a_{j-1})\Big)}\,.
\end{align*}
\end{small}%
This can be computed using \eqref{Eq:a-prim-2}, and Lemmae \ref{Lem:a-prim-2} and \ref{Lem:a-prim-3}, as 
 \begin{small}
\begin{align*}
=&\frac12 (a_i'\otimes a_j + a_j \otimes a_i')
+\frac12(a_1,\ldots,b_{j-1})^{-1} \Big( a_i' \otimes (a_1,\ldots,a_{j-1}) + (a_1,\ldots,a_{j-1}) \otimes a_i' \Big) \\
 &-\frac12(a_1,\ldots,b_{j-1})^{-1} \Big( a_i' \otimes (a_1,\ldots,b_{j-1}) - (a_1,\ldots,b_{j-1}) a_i' \otimes e_2 \Big) (a_1,\ldots,b_{j-1})^{-1}(a_1,\ldots,a_{j-1}) \\
=&\frac12 a_i' \otimes \Big[ a_j+(a_1,\ldots,b_{j-1})^{-1}(a_1,\ldots,a_{j-1}) \Big]
+\frac12 \Big[a_j+(a_1,\ldots,b_{j-1})^{-1}(a_1,\ldots,a_{j-1}) \Big]\otimes a_i' \\
=&\frac12( a_i' \otimes a_j' +  a_j' \otimes a_i')\,.
\end{align*}
\end{small}

When $i=j$, this is obviously zero if $j=1$ by \eqref{eq:Euler-double-bracket.a}. If $j> 1$ we use Lemma \ref{Lem:prep-4} to get $\lr{a_j',a_j}$ so that as in the previous case we can compute
 \begin{small}
\begin{align*}
\lr{a&_j',a_j'}
\\
=&\frac12 (a_j'\otimes a_j + a_j \otimes a_j') - a_j' \otimes a_j'\\
&+\frac12(a_1,\ldots,b_{j-1})^{-1} \Big( a_j' \otimes (a_1,\ldots,a_{j-1}) + (a_1,\ldots,a_{j-1}) \otimes a_j' \Big) \\
 &-\frac12(a_1,\ldots,b_{j-1})^{-1} \Big( a_j' \otimes (a_1,\ldots,b_{j-1}) - (a_1,\ldots,b_{j-1}) a_j' \otimes e_2 \Big) (a_1,\ldots,b_{j-1})^{-1}(a_1,\ldots,a_{j-1}) \\
=&\frac12 a_j' \otimes \Big[a_j+(a_1,\ldots,b_{j-1})^{-1}(a_1,\ldots,a_{j-1}) \Big]  \\
&+\frac12 \Big[a_j+(a_1,\ldots,b_{j-1})^{-1}(a_1,\ldots,a_{j-1}) \Big] \otimes a_j' - a_j' \otimes a_j' \,,
\end{align*}
\end{small}%
which is zero, as expected. So, $\lr{a'_j,a'_j}$ takes the form \eqref{eq:Alt-double-bracket.a}, and Proposition \ref{eq:Alt-double-bracket} follows.

 \subsection{Towards the quasi-bisymplectic form} \label{ss:qbisymp}

Working in full generalities as in \ref{sec:sec-double-quasi-Ham}, we consider a finitely generated unital algebra $A$ over $B=\oplus_{s\in I} \kk e_s$ and the $A$-bimodule $\Omega^1_B A$ of ($B$-relative) noncommutative differential $1$-forms. Following \cite{CQ95}, we put $\Omega_B A:=T_A \Omega^1_B A$ and note that $\du\colon A\to\Omega^1_B A$ can be extended to a differential on $\Omega_B A$. If $A$ is freely generated by elements $\{a_j\mid j\in J\}$, with $|J|<\infty$ and $a_j=e_{h_j} a_j e_{t_j}$, we can introduce the map 
\begin{equation}
 \tilde{\du}:=\sum_{j\in J}\du\! a_j\frac{\partial}{\partial a_j}:A\longrightarrow  A\otimes \Omega_B^1A\,, \quad 
 \tilde{\du}(b)= \sum_{j\in J}\left(\frac{\partial b}{\partial a_j}\right)'\otimes \du\! a_j\left(\frac{\partial b}{\partial a_j}\right)''\,,
\end{equation}
for the double derivations $\partial/\partial a_j\in \D_BA$ given by $\partial a_i/\partial a_j=\delta_{ij} \, e_{h_j}\otimes e_{t_j}$. 
Since we have that $A\otimes \Omega_B^1A\subset  \Omega_BA\otimes \Omega_BA$, we can decompose the differential as 
\begin{equation*}
 \du:A\stackrel{\tilde{\du}}{\longrightarrow} \Omega_BA\otimes \Omega_BA \stackrel{\tilde{\mc}}{\longrightarrow} \Omega_BA\,,
\end{equation*}
through the multiplication map $\tilde{\mc}$ on $\Omega_BA$. By iterating $\tilde{\du}$, we can introduce 
\begin{equation}
 \begin{aligned} \label{Eq:dtwo}
\Dtwo:=&\tilde{\mc}\circ (\du\otimes 1)\circ \tilde{\du}:A\longrightarrow  \Omega^2_BA \subset \Omega_BA\,,\\
\Dtwo(b)=&\sum_{i,j\in J} \left(\frac{\partial}{\partial a_i}\left(\frac{\partial b}{\partial a_j}\right)'\right)' \du\! a_i 
\left(\frac{\partial}{\partial a_i}\left(\frac{\partial b}{\partial a_j}\right)'\right)'' \du\! a_j\left(\frac{\partial b}{\partial a_j}\right)''\,.
 \end{aligned}
\end{equation}

The above construction can be applied to the algebra $\kk\overline{\Gamma}_n$, and after localisation to $\Bc(\Gamma_n)$. Using the double derivations given in \eqref{Eq:dder-ab}, we have 
\begin{equation*}
 \tilde{\du}=\sum_{j=1}^n \left( \du\! a_j\frac{\partial}{\partial a_j}+ \du\! b_j\frac{\partial}{\partial b_j}\right)\,.
\end{equation*}
In particular, we can define the map $\Dtwo:\Bc(\Gamma_n)\to \Omega_B^2(\Bc(\Gamma_n))$ by \eqref{Eq:dtwo}. 
Motivated by the construction of Paluba \cite[\S5.6]{Paluba} (see below), we introduce the element  
\begin{equation} \label{Eq:biform}
 \omega_n=\frac12 (b_n,\ldots,a_1)^{-1} \Dtwo(b_n,\ldots,a_1)
 - \frac12 (a_1,\ldots,b_n)^{-1} \Dtwo (a_1,\ldots,b_n)\,.
\end{equation}
Note that we use the two Euler continuants that are inverted in $\Bc(\Gamma_n)$ and which occur in the moment map $\Phi$, as defined in \eqref{eq:mult-moment-map-statement}. 

\begin{conj} \label{Conj:qbisymp}
 The triple $(\Bc(\Gamma_n),\Phi,\omega_n)$ is a quasi-bisymplectic algebra. Furthermore, the double quasi-Poisson bracket $\lr{-,-}$ given by \eqref{eq:Euler-double-bracket} is non-degenerate, and it is compatible with $\omega_n$. 
\end{conj}
Since we are not going to prove this result here\footnote{In the case $n=1$, Conjecture \ref{Conj:qbisymp} holds. 
Non-degeneracy of the double quasi-Poisson bracket was proved by Van den Bergh \cite[\S8.3]{VdB2}, while the other parts of the conjecture were recently shown by Bozec, Calaque and Scherotzke using the theory of relative Calabi-Yau structures, see  \cite[\S5]{BCS} with $e=a_1$, $e^\ast=b_1$.}, the precise definitions (of quasi-bisymplectic algebra, non-degeneracy, or compatibility) are omitted and the reader can find them in \cite{VdB2}. 
Let us nevertheless remark that by \cite[Proposition 6.1]{VdB2}, if the conjecture holds, we will have that $\omega_n$ turns any representation space  $\Rep(\overline{\Gamma}_n,(d_1,d_2))$ (with its natural $\Gl_{d_1}(\kk)\times \Gl_{d_2}(\kk)$-action)  into a quasi-Hamiltonian space in the sense of \cite{AMM98}. Moreover, the induced $2$-form will be compatible with the quasi-Poisson bracket on $\Rep(\overline{\Gamma}_n,(d_1,d_2))$ induced by the \emph{double} quasi-Poisson bracket from Theorem \ref{Thm:MAIN}, which will be non-degenerate by \cite[\S6--7]{VdB2}. We should emphasise that, in the previous sentence, compatibility and non-degeneracy are meant in the geometric sense of \cite{AKSM02}, not at the noncommutative/associative algebra level stated in the conjecture. 
To strengthen Conjecture \ref{Conj:qbisymp}, we explain how the induced results on $\Rep(\overline{\Gamma}_n,(d_1,d_2))$ hold in the \emph{particular case} of $\kk=\CC$ thanks to the factorisation presented in this section. 

Fix $\kk=\CC$ and $d_1,d_2\geq 1$. Let $\PP_{\Rep}$ and $\omega_{\Rep}$ be the (matrix) bivectors and $2$-forms induced on $\Rep(\overline{\Gamma}_n,(d_1,d_2))$ by $\PPn$ and $\omega_n$, respectively. Their explicit expressions can be obtained using the formalism of \cite[\S7]{VdB1}.

\begin{prop} \label{Pr:RepCorrespond}
 The $2$-form $\tr(\omega_{\Rep})$ is the $2$-form from the quasi-Hamiltonian structure on $\Rep(\overline{\Gamma}_n,(d_1,d_2))$ constructed by Boalch in \cite{Bo14} (through the identification with the fission variety $\mathcal{B}^{n+1}(\CC^{d_1},\CC^{d_2})$). Furthermore, $\tr(\omega_{\Rep})$ corresponds to the quasi-Poisson bivector $\tr(\PP_{\Rep})$. 
\end{prop}
\begin{proof}
 Let $A_i\in \Mat_{d_2\times d_1}(\CC)$, $B_i\in\Mat_{d_1\times d_2}(\CC)$ be the matrix-valued functions on $\Rep(\overline{\Gamma}_n,(d_1,d_2))$ returning the elements representing $a_i,b_i$ for any $1\leq i \leq n$. Then, we can write from \eqref{Eq:biform} that 
 \begin{equation} \label{Eq:biform-B}
\begin{aligned}
 \tr(\omega_{\Rep})=&\frac12 \tr\Big( (B_n,\ldots,A_1)^{-1} \Dtwo(B_n,\ldots,A_1) \Big) \\
 &- \frac12 \tr\Big( (A_1,\ldots,B_n)^{-1} \Dtwo (A_1,\ldots,B_n) \Big)\,, 
\end{aligned}
\end{equation}
where the matrix $2$-forms occurring on the right-hand side are given by 
$$\Dtwo C_1 \cdots C_m=\sum_{k<k'} C_1 \cdots C_{k-1} (dC_k) C_{k+1} \cdots C_{k'-1} (d C_{k'})C_{k'+1} \cdots C_m\,,$$
with matrix $1$-forms $(dC_j)$ defined in the obvious way. Upon relabelling the matrices as 
\begin{equation} \label{Eq:bifact}
 \mathrm{b}_\ell=\left\{ 
\begin{array}{cc}
A_{(2n-\ell+2)/2} & \text{ for }\ell \text{ even}, \\
B_{(2n-\ell+1)/2}& \text{ for } \ell \text{ odd},
\end{array}
\right. \qquad \ell=1,\ldots,2n\,, 
\end{equation}
(these are \emph{not} the elements $b_j\in \Bc(\Gamma_n)$), we have that 
 \begin{equation} \label{Eq:biform-C}
\begin{aligned}
 \tr(\omega_{\Rep})=&\frac12 \tr\Big( (\mathrm{b}_1,\ldots,\mathrm{b}_{2n})^{-1} \Dtwo(\mathrm{b}_1,\ldots,\mathrm{b}_{2n}) \Big) \\
 &- \frac12 \tr\Big( (\mathrm{b}_{2n},\ldots,\mathrm{b}_1)^{-1} \Dtwo (\mathrm{b}_{2n},\ldots,\mathrm{b}_1) \Big)\,. 
\end{aligned}
\end{equation}
Using the $2n$ matrices $(\mathrm{b}_\ell)$, the space $\Rep(\overline{\Gamma}_n,(d_1,d_2))$ can be directly identified with the higher fission space $\mathcal{B}^{n+1}(\CC^{d_1},\CC^{d_2})$  using the parametrisation given by Paluba in \cite[Chapter 5]{Paluba}. 
Furthermore, the formula \eqref{Eq:biform-C} for $\tr(\omega_{\Rep})$ exactly matches\footnote{There is an overall sign that changes. This is a consequence of the factorisation \eqref{Eq:bifact} which sends our moment map to the inverse of the one in \cite{Paluba}.} the quasi-Hamiltonian $2$-form on $\mathcal{B}^{n+1}(\CC^{d_1},\CC^{d_2})$ from \cite{Bo14} in the form written by Paluba in \cite[Theorem 5.6.1]{Paluba}. (In the case $n=1$, the $2$-form on $\mathcal{B}^{2}(\CC^{d_1},\CC^{d_2})$ is already written in that way in \cite[Theorem 2.9]{Bo14}).

For $n=1$, $\tr(\PP_{\Rep})$ is Van den Bergh's quasi-Poisson bivector \cite{VdB1}, which corresponds to $\tr(\omega_{\Rep})$ \cite{VdB2,Ya}, establishing the second part of the statement in that case. If $n\geq 2$, note that the factorisation  from Theorem \ref{Thm:Factorisation} yields a dense embedding of Hamiltonian quasi-Poisson manifolds 
\begin{equation}
\underset{1\leq i \leq n}{\times} \Rep(\overline{\Gamma}_1,(d_1,d_2)) 
 \simeq \Rep(\Afus,(d_1,d_2)) 
 \hookrightarrow \Rep(\overline{\Gamma}_n,(d_1,d_2)) \,,
\end{equation}
where the Hamiltonian quasi-Poisson structure on the left is obtained by fusion. Denote by $M^{\textrm{inv}} \subset \Rep(\overline{\Gamma}_n,(d_1,d_2))$ this dense subspace. 
Thus, when restricted to $M^{\textrm{inv}}$, $\tr(\PP_{\Rep})$ is obtained by fusion from $n$ copies of the $n=1$ case. 
We get from \cite[Theorem 5.6.1]{Paluba} that  $\tr(\omega_{\Rep})$ admits the same factorisation\footnote{The fact that the factorisation \eqref{Eq:Locgen2} is equivalent (at the level of representation spaces) to performing inductively the right factorisation map $f_R:\mathcal{B}^{\ell+1}\times \mathcal{B}^2\to \mathcal{B}^{\ell+2}$ of \cite[\S5.5]{Paluba} for $\ell=1,\ldots,n-1$ can be shown after some minor computations using the parametrisation \eqref{Eq:bifact}.} on $M^{\textrm{inv}}$ in terms of $n$ copies of the $n=1$ case. 
As we recalled, the correspondence of the bivector and the $2$-form is already known on $\Rep(\overline{\Gamma}_1,(d_1,d_2))$, so that 
$\tr(\PP_{\Rep})$ and $\tr(\omega_{\Rep})$ correspond to one another on $M^{\textrm{inv}}$. 
Being in correspondence (see \cite[Lemma 10.2]{AKSM02}) means that the composite 
\begin{equation*}
 \tr(\PP_{\Rep})^\sharp \circ \tr(\omega_{\Rep})^\flat: T\Rep(\overline{\Gamma}_n,(d_1,d_2))\to T\Rep(\overline{\Gamma}_n,(d_1,d_2))\,,
\end{equation*}
coincides with a specific holomorphic function when restricted to $TM^{\textrm{inv}}$. 
By holomorphicity, these functions must then coincide globally, so $\tr(\PP_{\Rep})$ and $\tr(\omega_{\Rep})$ correspond to one another globally.   
\end{proof}

  \end{document}